\documentclass[a4paper,10 pt, reqno]{amsart}
\usepackage{amsmath, amssymb, amscd,amsbsy,amsfonts,bbm}
\usepackage[dvips]{graphicx} %
\usepackage{enumerate, verbatim}
\usepackage{color}
\usepackage{xspace}
\usepackage[applemac]{inputenc}
\usepackage[T1]{fontenc}
\usepackage{lmodern}
\usepackage{textcomp}
\usepackage[all, cmtip]{xy} 
\usepackage{url}
\theoremstyle{plain}
\newtheorem{theorem}{Theorem}[section]
\newtheorem{prp}[theorem]{Proposition}
\newtheorem{lem}[theorem]{Lemma}
\newtheorem{crl}[theorem]{Corollary}

\newtheorem{criterion}[theorem]{Criterion}
\theoremstyle{definition}
\newtheorem{definition}[theorem]{Definition}

\theoremstyle{remark}
\newtheorem{remark}[theorem]{Remark}
\newtheorem{remarks}[theorem]{Remarks}
\newtheorem{example}[theorem]{Example}

\numberwithin{equation}{section}
\numberwithin{figure}{section}
\input{Definitions.tex}
\begin{document}
%-----------------------
%
\title[Groups of PL-homeomorphisms admitting invariant characters] 
{Groups of PL-homeomorphisms  admitting \\non-trivial invariant characters}
% first author
\author[D. L. Gon\c{c}alves] {Daciberg L. Gon\c{c}alves}
\address{Departametno de Matem\'atica - IME, Universidade de S\~ao Paulo\\
Caixa Postal 66.281 - CEP 05314-970, S\~ao Paulo - SP, Brasil}
\email{dlgoncal@ime.usp.br}
%
% second author
\author[P. Sankaran]{Parameswaran Sankaran}
\address{The Institute of Mathematical Sciences,
CIT Campus, Taramani, Chennai 600113, India}
\email{sankaran@imsc.res.in}
%
% third author
\author[R. Strebel]{Ralph Strebel}
\address{D\'{e}partement de Math\'{e}matiques,
Chemin du Mus\'{e}e 23,
Universit\'{e} de Fribourg,
1700 Fribourg, Switzerland}
\email{ralph.strebel@unifr.ch}

\begin{abstract}
We show 
that several classes of groups $G$ of PL-homeomorphisms of the real line
admit non-trivial homomorphisms  $\chi \colon G \to \R$ 
that are fixed by every automorphism of $G$.
The classes enjoying the stated property include the generalizations of Thompson's group $F$ 
studied by K. S. Brown, M. Stein, S. Cleary and Bieri-Strebel in \cite{Bro87a, Ste92, Cle95, Cle00, BiSt14},
but also the class of groups investigated by Bieri et al.\,in \cite[Theorem 8.1]{BNS}.
It follows that every automorphism of a group in one of these classes 
has infinitely many associated twisted conjugacy classes. 
\end{abstract}

\maketitle

%Version: finished on 09/09/2015
%
\setcounter{tocdepth}{1}
\tableofcontents
%\newpage

%=======
%
\section{Introduction}
\label{sec:Intro}
%
%=======
%
This paper has its origin in two articles \cite{BFG08} and \cite{GoKo10}
about twisted conjugacy classes of Thompson's group $F$.
In order to describe the aim of the cited papers, we recall some terminology.
Let $G$ be a group and $\alpha$ an automorphism of $G$.
Then $\alpha$ gives rise to an action $\mu_\alpha \colon G \times |G| \to |G|$
of $G$ on its underlying set $|G|$, 
defined by
\begin{equation}
\label{eq:alpha-action}
 \mu_\alpha (g, x) = g \cdot x \cdot \alpha(g)^{-1}.
\end{equation}
The orbits of this action are called \emph{twisted conjugacy classes},
or  \emph{Reidemeister} classes, of $\alpha$.
The twisted conjugacy classes of the identity automorphism, for instance, 
are nothing but the conjugacy classes.

Two questions now arise, firstly,  
whether a given automorphism $\alpha$ has infinitely many orbits
and, secondly, whether every automorphism of $G$ has infinitely many orbits.
As the latter property will be central to this paper,
we recall the definition of property $R_\infty$:
\begin{definition}
\label{definition:Property-Rinfty}
A group $G$ is said to have \emph{property} $R_\infty$ 
if the action $\mu_\alpha$ has infinitely many orbits for every automorphism $\alpha \colon G \iso G$.
\end{definition}

The problem of determining 
whether a given group, or a class of groups, satisfies property $R_\infty$ 
has attracted the attention of several researchers.
The  problem is rendered particularly interesting  by the fact 
there does not exist a uniform method of solution.
Indeed, a variety of techniques and ad hoc arguments from several branches of mathematics 
have been used to tackle the problem,
notably combinatorial group theory in \cite{GoWo09},
geometric group theory in \cite{LeLu00}),
$C^*$-algebras in \cite{FeTr12}
 and algebraic geometry in \cite{MuSa14b}.

 In \cite{BFG08} Thompson's group $F$ is shown to enjoy property $R_\infty$,
while \cite{GoKo10} establishes that property for Thompson's group $F$,
but also for many other groups $G$ having  the peculiarity 
that the complement of their BNS-invariant $\Sigma^1(G)$ is made up of finitely many rank 1 points.
In this paper,
we generalize both approaches and prove in this way
that many classes of groups of PL-homeomorphisms have property $R_\infty$.
%s
%-------------- 
\subsection{A useful fact}
\label{ssec:Crucial-fact}
%---------------
%
The papers by C. Bleak et al. and by Gonçalves-Kochloukova both exploit the following observation:
let $\alpha$ be an automorphism of a group $G$, 
let $\psi \colon G \to B$ be a homomorphism into an \emph{abelian} group
and assume $\psi$ is fixed by $\alpha$.
Then $\psi$ is constant on twisted conjugacy classes of $\alpha$;
indeed,
if the elements $x$ and $y$ lie in the same twisted conjugacy class 
there exists $z \in G$ so that $y = z \cdot x \cdot \alpha(z)^{-1}$;
the computation
\begin{equation*}
\psi(y) = \psi(z \cdot x \cdot \alpha(z)^{-1})=
\psi(x ) \cdot \psi(z) \cdot \left((\psi \circ \alpha)(z)\right) ^{-1} = \psi(x)
\end{equation*}
then establishes the claim.
\emph{A group $G$ has therefore property $R_\infty$ 
if it admits a homomorphism onto an \emph{infinite, abelian} group}
that is fixed by every automorphism of $G$.
%
%-------------- 
\subsection{Approach used by C. Bleak et al.}
\label{ssec:Approach-in-BFG08}
%---------------
%
In \cite{BFG08} the authors establish 
that Thompson's group $F$ has property $R_\infty$ by using the mentioned fact.
To find the homomorphism $\psi$,
they use a representation of $F$ by piecewise linear homeomorphisms of the real line:
$F$ is isomorphic to the group of all piecewise linear homeomorphisms $f$
with supports in the unit interval $I = [0,1]$,
slopes a power of 2, and \emph{break points}, 
\ie{} points where the left and right derivatives differ,
in the group $\Z[1/2]$ of dyadic rationals;
see, \eg{} \cite[p.\,216, §\,1]{CFP96}.
This representation affords them with two homomorphisms $\sigma_\ell$ and $\sigma_r$,
given by the right derivative in the \emph{left} end point $0$ 
and the left derivative in \emph{right} end point 1 of $I$;
respectively;
in formulae
\begin{equation}
\label{eq:definition-sigma-ell-sigma-r}
\left\{
\begin{aligned}
\sigma_\ell (f) &= \lim\nolimits_{\,t \searrow 0} f'(t),\\
\sigma_r(f) &= \lim\nolimits_{\,t \nearrow 1} f'(t),\\
\end{aligned}
\right.
\end{equation}
The images of $\sigma_\ell$ and $\sigma_r$ are both equal to $\gp(2)$,
the (multiplicative) cyclic group generated by the natural number $2$.
Theorem 3.3,
the main result of \cite{BFG08}, 
can be rephrased by saying that the homomorphism
\[
\psi \colon F \to \gp(2), \quad f \mapsto \sigma_\ell(f) \cdot \sigma_r(f)
\]
is fixed by every automorphism $ \alpha$ of $F$.
Its proof uses the very detailed information about $\Aut F$
established by M. Brin in \cite{Bri96}.
%
%-------------- 
\subsection{A generalization}
\label{ssec:Generalization-approach-BFG08}
%---------------
%
The stated description of Thompson's group $F$ 
invites one to introduce generalized groups of type $F$ in the following manner.
\begin{definition}
\label{definition:G(I;A,P)}
Let $I \subseteq \R$ be a closed interval,
$P$ be a subgroup of the multiplicative group of positive reals $\R^\times_{>0}$,
and $A$ be a subgroup of the additive group $\R_{\add}$ of the reals 
that is stable under multiplication by $P$.
Let $G(I;A,P)$ denote  the subset of $\PL_o(\R)$ made up of all PL-homeomorphisms $g$ satisfying 
the conditions:
\begin{enumerate}[a)]
\item the support $\supp g = \{t \in \R \mid g(t) \neq t \}$ of $f$ is contained in $I$,
\item the slopes of the finitely many line segments forming the graph of $g$ lie in $P$,
\item the break points of $g$ lie in $A$, and
\item $g$ maps $A$ onto $A$.
\end{enumerate}
\end{definition}
\begin{remarks}
\label{remarks:G(I;A,P)}
a) The subset $G(I;A,P)$ is closed under composition
\footnote{In this article we use \emph{left} actions
and the composition of functions  familiar to analysts;
thus $g_2 \circ g_1$ denotes the function $t \mapsto g_2(g_1(t))$ 
and $\act{g_1}{-0.5}{g_2}$ the homeomorphism $g_1 \circ g_2 \circ g_1^{-1}$.}
and inversion.
The set $G(I;A,P)$ equipped with these operations is a group; 
by abuse of notation, it will also be denoted by $G(I;A,P)$.

b) We shall always require that neither $P$ nor $A$ be reduced to the neutral element.
These requirements imply
that $A$ contains arbitrary small positive elements 
and thus $A$ is a dense subgroup of $\R$.
As concerns the interval $I$ we shall restrict attention to three types of intervals:
compact intervals with endpoints 0 and $b \in A_{>0}$,
the half line $[0, \infty[$ and the line $\R$;
we refer the reader to \cite[Sections 2.4 and 16.4]{BiSt14}
for a discussion of the groups associated to other intervals. 

c) The idea of introducing and studying the groups $G(I;A,P)$
goes back to the papers \cite{BrSq85} and \cite{BiSt85}.
\end{remarks}

%-------------- 
\subsubsection{The homomorphisms $\sigma_\ell$, $\sigma_r$ and $\psi$}
\label{sssec:Homomorphisms-sigma-ell-etc}
%---------------
%
The definitions of $\sigma_\ell$ and $\sigma_r$,
given in formula \eqref{eq:definition-sigma-ell-sigma-r},
admit straightforward extensions  to the groups $G(I;A,P)$;
note, however, that in case of the half line $[0, \infty[$,
the number $\sigma_r(f)$ will denote the slope of $f$ near $+\infty$,
and similarly for $I = \R$ and $\sigma_\ell$,  $\sigma_r$.  
The homomorphisms $\sigma_\ell$ and $\sigma_r$ 
allow one then to introduce an analogue of  $\psi \colon F \to \gp(2)$,
namely
\begin{equation}
\label{eq:Definition-psi}
\psi \colon G = G(I;A,P)  \to P, \quad g \mapsto \sigma_\ell(g) \cdot \sigma_r(g).
\end{equation}

There remains the question 
whether this homomorphism $\psi$ is fixed by every automorphism of $G$.
In the case of Thompson's group $F$ the question has been answered in the affirmative
by exploiting the detailed information about $\Aut F$ obtained by M. Brin in \cite{Bri96}.
Such a detailed description is not to be expected for every group of the form $G(I;A,P)$;
indeed, the results in \cite{BrGu98} show 
that the structure of the automorphism group gets considerably more involved
if one passes from the group $G([0,1];\Z[1/2], \gp(2))$,
the group isomorphic to $F$,
to the groups $G([0,1];\Z[1/n], \gp(n))$ with $n$ an integer greater than 2.

%-------------- 
\subsubsection{The first main results}
\label{sssec:First-main-result}
%---------------
%
It turns out
that one does not need very detailed information about $\Aut G(I;A,P)$ 
in order to construct a non-trivial homomorphism $\psi \colon G(I;A,P) \to \R^\times_{>0}$
that is fixed by every automorphism of the group $G(I;A,P)$;
it suffices to go back to the findings in the memoir \cite{BiSt85}
and to supplement them by some auxiliary results based upon them.
\footnote{The memoir \cite{BiSt85} has recently been reedited by R. Strebel 
and published in the repository of electronic preprints \emph{arXiv}; 
see \cite{BiSt14} for the precise reference.}
A first  outcome is
\begin{theorem}
\label{thm:Generalization-of-BFG08}
Assume the interval $I$, the group of slopes $P$ 
and the $\Z[P]$-module $A$ are as in Definition \ref{definition:G(I;A,P)}
and in Remark \ref{remarks:G(I;A,P)}b.
Then there exists an epimorphism $\psi \colon G(I;A,P) \epi P$ 
that is fixed by every automorphism of $G$.

The group $G(I;A,P) $ has therefore property $R_\infty$.
\end{theorem}

\begin{remark}
\label{remark:Subgroup-B}
Let $I$, $A$ and $P$ as before
and let $B = B(I;A,P)$ be the subgroup of $G(I;A,P)$ 
made up of all elements $g$ that are the identity near the endpoints.
Then $B$ is a characteristic subgroup of $G(I;A,P)$
and variations of Theorem \ref{thm:Generalization-of-BFG08}
hold  for many subgroups $G$ of $G(I;A,P)$ with  $B \subset G$;
for details, 
see Theorems 
\ref{thm:Existence-psi-I-compact-all-autos}, 
\ref{thm:Existence-psi-I-halfline-image-rho-non-abelian}
and  \ref{thm:Existence-psi-I-line-image-rho-non-abelian}.
\end{remark}
%
%
%-------------- 
\subsection{Route taken by Gonçalves-Kochloukova in \cite{GoKo10} }
\label{ssec:Approach-in-GoKo10}
%---------------
%
The proof of Theorem \ref{thm:Generalization-of-BFG08} 
does not exploit information about $\Aut G(I;A,P)$
that is as precise as that going into the proof of the  main result of \cite{BFG08}.
It uses, however, non-trivial features of the automorphisms of $G(I;A,P)$.
In \cite{GoKo10} Gonçalves and Kochloukova put forward the novel idea
of replacing detailed information about $\Aut G$ 
by information about the form of the BNS-invariant of the group $G$;
they carry out this program for the generalized Thompson group $F_{n,0}$ with $n \geq 2$, 
a group isomorphic to $G([0,1];\Z[1/n], \gp(n))$,
 and for many other groups.

In a nutshell, their idea is this.
Suppose $G$ is a finitely generated group
for which the complement of $\Sigma^1(G)$ is \emph{finite}.
\footnote{Recall that $\Sigma^1(G)$ is a certain subset of the space of all half lines 
$\R_{>0} \cdot \chi$ 
emanating from the origin of the real vector space $\Hom(G, \R)$, 
and that $\Aut(G)$ acts canonically on this subset, as well on as its complement.}
Then every automorphism of $G$ permutes the finitely many rays in $\Sigma^1(G)^c$.
This suggests that it might be possible to construct  a new ray $\R_{>0}\cdot \chi_0$ 
that is fixed by $\Aut G$.
If one succeeds in doing so,
then $\R \cdot \chi_0$ will be a 1-dimensional sub-representation 
of the finite dimensional real vector space $\Hom(G, \R)$,
acted on by $\Aut G$ via 
\[
(\alpha, \chi) \mapsto \chi \circ \alpha^{-1}.
\] 
\emph{A priori},  this invariant line need not be fixed pointwise.

Gonçalves and Kochloukova detected 
that the line $\R \cdot \chi_0$ is fixed pointwise by $\Aut G$ 
if the homomorphism $\chi_0\colon G \to \R$ has \emph{rank} 1,
\ie{} if its image  is infinite cyclic.
Using this fact they were then able to prove 
that Thompson's group $F$, but also many other groups $G$,
admit a rank 1 homomorphism that is fixed by $\Aut G$
and thus satisfy property $R_\infty$.
%

%-------------- 
\subsection{A generalization}
\label{ssec:Generalization-approach-GoKo10}
%---------------
%
In the second part of this paper 
we generalize the approach of Gon\c{c}alves and Kochloukova 
to PL-homeomorphism groups $G$  
for which $\Sigma^1(G)^c$ may contain points of rank greater than $1$. 
In pursuing this goal,
one runs into the following difficulty.
Suppose  $\R_{>0} \cdot \chi_0$ is a ray
that is fixed  by $\Aut G$ as a \emph{set}.
There may then exist an automorphism $\alpha$ 
which acts on the ray by multiplication by a positive real number $s \neq 1$;
if so,
the 1-dimensional subspace $\R \cdot \chi_0$ in the real vector space $\Hom(G,\R)$ 
is an eigenline with eigenvalue $s \neq 1$ of the linear transformation $\alpha^*$ 
induced by $\alpha$ on $\Hom(G,\R)$.
The existence of such an eigenvalue $s \neq 1$ would ruin our plan,
but, as we shall see,
it can be ruled out 
if the image of the character $\chi_0$ has only $1$ and $-1$ as units,
where the groups of units is defined as follows:
\begin{definition}
\label{definition:Group-of-units}
Given a  subgroup $B$ of the additive group $\R_{\add}$
we set  
\begin{equation}
\label{eq:Units-B}
U(B) = \{s \in \R^\times \mid s \cdot B = B \}
\end{equation}
and call $U(B)$ the \emph{group of units} of $B$ (inside the multiplicative group of $\R$).
\end{definition}

We shall establish the following 
\begin{theorem}
\label{thm:Generalization-GoKo10}
Suppose $G$ is a subgroup of $PL_o(I)$ 
that satisfies the following conditions: 
\begin{enumerate}[(i)]
\item no interior point of the interval $I = [0, b]$ is fixed by $G$;
\item the characters $\chi_\ell$ and $\chi_r$ are both non-zero;
\item the quotient group $G/(\ker \chi_\ell \cdot \ker \chi_r )$ is a torsion-group, and
\item at least one of the groups  $U(\im \chi_\ell)$ and   $U(\im \chi_r)$ is reduced  to $\{1, -1\}$.  
\end{enumerate}
Then there exists a non-zero homomorphism $\psi \colon G \to \R^\times_{>0}$
that is fixed by every automorphism of $G$.
The group  $G$ has therefore property $R_\infty$.
\end{theorem}

There remains the problem of finding subgroups $B \subset \R_{\add}$ 
that have only trivial units.
This problem is addressed in section \ref{ssec:Group of units}.
We shall show that a subgroup $B = \ln P$ has this property 
if the multiplicative group $P \subset \R^\times_{>0}$
is free abelian and generated by algebraic numbers.
In addition, we shall construct in section \ref{ssec:Uncountable-class}
a collection $\GG$ of pairwise non-isomorphic 3-generator groups $G_s$ 
enjoying the properties
that each group $G_s$ satisfies the assumptions of Theorem \ref{thm:Generalization-GoKo10}
and that the cardinality of $\GG$ is that of the continuum.

\subsection*{Acknowledgments} 
%\textcolor{red}{Daçiberg, please check all the words in Portuguese}
The first  author has been partially supported by the Fapesp project
``Topologia Alg\'ebrica, Geom\'etrica e Diferencial-2012/24454-8''.
The second author acknowledges financial support from the  
Department of Atomic Energy, Government of India, under a XII plan project.
The present work was initiated during the visit of the second author to the
University of S\~ao Paulo in August 2012. He thanks the first author for
the invitation and the warm hospitality. 
The third author thanks M. Brin and M. Zaremsky for numerous, helpful discussions.

%\newpage
%
%========
\section{Preliminaries on automorphisms of the groups $G(I;A,P)$}
\label{sec:Preliminaries-autos-G(I;A,P)}
%========
The groups $G(I;A, P)$ form a class of subgroups of the group $\PL_o(\R)$,
the group of all orientation preserving, 
piecewise linear homeomorphisms of the real line.
They enjoy some special properties,
in particular the following two:
each group acts approximately
\footnote{See \cite[Chapter A]{BiSt14} for details.} 
highly transitively on the interior of $I$
and all its automorphisms are induced by conjugation by homeomorphisms.
It is above all this second property that will be exploited in the sequel.

In this section,
we recall the basic representation theorem for automorphisms of the groups $G(I;A,P)$
and deduce then some consequences.
%
%----------
\subsection{Representation of isomorphisms}
\label{ssec:Representation-automorphisms-G(I;A,P)}
%----------
We begin by fixing the set-up of this section:
$P$ is a non-trivial subgroup of $\R^\times_{>0}$ 
and $A$ a non-zero subgroup of $\R_{\add}$ that is stable under multiplication by $P$.
Next,
$I$ is a closed interval of positive length;
we assume the left end point of $I$ is in $A$ if $I$ is bounded from below 
and similarly for the right end point.

\begin{remark}
\label{remark:Distinct-intervals}
Distinct intervals $I_1$, $I_2$ may give rise to isomorphic groups $G(I_1;A,P)$ and $G(I_2;A, P)$.
In  particular, it is true 
that every group $G(I_1;A,P)$ is isomorphic to one 
whose interval $I_2$  has one of the following three forms
\begin{equation}
\label{eq:Types-of-intervals}
[0,b] \text{ with } b \in A, 
\quad
[0, \infty[
\quad \text{and}\quad
\R.
\end{equation}
(see Sections 2.4 and 16.4 of \cite{BiSt14} for proofs).
\end{remark}

We come now to the announced result about isomorphisms of groups 
$G(I;A,P)$ and $G(\bar{I}; \bar{A}, \bar{P}$.
It asserts that each isomorphism of the first group onto the second one 
is induced by conjugation 
by a homeomorphism of the interior $\Int(I)$ of $I$ onto the onterior of $\bar{I}$.
This claim holds even for suitably restricted subgroups of 
$G(I;A,P)$ and of $G(\bar{I};\bar{A}, \bar{P})$.
In order to state the generalized assertion 
we need the subgroup of ``bounded elements''.
\begin{definition}
\label{definition:B}
Let $B(I;A,P)$ be the subgroup of $G(I;A,P)$ consisting of all PL-homeomorphisms $f$
that are the identity near the end points or, more formally, that satisfy the inequalities
$\inf I < \inf \supp f$ and $\sup \supp f < \sup I$.
\end{definition}

We are now in a position to state the representation theorem.
\begin{theorem}
\label{thm:TheoremE16.4}
Assume $G$ is a subgroup of $G(I;A,P)$ 
that contains the derived subgroup of $B(I;A,P)$,
and $\bar{G}$ is a subgroup of $G(\bar{I};\bar{A}, \bar{P})$ 
containing the derived group of $B(\bar{I};\bar{A},\bar{P})$.
Then every isomorphism $\alpha \colon G \iso \bar{G}$ is induced by conjugation 
by a unique homeomorphism $\varphi_\alpha$ of the interior $\Int(I)$ of $I$ 
onto the interior of  $\bar{I}$;
more precisely, the equation
\begin{equation}
\label{eq;Inducing-alpha}
\alpha(g) \restriction{\Int(\bar{I})}  = \varphi_\alpha \circ (g \restriction{\Int(I)}) \circ \varphi_\alpha ^{-1}
\end{equation}
holds for every $g \in G$. 
Moreover,  $\varphi_\alpha$ maps $A \cap \Int(I)$ onto  $\bar{A} \cap \Int(\bar{I})$
\end{theorem}
\begin{proof}
The result is a restatement of  \cite[Theorem E16.4]{BiSt14}. 
\end{proof}

\begin{remarks}
\label{remarks:varphi}
a)
Theorem \ref{thm:TheoremE16.4} has two simple, but important consequences.
First of all, every homeomorphism of intervals is either increasing or decreasing;
since the homeomorphism  $\varphi_\alpha$ 
inducing an isomorphism $\alpha \colon G \iso \bar{G}$ is uniquely determined by $\alpha$, 
there exist therefore two types of isomorphisms:
the \emph{increasing} isomorphisms induced by conjugation by an increasing homeomorphism
and the decreasing ones.

Assume now that $\bar{I} = I$.
It the homeomorphism $\varphi_\alpha \colon \Int(I) \iso \Int(I)$ is increasing 
it extends uniquely 
to a homeomorphism of $I$, 
but this may not be so if it is decreasing.
Indeed, $\varphi_\alpha$ extends if $I$ is a compact interval or the real line,
but not if $I$ is a half line.
If the extension exists, it will be denoted by $\tilde{\varphi}_\alpha$.

b) The increasing automorphisms of a group $G$
form a subgroup $\Aut_+ G $ of $\Aut G$ of index at most 2.
It will turn out that is often easier to find a non-zero homomorphism $\psi \colon G \to B$
that is fixed by the subgroup $\Aut_+ G$ 
than a non-zero homomorphism fixed by $\Aut G \smallsetminus \Aut_+ G$
(in case this set is non-empty).
For this reason, it is useful to dispose of criteria
guaranteeing that $\Aut G = \Aut_+ G$.

c) The derived group of $B(I;A,P)$ is a simple, infinite group 
(see \cite[Theorem C10.2]{BiSt14}), but $B(I;A,P)$ itself may not be perfect.
To date,
no characterization of the parameters $(I, A,P)$ corresponding to perfect groups $B(I;A,P)$ is known.
The quotient group $G(I;A,P)/B(I;A,P)$, on the other hand,
is a metabelian group that can be described explicitly in terms of the triple $(I, A, P)$ 
(see Section 12 and 5.2 in \cite{BiSt14}).
In the sequel,
we shall therefore restrict attention to subgroups $G$ containing $B(I;A,P)$.

d) The second important consequence of Theorem \ref{thm:TheoremE16.4} 
is the fact that $B(I;A,P)$ is a characteristic subgroup of every subgroup $G$
with $B(I;A,P) \subseteq G \subseteq G(I;A,P)$ 
(the proof is easy; 
see \cite[Corollary E16.5]{BiSt14}
or Corollary \ref{crl:alpha-and-ker-lambda} below).
\end{remarks}

In part a) of the previous remarks
the term \emph{increasing isomorphism} has been introduced.
In the sequel, this parlance will be used often and so we declare: 
\begin{definition}
\label{definition:Increasing-isomorphism}
Let $\alpha \colon G \iso \bar{G}$ be an isomorphism induced by the (uniquely determined) homeomorphism 
$\varphi_\alpha \colon \Int(I) \iso\Int(\bar{I})$.
If $\varphi$ is \emph{increasing} (respectively \emph{decreasing}) 
then $\alpha$ will be called increasing (respectively decreasing).
\end{definition}
%
%-------------- 
\subsection{The homomorphisms $\lambda$ and $\rho$ }
\label{ssec:lambda-rho}
%---------------
%
By Remark \ref{remarks:varphi}d the group $B = B(I;A,P)$ is a characteristic subgroup of every group $G$ containing it.
Now $G$ has, in addition,  subgroups containing $B$ 
that are invariant under the subgroup $\Aut_+ G$,
namely the kernels of the homomorphisms $\lambda$ and $\rho$.
To set these homomorphisms into perspective,
we go back to the homomorphisms 
\[
\sigma_\ell \colon G(I;A,P) \to P \quad \text{and} \quad  \sigma_r\colon G(I;A,P) \to P,
\]
introduced in section \ref{sssec:Homomorphisms-sigma-ell-etc}.
Their images are abelian and coincide with the group of slopes $P$.
If $I$ is not bounded from below,
there exist a homomorphism $\lambda$, related to $\sigma_\ell$,
whose image is contained in $\Aff(A,P)$,
the group of all affine maps of $\R$ with slopes in $P$ and displacements $f(0) \in A$.
The definition of $\lambda$ is this:
\begin{equation}
\label{eq:Definition-lambda}
\lambda \colon G(I;A,P) \to \Aff(A,P),
\quad
 g \mapsto \text{affine map coinciding with $g$ near $-\infty$}.
 \end{equation}
 If the interval $I$ is not bounded from above,
 there exists a similarly defined homomorphism $\rho \colon G(I;A,P) \to \Aff(A,P)$,
 given by
 \begin{equation}
\label{eq:Definition-rho}
\rho(g) =\text{ affine map coinciding with $g$ near $+\infty$}.
 \end{equation}
 The images of $\lambda$ and $\rho$ are, in general, smaller than $\Aff(A,P)$.
 They are equal to the entire group $\Aff(A,P)$ if $I = \R$;
 if $I$ is not bounded from below, but bounded from above, 
 the image of $\lambda$ is $\Aff(IP \cdot A, P)$ and the analogous statement holds for $\rho$.
 In the above, $IP \cdot A$ denotes the submodule of $A$ 
 generated by the products $(p-1) \cdot a$ with $p \in P$ and $a \in A$
 (see Section 4 and Corollary A5.3 in \cite{BiSt14}).

For uniformity of notation, we extend the definition of $\lambda$ and $\rho$ 
to compact intervals:
if $I = [0,b]$ and $f \in G(I;A,P)$ then $\lambda(g)$
is the linear map $t \mapsto \sigma_\ell(g) \cdot t$
and $\rho(g)$ is the affine map $ t \mapsto \sigma_r(f) \cdot (t-b)  + b$.
Similarly one defines $\lambda (g)$ if $I$ is the half line $[0, \infty[$.

The homomorphisms $\lambda$ and  $\rho$ allow one to restate the definition of $B(I;A,P)$;
one has
\begin{equation}
\label{eq:Reexpressing-B}
B(I;A,P) = \ker \lambda \cap \ker \rho.
\end{equation}
\begin{remark}
\label{remark:Restrictions-lambda-rho}
In the sequel, 
we shall often deal with subgroups, denoted $G$, of a group $G(I;A,P)$ that contain $B(I;A,P)$.
For ease of notation, 
we shall then denote the restrictions of 
$\lambda$ and $\rho$ to $G$ again by $\lambda$ and $\rho$.
\end{remark} 
 %
%-------------- 
\subsection{First consequences of the representation theorem}
\label{ssec:Representation-first-consequences}
%---------------
%
Let $G$ be a subgroup of $G(I;A,P)$ 
that contains the derived subgroup of $B(I;A,P)$ 
and let $\bar{G}$ be a subgroup of $G(\bar{I};\bar{A},\bar{P})$ 
containing the derived subgroup of $B(\bar{I};\bar{A},\bar{P})$.
Suppose $\varphi_\alpha$ is a homeomorphism of $\Int(I)$ onto $\Int(\bar{I})$
that induces an isomorphism $\alpha  \colon G  \iso \bar{G}$.
The map  $\varphi_\alpha$ need not be piecewise linear.
Theorem \ref{thm:TheoremE16.4}, however, has useful consequences even in such a case.
One implication is recorded in 
\begin{crl}
\label{crl:alpha-and-ker-lambda}
Assume $G$ and $\bar{G}$ are subgroups of $G(I;A,P)$ both of which contain $B(I;A,P)$,
and let $\lambda,\rho  \colon G \to \Aff(A,P)$ and $\bar{\lambda},\bar{\rho} \colon \bar{G} \to \Aff(A,P)$
be the obvious restrictions of the homomorphisms $\lambda, \rho$ introduced in section \ref{ssec:lambda-rho}.
Consider now an isomorphism $\alpha \colon G \iso \bar{G}$
that is induced by the homeomorphism $\varphi_\alpha \colon \Int(I) \iso \Int(I)$.
If $\varphi_\alpha$ is increasing then 
\begin{enumerate}[(i)]
\item $\alpha$ maps $\ker \lambda$ onto $\ker \bar{\lambda}$ 
and induces an isomorphism $\alpha_\ell$ of $G/\ker \lambda$  onto $\bar{G}/\ker \bar{\lambda}$;
\item $\alpha$ maps $\ker \rho$ onto $\ker \bar{\rho}$ and 
induces an isomorphism $\alpha_r$ of $G/\ker \rho$ onto $\bar{G}/\ker \bar{\rho}$.
\end{enumerate}
\end{crl}

\begin{proof}
(i) If  $g \in \ker \lambda$ then $g$ is the identity near $\inf I$.
As $\varphi_\alpha$ is \emph{increasing},
the image  $\alpha (g) = \varphi_\alpha \circ g \circ \varphi_\alpha^{-1}$ of $g$
is therefore also the identity near $\inf I$.
It follows that $\alpha(\ker \lambda) \subseteq \ker \bar{\lambda}$.
This inclusion is actually an equality, for $\alpha^{-1}:\bar{G}\to G$ is an 
isomorphism and so $\alpha^{-1}(\ker \bar{\lambda})\subseteq \ker \lambda$. 
Claim (ii) can be proved similarly.
 \end{proof}
 %
%===
\subsection{Automorphisms induced by finitary PL-homeomorphisms}
\label{ssec:Autos-induced-by-PL-homeomorphisms}
%===
%
Let $G \subseteq G(I;A,P)$ be as before,
and let $\alpha$ be an automorphism of $G$.
According to Theorem \ref{thm:TheoremE16.4},
$\alpha$ is induced by conjugation by a unique auto-homeomorphism $\varphi_\alpha$.
This auto-homeomorphism may not be piecewise linear,
but the situation improves if $P$, the group of slopes,
is not cyclic (and hence dense in $\R^\times_{>0}$):
\begin{theorem}
\label{thm:TheoremE17.1}
Suppose $P$ is \emph{not cyclic}.
For every automorphism $\alpha$ of $G$
there exists then a non-zero real number $s$ 
such that $A = s \cdot A$
and that the auto-homeomorphism $\varphi_\alpha  \colon \Int(I) \iso \Int(I)$ is piecewise linear
with slopes in the coset $s \cdot P$ of $P$.
Moreover, $\varphi_\alpha$ maps the subset $A \cap \Int(I)$ onto itself
and has only finitely many breakpoints in every compact subinterval of $\Int(I)$. 
\end{theorem}
\begin{proof}
The result is a special case of \cite[Theorem E17.1]{BiSt14}.
\end{proof}

Theorem \ref{thm:TheoremE17.1} indicates 
that automorphisms of groups with a non-cyclic group of slopes $P$ 
are easier to analyze than those of the groups with cyclic $P$.
Note, however, 
that the conclusion of Theorem \ref{thm:TheoremE17.1} does not rule out
that $\varphi_\alpha$ has infinitely many breakpoints 
which accumulate in one or both end points
\footnote{The notion of end point is to be interpreted suitably if $I$ is not bounded.}
and so $\varphi_\alpha$ may not be differentiable in the end points.
In section \ref{ssec:Differentiability-criterion}
we shall therefore be interested in differentiability criteria.
%
%========
\section{Characters fixed by $\Aut G([0,b];A,P)$}
\label{sec:Homomorphisms-fixed-by-Aut-I-compact}
%========
In this section,
we prove Theorem \ref{thm:Generalization-of-BFG08} for the case of a compact interval
and various extensions of it.
An important ingredient in the proofs of these results is a criterion 
that allows one to deduce 
that an auto-homeomorphism $\varphi_\alpha$ 
inducing an automorphism $\alpha$ of the group 
is differentiable near one or both of its end points.
%
%===
\subsection{A differentiability criterion}
\label{ssec:Differentiability-criterion}
%===
%
The proof of the criterion is rather involved.
Prior to stating the criterion and giving its proof,
we discuss therefore a result 
that explains the interest in the criterion.
\begin{prp}
\label{prp:Importance-of-differentiability-I-compact}
Let $G$ be a subgroup of $G([a,b];A,P)$ 
that contains the derived subgroup of $B(I;A,P)$.
Suppose $\tilde{\varphi} \colon [0,b] \iso [0,b]$ is an auto-homeomorphism 
that induces, by conjugation, an automorphism $\alpha$ of $G$.
Then the following statements hold:
\begin{enumerate}[(i)]
\item if $\tilde{\varphi}$ is increasing, differentiable in 0 and $\tilde{\varphi}'(0) > 0$  
then $\alpha$ fixes $\sigma_\ell$;
\item if $\tilde{\varphi}$ is increasing, differentiable in $b$ with $\tilde{\varphi}'(b) > 0$ 
then $\alpha$ fixes $\sigma_r$;
\item if $\tilde{\varphi}$ is differentiable both in 0 and in $b$,
with non-zero derivatives,
then $\alpha$ fixes the homomorphism 
$\psi \colon g \mapsto \sigma_\ell (g) \cdot \sigma_r(g)$.
\end{enumerate}
\end{prp}
\begin{proof}
\text{(i) and (ii) } 
Suppose the extended auto-homeomorphism  
$\tilde{\varphi}= \tilde{\varphi}_\alpha$ is increasing
and fix $g \in G$.
If $\tilde{\varphi}$ is differentiable in 0 and $\tilde{\varphi}'(0) > 0$,
the chain rule justifies the following computation
\begin{equation}
\label{eq:Calculation-alpha-increasing}
\sigma_\ell(\alpha(g)) 
=
(\tilde{\varphi} \circ g \circ \tilde{\varphi}^{-1})(0)
=
\tilde{\varphi}'(0) \cdot g'(0) \cdot (\tilde{\varphi}^{-1})'(0)
=
\sigma_\ell(g)
\end{equation}
It follows that $\sigma_\ell$ is fixed by $\alpha$.
If $\tilde{\varphi}$ admits a left derivative in $b$ and if $\tilde{\varphi}'(b) > 0$,
one sees similarly, that $\sigma_r$ is fixed by $\alpha$.
\smallskip

\text{(iii) } 
Assume now that $\tilde{\varphi}= \tilde{\varphi}_\alpha$ is differentiable, both in 0 and in $b$,
and that both derivatives are different from 0.
If $\tilde{\varphi}$ is \emph{increasing} parts (i) and (ii) guarantee
that $\sigma_\ell$ and $\sigma_r$ are fixed by $\alpha$, 
whence so is their product $\psi$.
If, on the other hand, $\tilde{\varphi}$ is \emph{decreasing}
the calculation
\begin{equation}
\label{eq:Calculation-alpha-decreasing}
\sigma_r(\alpha(g)) 
=
\left(\tilde{\varphi} \circ g \circ \tilde{\varphi}^{-1}\right)'(b)
=
\tilde{\varphi}'(0) \cdot g'(0) \cdot (\tilde{\varphi}^{-1})'(b)
=
\sigma_\ell(g)
\end{equation}
holds for every $g \in G$
and establishes the relation $\sigma_r \circ \alpha = \sigma_\ell$.

A similar calculation shows that the relation $\sigma_\ell \circ \alpha = \sigma_r$ is valid.
The claim for $\psi$ is then a consequence of the following computation:
\begin{align*}
\left(\psi \circ \alpha\right) (g) 
&=
\sigma_\ell(\alpha(g)) \cdot \sigma_r(\alpha(g))
=
\sigma_r(g) \cdot \sigma_\ell(g)
= \psi (g).\qedhere
\end{align*}
\end{proof}
%
%------------
\subsubsection{Statement and proof of the criterion}
\label{ssec:Differentiability-criterion-Statement-and-proof}
%----------

We come now to the criterion;
we choose a formulation that is slightly more general than 
what is needed for the case at hand;
the extended version will be used in  Section \ref{sec:Homomorphisms-fixed-by-Aut-I-half-line}.
\begin{prp}
\label{prp:phi-is-linear-I-compact}
Suppose $I$ is an interval of one of the forms $[0, b]$ or $[0, \infty[$,
and $G$ as well as $\bar{G}$ are subgroups of $G(I;A,P)$ that contain $B(I;A,P)$.
Assume $\tilde{\varphi} \colon I \iso I$ is an \emph{increasing} auto-homeomorphism
that induces, by conjugation, an isomorphism $\alpha$ of the group $G$ onto the group $\bar{G}$.

If the image of $\sigma_\ell \colon G \to P$ is \emph{not cyclic},
then $\tilde{\varphi}$ is linear on a small interval of the form $[0,\delta]$
and so $\tilde{\varphi}$ is differentiable in 0 with positive derivative.
\end{prp}

\begin{proof}
The following argument uses ideas from the proofs 
of Proposition E16.8  and Supplement E17.3 in \cite{BiSt14}.
The proof will be divided into three parts.
In the first one,
we show that $\alpha \colon G \iso \bar{G}$ induces an isomorphism 
$\alpha_\ell \colon \im \sigma_\ell \iso \im \bar{\sigma}_\ell$,
that takes $p \in \im \sigma_\ell$ to $p^r = \e^{r\cdot \log p}$ 
for some positive real number $r$ that does \emph{not} depend on $p$.
In the second part,
we establish that $\tilde{\varphi}$ satisfies the relation
\begin{equation}
\label{eq:Represention-of-phi}
\tilde{\varphi}(p \cdot t) =  p^r \cdot \tilde{\varphi}(t)
\end{equation}
for every $p \in (\im \sigma_\ell \,\cap \,]0,1[\,)$
and $t$ varying in some small interval $[0, \delta]$.
In the last part,
we deduce from this relation that $\tilde{\varphi}$ is linear near 0.

We embark now on the first part.
Since $\varphi$ is increasing 
Corollary \ref{crl:alpha-and-ker-lambda} applies and shows
that $\alpha$ maps the kernel of $\sigma_\ell \colon G \to P$ 
onto the kernel of the homomorphism $\bar{\sigma}_\ell \colon \bar{G} \to P$,
and induces thus an isomorphism $\alpha_\ell \colon \im \sigma_\ell \iso \im \bar{\sigma}_\ell$
that renders the square
\begin{equation*}
\xymatrix{G \ar@{->}[r]^-{\alpha} \ar@{->>}[d]^-{\sigma_\ell} 
& \bar{G} \ar@{->>}[d]^-{\bar{\sigma}_\ell}\\
\im \sigma_\ell \ar@{->}[r]^-{\alpha_\ell} &\im \bar{\sigma}_\ell}
\end{equation*}
\noindent
commutative.
We claim
$\alpha_\ell$ maps the set $(\im \sigma_\ell) \, \cap \,]0, 1[$ 
onto $(\im \bar{\sigma}_\ell) \, \cap \,]0, 1[$.
Indeed,
let $p \in \im \sigma_\ell$ be a slope with $p < 1$ 
and let $f_p \in G$ be a preimage of $p$.  
Then $\alpha(f_p)$ is linear on some interval $[0, \varepsilon_p]$ 
and has there slope $\bar{\sigma}_\ell(\alpha(f_p)) = \alpha_\ell(p)$.  
Since $\tilde{\varphi}$ is continuous in 0,
there exists $\delta_p > 0$ 
so that $f_p$ is linear on $[0, \delta_p]$ 
and that $\tilde{\varphi}([0,\delta_p]) \subseteq [0,\varepsilon_p]$.
Fix $t \in [0,  \delta_p]$.
The hypothesis that $\alpha$ is induced by conjugation by $\tilde{\varphi}$ 
then leads to the chain of equalities 
\begin{align}
\label{eq:Representation-of-phi-1}
\tilde{\varphi} (p \cdot t)  
&=
(\tilde{\varphi}\circ f_p)(t) 
=
(\alpha(f_p) \circ \tilde{\varphi})( t)
=
\alpha(f_p) (\tilde{\varphi}( t))
=
\alpha_\ell(p) \cdot \tilde{\varphi}(t).
\end{align}
Since $\tilde{\varphi}$ is increasing and as $p < 1$,
the chain of equalities implies that $\alpha_\ell(p) < 1$.
It follows that $\alpha_\ell$ maps $(\im \sigma_\ell) \, \cap \, ]0,1[$ 
into $\im \bar{\sigma}_\ell \, \cap \, ]0,1[$
and then, by applying the preceding argument to $\varphi^{-1}$,
that 
\[
\alpha_\ell \left(\im \sigma_\ell \, \cap \, ]0,1[\,\right) = \im \bar{\sigma}_\ell\, \cap \, ]0,1[\;.
\]

We show next
that $\alpha_\ell(p) = p^r$ for all $p \in \im \sigma_\ell$ and some positive real number $r$. 
We begin by passing from the multiplicative subgroup  $\im \sigma_\ell \subset \R^\times_{>0}$  
to a subgroup of $\R_{\add}$; 
to that end, we introduce the homomorphism
\[
L_0 =  \ln \circ \, \alpha_\ell \circ \exp \colon \ln(\im \sigma_\ell) \iso \ln(\im \bar{\sigma}_\ell).
\]
The previous verification implies 
that $L_0$ is an order preserving isomorphism;
by the assumption on $\im \sigma_\ell$
the domain of $L_0$ is a dense subgroup of $\R_{\add}$. 
It follows 
that $L_0$ extends uniquely to an order preserving automorphism  
$L \colon \R_{\add} \to \R_{\add}$.
The homomorphism $L$ is continuous,
hence linear, and so given by multiplication by some positive real number $r$.
The isomorphism $\alpha_\ell $ has therefore the form  
\[
p \mapsto p^r = \exp(r \cdot \ln p ) \quad \text{with}\quad  r > 0.
\]

We come now to the second part of the proof.
Fix a slope $p_1 < 1$ in $\im \sigma_\ell$.
Formula \eqref{eq:Representation-of-phi-1} 
and the previously found formula for $\alpha_\ell$ then imply 
that there exists a small positive number $\delta_{p_1}$ such 
that the equation
\begin{equation}
\label{eq:Representation-of-phi-2}
\tilde{\varphi}(p_1 \cdot t)= p_1^r  \cdot  \tilde{\varphi} (t) 
\end{equation}
holds for every $t \in [0, \delta_{p_1}]$.
Consider next another slope $p < 1$.
There exists then, as before, a real number $\delta_p > 0$ 
so that $\tilde{\varphi}(p\cdot t) = p^r\cdot \tilde{\varphi}(t)$ for $t\in [0,\delta_p]$.  
Choose now $m\in \N$ so large that $p_1^m \cdot \delta_{p_1} \leq \delta_p$.  
The following chain of equalities then holds for each $t \in [0,\delta_{p_1}]$:
\begin{equation*}
p^{m\cdot r}_1 \cdot \tilde{\varphi}(p\cdot t) 
= 
\tilde{\varphi}(p_1^m\cdot p \cdot t) 
= 
\tilde{\varphi}( p \cdot p_1^m \cdot t) 
= 
p^r \cdot \tilde{\varphi}(p^m_1\cdot t)
= 
 p^r \cdot p^{m\cdot r}_1 \cdot \tilde{\varphi}(t).
\end{equation*}
The calculation shows
that $\tilde{\varphi}(p\cdot t) = p^r\cdot \tilde{\varphi}(t)$ for every $t\in [0,\delta_1]$. 
Upon setting $\delta = \delta_{p_1}$
one arrives at formula \eqref{eq:Represention-of-phi}. 

The proof is now quickly completed.
By assumption,
$\im \sigma_\ell$ is not cyclic and so formula \eqref{eq:Represention-of-phi} holds 
for a dense set of slopes $p$ and a fixed argument $t$, say $t = \delta$.
Since $\varphi$ is continuous and increasing,
formula \eqref{eq:Represention-of-phi} continues to hold for every real $x \in\; ]0, 1[$.
The formula
\[
\varphi (x\cdot \delta ) = \exp(r \cdot \ln x) \cdot \varphi(\delta) = x^r \cdot \varphi(\delta) 
\] 
is therefore valid for every $x \in \;]0, \delta]$. 
By Theorem \ref{thm:TheoremE17.1},  on the other hand,
$\varphi$ is piecewise linear on $]0,\delta]$.
So the exponent $r$ must be equal to 1,
whence $\varphi$ is linear on $[0,\delta]$ with slope $\varphi(\delta) /\delta> 0 $ 
and so, in particular, differentiable in 0.
\end{proof}

\begin{remark}
\label{remark:Criterion-for-differentiability}
Assume $I$ is a compact interval of the form $[0,b]$ with $b \in A_{>0}$
and the images of $\sigma_\ell$ and $\sigma_r$ are both not cyclic.
It follows then from Proposition \ref{prp:phi-is-linear-I-compact} 
that every \emph{increasing} automorphism $\alpha \colon G \iso G$
is induced by an auto-homeomorphism $\tilde{\varphi}$ 
that is affine near both end points. 
By \cite[Proposition E16.9]{BiSt14}  
the homeomorphism $\tilde{\varphi}$ is thus \emph{finitary} piecewise linear. 
\end{remark}

%------------
\subsubsection{First application}
\label{ssec:Differentiability-criterion-First-application}
%----------
%
As a further step towards the main results 
we give a corollary
that combines Propositions
\ref{prp:Importance-of-differentiability-I-compact}
and
\ref{prp:phi-is-linear-I-compact}.
\begin{crl}
\label{crl:I-compact-summary-differentiability}
Let $G$ be a subgroup of $G(I;A,P)$ that contains $B(I;A,P)$.
Assume $I = [0, b]$ and let $\alpha$ be an automorphism of $G$ 
that is induced  by the auto-homeomorphism $\tilde{\varphi}\colon I \iso I$.
Then the following statements hold: 
\begin{enumerate}[(i)]
\item if $\alpha$ is increasing
\footnote{See Definition \ref{definition:Increasing-isomorphism}.}
and $\im \sigma_\ell$ not cyclic,
then $\sigma_\ell$ is fixed by $\alpha$;
\item if $\alpha$ is increasing and $\im \sigma_r$ not cyclic,
then $\sigma_r$ is fixed by $\alpha$;
\item 
if $\tilde{\varphi}$ is decreasing and $\im \sigma_\ell$ is not cyclic,
then $\tilde{\varphi}$ is affine near both end points 
and the homomorphism $\psi \colon g \mapsto \sigma_\ell (g) \cdot \sigma_r(g)$ is fixed by $\alpha$.
\end{enumerate}
\end{crl}
\begin{proof}
(i) is a direct consequence of Proposition \ref{prp:phi-is-linear-I-compact}
and part (i) of Proposition \ref{prp:Importance-of-differentiability-I-compact}.

(ii) We invoke Proposition \ref{prp:phi-is-linear-I-compact} 
for an auxiliary group $G_1$.
Let $\vartheta \colon I \iso I$ be the reflection in the mid-point of $I$; 
set $G_1 = \vartheta \circ G \circ  \vartheta^{-1}$ 
and $\varphi_1 =  \vartheta \circ \tilde{\varphi}_\alpha \circ \vartheta^{-1}$.
Since $G(I;A,P)$ and  $B(I;A,P)$ are both invariant under conjugation by $\vartheta$,
and as the image of $\sigma_r$ is not cyclic,
Proposition \ref{prp:phi-is-linear-I-compact} applies to the couple $(G_1,\varphi_1)$  
and shows 
that $\varphi_1$ is linear in a small interval $[0, \delta_1]$ of positive length.
But if so $\varphi_\alpha$ is affine in the interval $[b-\delta_1, b]$.
Use now part (ii) in Proposition \ref{prp:Importance-of-differentiability-I-compact}.

(iii) Since $\tilde{\varphi}$ is \emph{decreasing},
the subgroups $\im \sigma_\ell$ and $ \im \sigma_r$ are isomorphic
by Lemma \ref{lem:Consequence-decreasing-auto-I-compact} below;
the hypothesis on $\im \sigma_\ell$ implies therefore
the image of $\sigma_r$ is not cyclic, either.
Let $\vartheta \colon  \Int(I) \iso \Int(I)$ be the reflection in the midpoint of the interval $I$
and set $\tilde{\varphi}_1 = \vartheta \circ \tilde{\varphi}$ 
and $\bar{G} = \vartheta \circ G \circ \vartheta^{-1}$.
Conjugation by $\tilde{\varphi}_1$ induces then an increasing isomorphism 
$\alpha_1 \colon G \iso \tilde{G}$.
Since $G(I;A,P)$ and $B(I;A,P)$ are both invariant under conjugation by $\vartheta$,
Proposition \ref{prp:phi-is-linear-I-compact} 
applies to $\tilde{\varphi}_1$ in the rôle of $\tilde{\varphi}$ 
and  shows that $\tilde{\varphi}_1$ is linear near $0$.
But if so,
$\tilde{\varphi}$ is linear near $0$.
Consider now the auto-homeomorphism 
$\tilde{\varphi}_2 = \tilde{\varphi} \circ  \vartheta$  of $I$.
It induces an isomorphism $\alpha_2 \colon \bar{G} \iso G$  by conjugation;
an argument similar to the preceding one then reveals 
that  $\tilde{\varphi}$ is affine near $b$.
The remainder of the claim follows from part (iii) in Proposition 
\ref{prp:Importance-of-differentiability-I-compact}.
\end{proof}
%
%----------
\subsection{Construction of homomorphisms fixed by $\Aut_+ G$}
\label{ssec:Construction-psi-fixed-by-Aut-plus(G)-I-compact}
%----------
The first main result holds for all groups $G$ with $B(I;A,P) \subsetneq G \subseteq G(I;A,P)$,
but the exhibited homomorphisms may only be fixed by $\Aut_+ G$.
\begin{theorem}
\label{thm:Existence-psi-I-compact-increasing-autos}
Suppose  $I = [0, b]$ with $b \in A_{>0}$ 
and let $G$ be a subgroup of $G(I;A,P)$ that contains $B(I;A,P)$ properly.
Then the homomorphisms $\sigma_\ell$ and $\sigma_r$ are fixed by $\Aut_+ G$,
and at least one of them is non-trivial.
\end{theorem}

\begin{proof}
Let $\alpha$ be an increasing automorphism of $G$ 
and  let $\tilde{\varphi}$ be the auto-homeo\-mor\-phism of $I$ 
that induces $\alpha$.
(The map $\tilde{\varphi}$ exists by Theorem \ref{thm:TheoremE16.4} 
and Remark \ref{remarks:varphi}a.)
Since the quotient group $G(I;A,P)/B(I;A,P)$ is isomorphic to 
the image of $ \sigma_\ell \times \sigma_r \colon G(I;A,P)\to P \times P$
and as $G$ contains $B(I;A,P)$ properly,
at least one of the homomorphisms $\sigma_\ell$ and $\sigma_r$ is non-zero.

Assume first that $\psi = \sigma_\ell$ is non zero.
Two cases then arise, 
depending on whether the image of $\psi$ is cyclic or not.
If the image of $\psi$ is \emph{not cyclic} 
then part (i) in Corollary \ref{crl:I-compact-summary-differentiability} shows 
that  $\alpha$ fixes $\psi$.
If, on the other hand,
$\psi$ is cyclic,
consider the generator $p \in \im \psi$ with $p < 1$
and pick a preimage $g_p \in G$ of $p$.
Then $g_p$ attracts points in every sufficiently small interval of the form $[0, \delta]$
towards 0;
hence so does $\alpha(g_p) = \tilde{\varphi} \circ g_p \circ \tilde{\varphi}^{-1}$ 
and thus $p' = (\alpha(g_p))'(0) <1$.
Now $p'$ generates also $\im \sigma_\ell$; being smaller than 1,
it coincides therefore with $ p= \psi(g_p)$ and so $\psi = \psi\circ \alpha$.

Assume next that $\psi = \sigma_r$ is not zero.
If its image is not cyclic
part (ii) of Corollary \ref{crl:I-compact-summary-differentiability}
alllows us to conclude that $\alpha$ fixes $\psi$.
If $\im \psi$ is cyclic, 
consider the generator $p \in \im \psi$ with $p < 1$ and pick a preimage $g_p \in G$.
Then $g_p$ attracts points in every sufficiently small interval $[b-\delta, b]$ towards $b$.
It then follows, as before, that $\psi (\alpha(g_p)) =p =  \psi(g_p)$,
whence $\psi \circ \alpha = \alpha$.
\end{proof}

%======
\subsection{Existence of decreasing automorphisms}
\label{ssec:Existence-decreasing-auto-I-compact)}
%======
%
Theorem \ref{thm:Existence-psi-I-compact-increasing-autos} is very satisfactory in 
that it produces a non-zero homomorphism $\psi$ onto an infinite abelian group 
whenever such a homomorphism is likely to exist, \ie{} if $G$ contains $B(I;A,P)$ properly.
This homomorphism is, however, only guaranteed to be fixed 
by the subgroup $\Aut_+ G$ of $\Aut G$ which has  index 1 or 2 in $\Aut G$.
If the index is 1,
the conclusion of Theorem \ref{thm:Existence-psi-I-compact-increasing-autos} is as good 
as we can hope for.
So the question arises whether there are useful criteria that force the index to be 1.
Here is a very simple observation that leads to such a criterion:
\begin{lem}
\label{lem:Consequence-decreasing-auto-I-compact}
Assume $I = [0, b]$ with $b \in A_{>0}$ and
let $G$  be a subgroup of $G(I;A,P)$ that contains $B(I;A,P)$.
Then every decreasing automorphism $\alpha$ induces an isomorphism 
$\alpha_* \colon \im \sigma_\ell \iso \im \sigma_r$.
\end{lem}

\begin{proof}
The kernel of $\sigma_\ell$ consists of all elements in $G$ that are the identity near 0.
Since $\alpha$ is induced by conjugation by a homeomorphism of $I$ that maps 0 onto $b$,
the image of $\ker \sigma_\ell$ consists  of elements that are the identity near $b$,
and so $\alpha(\ker \sigma_\ell) \subseteq  \ker \sigma_r$.
Since $\alpha^{-1}$ is also a decreasing automorphism, 
the preceding inclusion is actually an equality.
So $\alpha$ induces an isomorphism $\alpha_* \colon \im \sigma_\ell \iso \im \sigma_r$ 
that renders the square
\begin{equation}
\label{eq:square-minus-plus}
\xymatrix{G \ar@{->}[r]^-{\alpha} \ar@{->>}[d]^-{\sigma_\ell} & G \ar@{->>}[d]^-{\sigma_r}\\
\im \sigma_\ell\ar@{->}[r]^-{\alpha_{*}} &\im \sigma_r}
\end{equation}
commutative.
\end{proof}

\begin{example}
\label{example:Existence-decreasing-auto-I-compact}
Suppose the slope group $P$ is finitely generated 
and hence free abelian of finite rank $r$, say.
Choose subgroups $Q_\ell$ and $Q_r$ of $P$ and set
\begin{equation}
\label{eq:Definition-G-Qell-Qr}
G(Q_\ell, Q_r) = \{g \in G(I;A,P) \mid (\sigma_\ell(g), \sigma_r(g) \in Q_\ell \times Q_r \}.
\end{equation}
Then $\im \sigma_\ell = Q_\ell$ and $\im \sigma_r = Q_r$,
and 
the image of $(\sigma_\ell, \sigma_r) \colon G \to P \times P$ coincides with $Q_\ell \times Q_r$
(these claims follow from Corollary A5.5 in \cite{BiSt14}).

Assume now that $G(Q_\ell, Q_r)$ admits a decreasing automorphism,
say $\alpha$.
By Lemma \ref{lem:Consequence-decreasing-auto-I-compact} 
the groups $Q_\ell$ and $Q_r$ are then isomorphic, and thus have the same rank.
But more is true:
if $Q_\ell = \im \sigma_\ell$ is \emph{not} cyclic,
then part (iii) of Corollary \ref{crl:I-compact-summary-differentiability} applies 
and shows that $\sigma_r = \sigma_\ell \circ \alpha$,
whence $Q_r$, the image of $\sigma_r$, 
coincides with $Q_\ell$, the image of $\sigma_\ell$.
The same conclusion holds if  $Q_r$ is not cyclic.
Conversely, 
if $Q_\ell = Q_r$
then $G(Q_\ell, Q_r)$ admits  decreasing automorphisms,
for instance the automorphism induced by conjugation by the reflection about the mid point of $I$.

So the only case 
where the existence of a decreasing automorphism is neither obvious 
nor easy to rule out by the preceding arguments
is that where $Q_\ell$ and $Q_r$ are both cyclic, but distinct.
We shall come back to this exceptional case in Example \ref{Example2-for-main-result-I-compact}.
\end{example}
%
%----------
\subsection{Construction of a homomorphism fixed by $\Aut G$}
\label{ssec:Construction-psi-fixed-by-Aut}
%----------
We move on to the construction of a homomorphism fixed by all of $\Aut G$.
The following result is our main result.
\begin{theorem}
\label{thm:Existence-psi-I-compact-all-autos}
Suppose $I$ is a compact interval of the form $[0,b]$ with $b \in A_{>0}$.
Let $G$ be a subgroup of $G(I;A,P)$ containing $B(I;A,P)$
and let $\psi \colon G \to P$ be the homomorphism $g \mapsto \sigma_\ell(g) \cdot \sigma_r(g)$.
Then $\psi$ is fixed by $\Aut G$,  
\emph{except possibly} when $G$ satisfies the following three conditions:
\begin{enumerate}[a)]
\item $\im (\sigma_\ell \colon G \to P)$ is cyclic,
\item $G$ admits a decreasing automorphism, 
\item $G$ does not admit a decreasing automorphism induced 
by an auto-homeomorphism $\vartheta \colon I \iso I$ 
that is differentiable  in both end points with non-zero values.
\end{enumerate}
\end{theorem}

\begin{proof}
Let $\alpha$ be an automorphism of $G$ 
and let $\varphi$ be the auto-homeomorphism of $\Int(I)$ 
that induces $\alpha$ by conjugation.
If $\varphi$ is \emph{increasing} both $\sigma_\ell$ and $\sigma_r$ are fixed by $\alpha$
(see Theorem \ref{thm:Existence-psi-I-compact-increasing-autos}) 
and hence so is $\psi$.
If, on the other hand, $\alpha$ is \emph{decreasing} 
and the image of $\sigma_\ell$ is \emph{not cyclic} 
then part (iii) of Corollary \ref{crl:I-compact-summary-differentiability}
yields the desired conclusion.

Suppose now that $G$ admits an automorphism $\beta$ 
that is induced by a decreasing auto-homeomorphism $\tilde{\varphi}_\beta$ of $I$
that is differentiable in 0, as well as in $b$, and has there non-zero derivatives.
Then part (iii) of Proposition \ref{prp:Importance-of-differentiability-I-compact} 
allows us to conclude that $\psi$ is fixed by $\beta$.
Since $\beta$ represents the coset $\Aut G \smallsetminus \Aut_+ G$
and as $\psi$ is fixed by $\Aut_+ G$, 
it follows that $\psi$ is fixed by every decreasing automorphism.

All taken together we have proved 
that the automorphism $\alpha$ fixes $\psi$
except, possibly,
if $\im \sigma_\ell$ is cyclic, $\alpha$ is decreasing and 
if there does not exists a decreasing automorphism $\beta$
that is differentiable in the end points and has there non-zero derivatives.
\end{proof}

We state next some consequences of Theorems 
\ref{thm:Existence-psi-I-compact-increasing-autos}
and 
\ref{thm:Existence-psi-I-compact-all-autos}.
We begin with the special case where $G$ is all of $G(I;A,P)$.
Then $G$ is normalized by the reflection in the mid-point of $I$ 
and so Theorem \ref{thm:Existence-psi-I-compact-all-autos} leads to
\begin{crl}
\label{crl:G=G(I;A,P)}
If $G$ coincides with $G([0,b];A,P)$ the homomorphism $\psi \colon G \to P$ 
taking $g \in G$ to $\sigma_\ell(g) \cdot \sigma_r(g)$ is surjective, hence non-zero,
and fixed by $\Aut G$.  
\end{crl}

The second result is a consequence of the proof of Theorem 
\ref{thm:Existence-psi-I-compact-increasing-autos}.
\begin{crl}
\label{crl:sigma-ell-fixed-y-AutG-for-G(half-line;A,P)}
Suppose $I$ is the half line $[0, \infty[$
and $G$ is a subgroup of $G(I;A,P)$ containing $B(I; A,P)$.
If $G$ does not admit a decreasing automorphism 
 then $\psi = \sigma_\ell$ is fixed by $\Aut G$.  
\end{crl}

\begin{proof}
The claim follows from Proposition \ref{prp:phi-is-linear-I-compact}
and from the proof of part (i) in Proposition \ref{prp:Importance-of-differentiability-I-compact}
upon noting that the cited proof does not presuppose that the interval $I$ be bounded from above.
\end{proof}
%
%----------
\subsection{Some examples}
\label{ssec:Some-examples-I-compact}
%----------
We exhibit some specimens of groups $G$ 
that possess a homomorphism $\psi \colon G \to P$ fixed by $\Aut G$.
The existence of $\psi$ will be established by recourse to 
Theorems \ref{thm:Existence-psi-I-compact-increasing-autos} 
and \ref{thm:Existence-psi-I-compact-all-autos}
and to Corollary \ref{crl:G=G(I;A,P)}.
%
%----------
\begin{example}
\label{Example1-for-main-result-I-compact}
%----------
We begin with variations on Thompson's group $F$.
Assume $P$ is infinite cyclic and $A$ is a (non-trivial) $\Z[P]$-submodule of $\R$. 
Set $G_0 = G([0,b]; A, P)$ with $b \in A_{>0}$
and consider the following subgroups of $G_0$:
\begin{align}
G_1 &= \{g \in G_0 \mid \sigma_\ell(g) = 1 \},
\label{eq:Subgroup1-with-infinite-cyclic-quotient}\\
G_2 &=\{g \in G_0 \mid \sigma_\ell(g) = \sigma_r(g) \},
\label{eq:Subgroup2-with-infinite-cyclic-quotient}\\
G_3 &=\{g \in G_0 \mid \sigma_\ell(g) = \sigma_r(g)^{-1} \}.
\label{eq:Subgroup3-with-infinite-cyclic-quotient}
\end{align}

The group $G_0$ is the entire group $G(I;A,P)$ 
and so Corollary \ref{crl:G=G(I;A,P)} tells us 
that the homomorphism $\psi \colon g \mapsto \sigma_\ell(g) \cdot \sigma_r(g)$ 
is non-zero and and fixed by $\Aut G_0$.

The group $G_1$ is an ascending union of subgroups $H_n = G([a_n, b];A,P)$
given by a strictly decreasing sequence $n \mapsto a_n$ of elements in $A$ 
that converges to 0,
and so the group $G_1$ is infinitely generated.
It does not admit a decreasing automorphism
(for instance because of Lemma \ref{lem:Consequence-decreasing-auto-I-compact})
and so Theorem \ref{thm:Existence-psi-I-compact-increasing-autos} allows us to infer
that the epimorphism $\sigma_r \colon G_1 \epi P$ is fixed by all of $\Aut G_1$.

The group $G_2$ is an ascending HNN-extension with a base group 
that is isomorphic to $G_0$ (see \cite[Lemma E18.8]{BiSt14}).
If $G$ is finitely generated or finitely presented,
so is therefore $G_2$. 
The group is normalized by the reflection in the mid-point of $I$
and so Theorem \ref{thm:Existence-psi-I-compact-all-autos} implies
that $\psi \colon g \mapsto \sigma_\ell(g) \cdot \sigma_r(g)$ is fixed by $\Aut G_2$.
This homomorphism $\psi$ is non-zero, for it coincides with $\sigma_\ell^2$. 
(Actually, $\sigma_\ell$ and $\sigma_r$ are also fixed by $\Aut G_2$.)

Now to the group $G_3$.
It differs from $G_2$ in several respects:
it cannot be written as an ascending HNN-extension 
with a finitely generated base group contained in $B(I;A,P)$;
it is finitely generated if $G_0$ is so,
but, if finitely generated, it does not admit a finite presentation
(see part (ii) of Lemma E18.8 and Remark E18.10 in \cite{BiSt14}).
The group $G_3$ is normalized by the reflection in the mid point of $I$
and so $\psi \colon G_3 \to P$ is fixed by $\Aut G_3$;
this conclusion, however, is of no interest as $\psi$ is the zero map.
Actually, more is true: 
every homomorphism $\psi' \colon G_3 \to P$ fixed by $\rho$ and vanishing on the bounded subgroup
$B_3$ of $G_3$ is the zero-map: 
by definition \eqref{eq:Subgroup3-with-infinite-cyclic-quotient}
the group $G_3/B_3$ is infinite cyclic 
and so $\psi'$ must be a multiple of $\sigma_\ell$.
\end{example}

\begin{remark}
\label{remark-by-Paramesh}
The previous discussion shows 
that $G_0$, $G_1$ and $G_2$ admit non-trivial homomorphisms into $P$ 
that are fixed by the corresponding automorphism groups.  
This fact and the observation made in section \ref{ssec:Crucial-fact} imply
that every automorphism of one of these groups 
has infinitely many corresponding twisted conjugacy classes.
This reasoning does not hold for  $G_3$,
for $\psi \colon G_3 \to P$ is the zero homomorphism.

So the question whether or not an automorphism $\alpha$ of $G_3$ has infinitely many twisted conjugacy classes has to be tackled by another approach.
Note first that the homomorphisms $\sigma_\ell$ and $\sigma_r$ are both non-zero;
as $G_3$ satisfies the assumptions of Theorem 
\ref{thm:Existence-psi-I-compact-increasing-autos}
these homomorphisms are therefore fixed by $\Aut_+G_3$.
It follows that every increasing automorphism $\alpha$ of $G_3$
has infinitely many $\alpha$-twisted conjugacy classes.
We are thus left with the coset of decreasing automorphisms of $G_3$.

Consider, for example, the automorphism $\beta$ induced by conjugation 
by the reflection $\vartheta$ in the mid-point of the interval $I$. 
Our aim is to construct an infinite collection of elements $g_n \in G_3$
and to verify then 
that they represent pairwise distinct $\beta$-twisted conjugacy classes.
This verification will be based on the fact $\beta$ has order $2$
and a connection between twisted and ordinary conjugacy classes,
available for automorphisms of finite order.
\footnote{Cf.\;Lemma 2.3 in \cite{GoSa14}.}

Let $f$ and $g$ be elements of $G_3$ 
that lie in the same $\beta$-twisted conjugacy class.
By definition, there exists then $h \in G_3$
that satisfies the equation $g = h \circ f \circ \beta(h^{-1})$. 
The calculation
\begin{align*}
g \circ \beta(g) 
&= 
\left(h \circ f \circ \beta(h^{-1})\right) \circ \beta\left( h \circ f \circ \beta(h^{-1}) \right)\\
&=
h \circ (f \circ \beta(f) ) \circ \beta^2(h^{-1}) = \act{h}{1} {(f \circ \beta(f) )}
\end{align*}
shows then
that the elements $f \circ \beta(f)$ and  $g \circ \beta(g)$ are conjugate.
It suffices therefore to find a sequence of elements $n \mapsto f_n$ 
with the property
that the compositions $f_{n_1} \circ \beta(f_{n_1})$ and $f_{n_2} \circ \beta(f_{n_2})$ 
represent distinct conjugacy classes whenever $n_1 \neq n_2$.

To obtain such a sequence,
we use the fact that $G$ contains $B(I;A,P)$ 
and that $B(I;A,P)$ consists of all PL-homeomorphisms with slopes in $P$,
breakpoints in the dense subgroup $A$,
and which are the identity near the end points. 
For every positive integer $n$
there exists therefore a non-trivial element $f_n \in B(I;A,P)$ 
whose support has $n$ connected components, all contained in the interval $]0, b/2[$.
Then $h_n = f_n \circ \beta(f_n) = f_n \circ (\vartheta \circ f_n \circ \vartheta^{-1})$
has $2n$ connected components,
and so $h_{n_1}$ is not conjugate to $h_{n_2}$ for $n_1 \neq n_2$.
It follows that $G_3$ has infinitely many $\beta$-twisted conjugacy classes.

The previous reasoning allows of some improvements,
but it does not seem powerful enough to establish that $G_3$ has infinitely many 
$\alpha$-twisted conjugacy classes for every decreasing automorphism $\alpha$ of $G_3$. 
\end{remark}

%----------
\begin{example}
\label{Example2-for-main-result-I-compact}
%----------
%
Example \ref{Example1-for-main-result-I-compact} admits a generalization 
that is worth being brought to attention of the reader.
Assume $P$ is a non-trivial subgroup of the positive reals,
$A$ is a (non-trivial) $P$-submodule of $\R$ 
and $\nu$ is an endomorphism of $P$.
Fix $b \in A_{>0}$, set $I = [0, b]$ and define
\begin{equation}
\label{eq:definition-G(nu)}
G_\nu = \{g \in G([0,b];A,P) \mid \sigma_r(g) = \nu(\sigma_\ell(g)) \}.
\end{equation}
We are interested in finding a non-zero homomorphism $\psi \colon G_\nu \to P$
that is fixed by $\Aut G_\nu$.
Theorem \ref{thm:Existence-psi-I-compact-all-autos} implies
that the homomorphism $\psi \colon g \mapsto \sigma_\ell (g) \cdot \sigma_r(g)$ is fixed  
by $\Aut G_\nu$ whenever $P$ is not cyclic;
this homomorphism is non-zero unless $\nu$ is the map 
that sends $p \in P$ to its inverse $p^{-1}$.

\emph{Assume now that $P$ is cyclic}.
Then $G_\nu$ is isomorphic to one of the groups $G_1$, $G_2$ or $G_3$ 
discussed in Example \ref{Example1-for-main-result-I-compact}.
This claim is clear if $\nu$ is the zero map,
for $G_\nu$ coincides then with $\ker \sigma_r$ and is therefore  isomorphic to $G_1$.
Assume now that $\nu$ is not zero.
The quotient group $G_\nu/ B(I;A,P)$  is then an infinite cyclic subgroup 
of the quotient group $G(I;A,P)/B(I;A, P)$
which is free abelian group of rank 2.
By the classification in Section 18.4b of \cite{BiSt14},
the group $G_\nu$ is therefore isomorphic, either to $G_2$,
or to $G_3$. 
Since the isomorphism $G_\nu \iso G_2$, respectively $G_\nu \iso G_3$,
is induced by conjugation by an auto-homeomorphism of $]0,b[$
and as conjugation by the reflection in $b/2$ 
induces decreasing automorphisms in $G_2$ and in $G_3$,
\emph{the group $G_\nu$ admits  a decreasing automorphism, say $\beta$};
it induces an isomorphism $\beta_{*} \colon \im \sigma_\ell \iso \im \sigma_r$
(see Lemma \ref{lem:Consequence-decreasing-auto-I-compact}). 
Our next aim is to obtain a formula for $\beta_*$.

The definition of $G_\nu$ shows, first of all,
that $\im \sigma_\ell = P$ and that $\im \sigma_r = \nu(P)$.
Let $p$ be the generator of $P$ with $p < 1$.
Then $\nu(p) = p^m$ for some non-zero integer $m$ (recall that $\nu$ is not the zero map).
Pick an element $g_p \in G_\nu$ with $\sigma_\ell(g_p) = p$.
Then 0 is an attracting fixed point of $g_p$ 
restricted to a sufficiently small interval of the form $[0, \delta]$,
and hence $b$ is an attracting fixed point for the restriction of $\beta(g_p)$ to a sufficiently small interval of the form $[b-\varepsilon,b]$.
Thus $\beta(g_p) < 1$.
Since $\beta(g_p)$ generates $\im \sigma_r = \nu(P) = \gp(p^m)$
it follows that $\beta_*$ is given by the formula
\begin{equation}
\label{eq:Identification-beta-star}
\beta_* \colon P \to P, \quad p \mapsto p^{|m|}.
\end{equation}

Consider now the commutative square \eqref{eq:square-minus-plus},
but with $\alpha$ replaced by $\beta$.
It shows that
\begin{equation}
\label{eq:Relation-beta-digma-ell-sigma-r}
(\sigma_r \circ \beta)(g_p) 
=
\beta_*(\sigma_\ell(g_p)) = \beta_*(p) =  p^{|m|} = \left(\sigma_\ell(g_p)\right)^{|m|}
\end{equation}
and so $\sigma_r \circ \beta = \sigma_\ell^{|m|}$.
The preceding reasoning is also valid with $\beta^{-1}$ in place of  $\beta$,
for $\beta^{-1}$ is also a decreasing automorphism of $G_\nu$,
and so the relation $\sigma_r \circ \beta^{-1} = (\sigma_\ell)^{|m|}$ holds,
hence also the relation $\sigma_\ell^{|m|} \circ \beta = \sigma_r$.

Consider next the homomorphism $\psi \colon G_\nu \to P$
that takes $g$ to $\sigma_\ell(g)^{|m|} \cdot  \sigma_r(g)$.
The calculation
\[
(\psi \circ \beta)(g) 
=
\sigma_\ell^{|m|} (\beta (g)) \cdot \sigma_r(\beta(g))
=
\sigma_r(g) \cdot \sigma_\ell^{|m|}(g) = \psi(g)
\]
shows then that $\psi$ is fixed by $\beta$. 
Note, however, that $\psi$ is the zero homomorphism whenever $m$ is negative,
for in this case the definition of $G_\nu$ implies 
that
\[
\psi(g) 
= 
(\sigma_\ell(g))^{|m|} \cdot \sigma_r(g)
=
(\sigma_\ell(g))^{|m|} \cdot (\sigma_\ell(g))^m = 1
\]
for every $g \in G_\nu$, just as it happens with $G_3$ 
in Example \ref{Example1-for-main-result-I-compact}.

\begin{remark}
\label{remark:Exceptions-in-thm-Existence-psi-I-compact-all-autos}
Suppose $P$ is cyclic and $\nu \colon P \to P$ is neither the identity 
nor the passage to the inverse.
Then $G_\nu$ admits decreasing automorphisms  $\beta$, 
but none of them can be induced by an auto-homeomorphism $\tilde{\varphi} \colon I \iso I$
that is differentiable in the end points;
indeed, 
formula \ref{eq:Relation-beta-digma-ell-sigma-r} shows 
that $\sigma_r \circ \beta \neq \sigma_\ell$,
in contrast to what happens if the chain rule can be applied
(see Proposition \ref{prp:Importance-of-differentiability-I-compact}).
It follows, 
in particular,
that the three conditions a), b) and c) stated in Theorem
\ref{thm:Existence-psi-I-compact-all-autos}
can occur simultaneously. 
\end{remark}
\end{example}

%========
%
\section{Characters fixed by $\Aut G([0,\infty[\,;A,P)$}
\label{sec:Homomorphisms-fixed-by-Aut-I-half-line}
%
%========
%
The results in this section differ
from those of Section \ref{sec:Homomorphisms-fixed-by-Aut-I-compact} 
in two important respects:
in many situations several candidates for $\psi \colon G \to P$ are available
and one of these candidates need not fixed by $\Aut G_+$.
%
%-----------------
\subsection{Existence of decreasing automorphisms}
\label{ssec:I-half-line-Existence-decreasing-autos}
%-----------------
%
Every compact interval of the form $[0,b]$, and also the line, is invariant under a reflection.
It follows that the groups $G(I;A,P)$ with $I$ one of these intervals,
but also many of their subgroups, admit decreasing automorphisms.
The case where $I$ is a half line, say $[0, \infty[$,  is different:
then $G([0,\infty[\,;A,P)$ does not admit a decreasing automorphism.

In this section, 
we first justify this claim 
and discuss then the extent to which it continues to be valid for subgroups of $G([0,\infty[\,;A,P)$.
We begin with an analogue of Lemma \ref{lem:Consequence-decreasing-auto-I-compact}.
\begin{lem}
\label{lem:Consequence-decreasing-auto-I-half-line}
Assume $I$ is the half line $[0, \infty[$ and
$G$  is a subgroup of $G(I;A,P)$ that contains $B(I;A,P)$.
Then every decreasing automorphism $\alpha$ induces an isomorphism 
$\alpha_* \colon \im \sigma_\ell \iso \im \rho $.
\end{lem}

\begin{proof}
The proof is very similar to that of  Lemma \ref{lem:Consequence-decreasing-auto-I-compact}.
The kernel of $\sigma_\ell$ consists of all elements in $G$ that are the identity near 0,
while that of $\rho$ is made up of the elements in $G$ that are the identity near $\infty$.
Since $\alpha$ is induced by conjugation by a decreasing homeomorphism of $]0, \infty[$,
the image of $\ker \sigma_\ell$ consists  of elements $\alpha(g)$
that are the identity on a half line of the form $[t(g), \infty[$,
and so $\alpha(\ker \sigma_\ell) \subseteq  \ker \rho$.
Since $\alpha^{-1}$ is also a decreasing automorphism, 
the preceding inclusion is actually an equality.
It follows that $\alpha$ induces an isomorphism $\alpha_* \colon \im \sigma_\ell \iso \im \rho$ 
that renders the square
\begin{equation}
\label{eq:square-sigma-ell-rho}
\xymatrix{G \ar@{->}[r]^-{\alpha} \ar@{->>}[d]^-{\sigma_\ell} & G \ar@{->>}[d]^-{\rho}\\
\im \sigma_\ell\ar@{->}[r]^-{\alpha_{*}} &\im \rho}
\end{equation}
commutative.
\end{proof}

The preceding lemma leads directly to a criterion 
for the non-existence of decreasing automorphisms.
Indeed,
the image of $\sigma_\ell$ is abelian,
while that of $\rho$ is often a non-abelian,  metabelian group,
and so we obtain

\begin{criterion}
\label{criterion:Non-existence-decreasing-auto-I-half-line}
Assume $I$ is the half line $[0, \infty[$ and
$G$ is a subgroup of $G(I;A,P)$ that contains $B(I;A,P)$.
If $\im \rho$ is \emph{not} abelian then $\Aut G = \Aut_+ G$.
\end{criterion}
%
%-----------------
\subsection{Construction of homomorphisms: part I}
\label{ssec:I-half-line-Construction-homomorphisms-I}
%-----------------
%
We turn now to the construction of homomorphisms 
fixed by $\Aut_+ G$, or even by $\Aut G$.
Several homomorphisms are at our disposal.
The first of them is $\sigma_\ell$.
Corollary 
\ref{crl:sigma-ell-fixed-y-AutG-for-G(half-line;A,P)}
tells us then:
\begin{prp}
\label{prp:Existence-psi-I-halfline-increasing-autos-1}
Assume $G$ is a subgroup of $G([0,\infty[\,;A,P)$ 
that contains $B(I;A,P)$.
Then the homomorphisms $\sigma_\ell$ is fixed by $\Aut_+ G$.
\end{prp}

We move on to the homomorphism $\rho$.
Here two cases arise, depending on whether its image is abelian or non-abelian.
In the second case, a very satisfying conclusion holds.
It is enunciated in 
\begin{theorem}
\label{thm:Existence-psi-I-halfline-image-rho-non-abelian}
Assume $I = [0, \infty[$ and $G$ is a subgroup of $G(I;A,P)$ containing $B(I;A,P)$.
If $\im \rho$ is not abelian 
$\sigma_r$ is a non-zero homomorphism fixed by $\Aut G$.
\end{theorem}
\begin{proof}
Suppose $\im \rho$ is non-abelian.
Then Lemma \ref{lem:Consequence-decreasing-auto-I-half-line} forces $\alpha$ to be increasing.
Let $\varphi \colon ]0, \infty[\, \iso \, ]0, \infty[$  be the auto-homeomorphism
that induces $\alpha$ by conjugation.
As it is increasing, 
it is affine near $\infty$ by Proposition \ref{prp:Affine-near-infty-I-half-line} below,
and so the following calculation
\begin{align*}
\sigma_r(\alpha(g)) 
&=
\lim\nolimits_{t \to \infty} \left(\varphi \circ g \circ \varphi^{-1}\right)'(t)\\
&=
\lim\nolimits_{t \to \infty} 
\left(
\varphi'(g \circ \varphi^{-1}(t)) \cdot g'(\varphi^{-1}(t)) \cdot (\varphi^{-1})'(t) 
\right)\\
&=
\lim\nolimits_{t \to \infty} 
\left(
\varphi'(t) \cdot g'(t) \cdot (\varphi^{-1})'(t) 
\right)\\
&=
\lim\nolimits_{t \to \infty} g'(t) = \sigma_r(g)
\end{align*}
is valid  for every $g \in G$.
It shows
that $\alpha$ fixes the homomorphism $\sigma_r$.
This homomorphism is non-zero. 
Indeed, 
$G/\ker \rho \iso \im \rho$ is not abelian by hypothesis,
while $ \ker \sigma_r/\ker \rho$ is abelian and thus the third term of the extension
\[
\ker \sigma_r/\ker \rho \mono G/ \ker \rho \epi G/\ker \sigma_r 
\]
is not zero, whence $\ker \sigma_r \neq G$.
\end{proof}

We are left with proving an analogue of Proposition \ref{prp:phi-is-linear-I-compact}.
For later use, we state it in greater generality than needed at this point,
namely as
\begin{prp}
\label{prp:Affine-near-infty-I-half-line}
Assume $G$ and $\bar{G}$ are subgroups of $G(I;A,P)$, 
both containing the subgroup $B(I;A,P)$,
and that $I$ is either the half line $[0, \infty[$ or the line $\R$.
Let $\alpha \colon G \iso \bar{G}$ be an isomorphism
and let $\varphi_\alpha$ be an auto-homeomorphism of $\Int(I)$ 
that induces $\alpha$ by conjugation.

If $\im \rho$ is not abelian and $\varphi_\alpha$ is increasing 
then $\varphi_\alpha$ is affine near $\infty$.
\end{prp}

\begin{proof}
We adapt the argument of Part 2 in the proof of \cite[Supplement E17.3]{BiSt14}
to the case at hand.
By assumption,
the image of $\rho_* \colon G \to \Aff(A,P) \iso A \rtimes P$ is not abelian;
its derived group is therefore (isomorphic to) a non-trivial submodule $A_1$ of $A$
which, being non-trivial,  contains arbitrary small positive elements
and so is dense in $\R$.
Let $\bar{\rho}_* \colon \bar{G} \to A \rtimes P$ be the similarly defined homomorphism;
the derived group of its image is then isomorphic to a non-trivial submodule $\bar{A}_1$ of $A$.

By part (ii) of Corollary \ref{crl:alpha-and-ker-lambda},
the isomorphism $\alpha $ induces an isomorphism 
$\alpha_*$ of  $G/\ker \rho$ onto $\bar{G}/\ker \bar{\rho}$;
hence an isomorphism of $\im \rho$ onto $\im \bar{\rho}$,
and, finally, an isomorphism $\alpha_1$ of $A_1$ onto $\bar{A}_1$.
They render commutative the following diagram
\begin{equation}
\label{eq:square-rho}
\xymatrix{
G \ar@{->>}[r]^-{\rho} \ar@{->}[d]^-{\alpha} 
& 
\im \rho \ar@{<-<}[r]^-{} \ar@{->}[d]^-{\alpha_*} 
&
A_1  \ar@{->}[d]^-{\alpha_1} 
\\%
\bar{G} \ar@{->>}[r]^-{\bar{\rho} }
& 
\im \bar{\rho} \ar@{<-<}[r]^-{}
&
\bar{A}_1 .
}
\end{equation}
We claim the automorphism $\alpha_1 \colon A_1 \iso \bar{A}_1$ is strictly \emph{increasing}.

Let $b \in A_1$ be an arbitrary positive element 
and  let $f_b \in G$ be a PL-homoe\-mor\-phisms 
that is a translation with amplitude $b$ near $\infty$,
say on $[t_{b,1}, \infty[$.
Then $\alpha(f_b)$ is a PL-homeomorphism 
which is a translation with amplitude $\alpha_1(b)$ near $\infty$, 
say for $t \geq \varphi(t_{b,2})$.
Since $\alpha$ is induced by conjugation by $\varphi$,
one has $\alpha(f_b) = \act{\varphi}{1}{ f_b}$; 
so $\varphi \circ f_b = \alpha(f_b) \circ  \varphi$.
By evaluating this equality at $t \geq \max\{t_{b,1}, t_{b, 2}\}$ 
one obtains the chain of equations
\[
\varphi(t + b) = (\varphi \circ f_b)(t) = (\alpha(f_b) \circ \varphi)(t) = \alpha_1(b) + \varphi(t).
\]
It implies that $\alpha_1(b)$ is positive, 
for $b$ is so by assumption and $\varphi$ is increasing.

We show next
that $\alpha_1$ is given by multiplication by a positive real number $s_1$.
As stated in the first paragraph of the proof,
$A_1$ is a dense subgroup of $\R_{\add}$.
Since $\alpha_1$ is strictly increasing  
it extends to a (unique) strictly increasing automorphism  
$\tilde{\alpha}_1 \colon \R \iso \R$.
This automorphism is continuous and hence an $\R$-linear map,
given by multiplication by some positive real number $s_1$.

We come now to the final stage of the analysis of $\varphi$.
In it we show 
that the restriction of $\varphi$ to a suitable interval of the form $[t_*, \infty[$ is \emph{affine}.
Choose a positive element $b_* \in A_1$ 
and let $f_{b_*} \in G$ be an element 
whose image under $\rho$ is a translation with amplitude $b_*$.
It then follows, as before,
that there is a positive number $t_*$ 
so that the equation
\begin{equation}
\label{eq;Describing-varphi}
\varphi(t + b_*) = \alpha_1(b_*) + \varphi(t) =  \varphi(t) + s_1 \cdot b_*
\end{equation} 
holds for every $t \geq t_*$.
Consider now an arbitrary positive element $b \in A_1$.
There exists then a positive number $t_{b}$ 
such that the calculation
\[
\varphi(t + b) = \alpha_1(b) + \varphi(t) = \varphi(t) +  s_1 \cdot b  
\]
is valid for $t \geq t_{b}$.
Choose a positive integer $m$ which is so large 
that $t_{b} \leq t_* + m \cdot b_*$.
For every $t \geq t_*$ the following calculation is then valid:
\begin{align*}
 \varphi(t + b) +s_1 \cdot m  b_*
&=
\varphi( t +b +  m \cdot b_* )\\
&=
\varphi( t +  m \cdot b_*) + s_1 \cdot b
= 
\varphi(t) + s_1 \cdot m b_* +s_1 \cdot b.
\end{align*}
It follows, in particular,
that the equation
\begin{equation}
\label{eq:Representation-.varphi-bis}
\varphi (t_* + b) = \varphi (t_*) + s_1 \cdot b
\end{equation}
holds for every positive element $b \in A_1$ and $t \geq t_*$.
Since $\varphi$ is continuous and increasing and as $A_1$ is dense in $\R$,
this equation allows us to deduce 
that $\varphi$ is affine with slope $s_1$ on the half line $[t_*, \infty[$,
and so the proof is complete.
\end{proof}

%\begin{remark}
%\label{remark:lem-Existence-psi-I-halfline-increasing-autos}
%%
%Replace half-line by line
%\end{remark}

The hypotheses of the Theorem \ref{thm:Existence-psi-I-halfline-image-rho-non-abelian} are satisfied
if $G = G([0,\infty[;A,P)$;
the theorem, Lemma \ref{lem:Consequence-decreasing-auto-I-half-line} 
and Corollary \ref{crl:sigma-ell-fixed-y-AutG-for-G(half-line;A,P)}
thus yield the pleasant
\begin{crl}
\label{crl:Existence-psi-I-halfline-fixed-by-AutG}
If $G$ coincides with $G([0,\infty[\,;A,P)$
 then both $\sigma_\ell \colon G \to P$
 and $\sigma_r \colon G \to P$ are surjective homomorphisms fixed by $\Aut G$.
\end{crl}

Corollary \ref{crl:Existence-psi-I-halfline-fixed-by-AutG}
 is the analogue of Corollary \ref{crl:G=G(I;A,P)},
but with the compact interval $I$ replaced by a half line.
Groups $G(I;A, P)$ with $I$ a half line have, so far, been investigated less often
than groups with $I$ a compact interval;
they have, however, their own merits, 
in particular the following one:
to date,
finitely generated groups of the form $G(I;A,P)$ with $I$ compact 
are only known for very special choices of the parameters $(A, P)$.
\footnote{See \cite[p.\;vii]{BiSt14} for the list of the groups known at the end of 2014.}
By contrast,
finitely generated groups with $I$ a half line are far more common,
as is shown by the following characterization:
\begin{prp}[Theorem B8.2 in \cite{BiSt14}]
\label{prp:Characterization-G-fg-I-half-line}
The group $G([0,\infty[\,;A,P)$ is finitely generated if, and only if, 
the following conditions are satisfied:
\begin{enumerate}[(i)]
\item $P$ is finitely generated,
\item $A$ is a finitely generated $\Z[P]$-module, and 
\item $A/(IP \cdot A)$ is finite.
\end{enumerate}
\end{prp}
%
%-----------------
\subsection{Construction of homomorphisms: part II}
\label{ssec:I-half-line-Construction-homomorphisms-II}
%-----------------
Theorem \ref{thm:Existence-psi-I-halfline-image-rho-non-abelian} is very pleasing:
it shows that the homomorphism $\sigma_r$ is fixed by all automorphisms
provided merely \emph{the image of $\rho \colon G \to \Aff(IP \cdot A, P)$ is not abelian}.
In this section, we discuss the remaining case.

The image of $G([0,\infty[\, ; A, P)$ under $\rho$ is the affine group 
\[
\Aff(IP \cdot A,P) \iso (IP \cdot A) \rtimes P
\]
(see section \ref{ssec:lambda-rho}).
This group is metabelian and contains two obvious kinds of abelian subgroups:
those made up of translations, corresponding to the subgroups of $IP \cdot A$,
and the subgroups consisting of homotheties $t \mapsto q\cdot t$
with ratio $q$ varying in a subgroup $Q$ of $P$.
We begin by discussing the second type of abelian subgroups.

%--------
\subsubsection{Image of $\rho$ is made up of homotheties}
\label{sssec:I-half-line-Groups-with-im-rho-homotheties}
%-------
%
Given a subgroup $Q$ of $P$ let $G_{Q}$ be the the subgroup of $G =G([0, \infty[\,;A,P)$ 
consisting of the products $f \circ g$ with $g \in B = B([0 ,\infty[\,;A,P)$ 
and $f$ a homothety $t \mapsto q \cdot t$ with $q \in Q$;
since $B$ is normal in $G$ the set so defined is actually a subgroup of $G$.
We do not know which of these subgroups $G_Q$ admit decreasing automorphisms,
but those with $Q$ cyclic have this peculiarity,
as can be seen from
\begin{lem}
\label{lem:I-half-line-Construction-decreasing-autos}
Assume $I$ is the half line $[0, \infty[$ and $Q$ is a cyclic subgroup of $P$.
Then the subgroup
\begin{equation}
\label{eq:Definition-subgroup-G-sub-Q}
G_Q = \{ f \circ g \mid f = (t \mapsto q \cdot t) \text{ with } q \in Q \text{ and } g \in B\}
\end{equation}
of the group $G([0,\infty[\,;A,P)$ does admit a decreasing automorphism.
\end{lem}

\begin{proof}
Let $q_0$ be the generator of $Q$ with $q_0>1$ 
and choose a positive element $a_0 \in IP \cdot A$.
For each $k \in \Z$ set $t_k = q^k \cdot a_0$ 
and define $\varphi \colon \,]0, \infty[\; \iso \;]0, \infty[$ 
to be the affine interpolation of the assignment $(t_k \mapsto t_{-k})_{k \in \Z}$.
Then $\varphi$ is an infinitary PL-auto-homeomorphism of $]0, \infty[$
whose interpolation points lie in $(IP \cdot A) \times  (IP \cdot A)$.
The slopes of the segments forming the graph of $\varphi$ are the negatives of powers of $q_0$;
indeed,
\begin{align*}
t_{k+1} - t_k 
&= 
q^{k+1}_0 \cdot a_0 -q^{k}_0 \cdot a_0 
=
(q_0-1) \cdot q_0^k \cdot a_0\\
\varphi(t_{k+1}) - \varphi(t_k) 
&=
(1/q_0)^{k+1} \cdot a_0 - (1/q_0)^k \cdot a_0 
= 
(1 - q_0) \cdot q_0^{-k-1} \cdot a_0
\end{align*}
and so $\varphi$ has  slope $(-1) \cdot q_0^{-2k - 1}$ on the interval $[t_{k}, t_{k+1}]$.

It follows that $\varphi$ maps $IP \cdot A$ onto itself.
Consider now a conjugate $\act{\varphi}{1}{ h} = \varphi \circ h \circ \varphi^{-1}$ 
of an element $h \in G_Q$.
If $h \in B(I;A,P)$, then $h$ has support contained in some interval of the form 
$I_{k(h)} = [t_{-k(h)}, t_{k(h)}]$ for some $k(h) > 0$ 
and so $\act{\varphi}{0}{ h}$ has support in $\varphi(I_{k(h)}) = I_{k(h)}$,
slopes in $P$, break points in $IP \cdot A$ and is thus an element of $B \subset G_Q$.
If, on the other hand, $h$ is the homothety with ratio $q_0$,
then $h(t_k) = t_{k+1}$ for each index $k \in \Z$ and its conjugate $\act{\varphi}{0}{ h}$
is the PL-function  with interpolation points $(t_k, t_{k-1})$, 
hence the homothety with centre 0 and ratio  $q_0^{-1}$ 
and thus $\act{\varphi}{0}{ h }= h^{-1}$ lies in $G_Q$.
As $G_Q$ is generated by $B \cup \{ (t \mapsto q_0 \cdot t)\}$,
the previous reasoning shows 
that the decreasing auto-homeomorphism $\varphi$ induces 
by conjugation an automorphism of $G_Q$ and so the lemma is established.
\end{proof}

\begin{remark}
\label{remark:Existence-psi-for-homotheties}
Assume $A$, $P$ and $Q$ are as in the statement of the lemma.
Then the bounded group $B =B([0,\infty[\,;A,P)$ may be perfect and hence simple;
cf.\;\cite[Section 12.4]{BiSt14}.
In such a case,
$B$ is the only normal subgroup $N$ of $G_Q$ with $G/N$  \emph{infinite abelian}
and so the lemma implies
that no homomorphism of $G_Q$ onto an infinite abelian group is fixed by all of $\Aut G_Q$.
Note, however, that $\rho$ is fixed by every increasing automorphism of $G_Q$.
\end{remark}

%--------
\subsubsection{Image of $\rho$ consists of translations}
\label{sssec:I-half-line-Groups-with-im-rho-translations}
%-------
We turn now to the other type of abelian subgroups of $\Aff(IP \cdot A, P)$,
but concentrate on a special case.
Given a subgroup $Q$ of $P$ 
and a subgroup $A_0 \subseteq IP \cdot A$, 
we set
\begin{equation}
\label{eq:Definition-G-sub-Q-A0}
G_{Q, A_0}  
= \left\{g \in G([0, \infty[\,;A,P) \mid \sigma_\ell(g) \in Q \text{ and }
\rho(g) \in A_0 \rtimes \{1\} \right\}.
\end{equation}
The group $G_{Q, A_0}$ is an extension of $B([0,\infty[\,;A,P)$ by the abelian group $Q \times A_0$.

The class of groups having the form  $G_{Q, A_0}$ is of interest for several reasons.
Firstly, if $Q$ and $A_0$ are \emph{not} isomorphic,
every automorphism of $G_{Q,A_0}$ is increasing 
by Lemma \ref{lem:Consequence-decreasing-auto-I-half-line}.
This case occurs frequently, as is brought home by the following kind of examples.
Suppose $Q$ is finitely generated and contains an integer $p > 1$,
while $A_0$ is a non-zero submodule of $IP \cdot A$.
Then $A_0$ is divisible by $p$ and, in particular, not free abelian. 

Some groups of the form $G_{Q, A_0}$ admit decreasing automorphisms, 
in particular the following ones.
Let $P$ be a cyclic group generated by the real number $p>1$,
let $A$ be a $\Z[P]$-submodule of $\R_{\add}$ 
and choose a positive element $b \in A$.
The group $\bar{G} = G([0,b]; A, P)$ admits decreasing automorphisms,
for instance the automorphism induced by conjugation by the reflection $\bar{\varphi}$ 
at the midpoint of $I = [0, b]$. 

Consider now the group 
$G = G_{P, \Z\cdot (p-1)b} \subset G([0, \infty[\, ; P, A)$.
It is isomorphic to $\bar{G}$;
there exists actually an isomorphism induced by an increasing, infinitary PL-homeomorphism 
$\varphi_b \colon [0, \infty[\, \iso [0, b[$  (see \cite[Lemma E18.2]{BiSt14}).
The composition $\varphi_b^{-1} \circ \bar{\varphi} \circ \varphi_b$
induces then by conjugation a decreasing automorphism of $G$.

Thirdly,
let  $\tau_r \colon G_{Q,A_0} \to \R_{\add}$ be the homomorphism
that maps the PL-ho\-meo\-morphism $g \in G_{Q,A_0}$ 
to the amplitude of the translation $\rho(g)$.
This homomorphism seems to have a good chance 
of being fixed by $\Aut_+ G_{Q, A_0}$, 
but this impression is mistaken.
Indeed, let $\Aut_P A_0$ be the set of elements $p \in P$ with $p \cdot A_0 = A_0$;
this set is a subgroup of $P$
and the semi-direct product $A_0\rtimes \Aut_P A_0$ is a subgroup of $(IP \cdot A) \rtimes P$;
let $\tilde{G} $ denote the preimage of $A_0\rtimes \Aut_P A_0$
under the epimorphism 
\[
\bar{\rho} \colon G([0, \infty[\,;A,P) \overset{\rho}{\epi} \Aff(IP \cdot A, P) \iso  (IP \cdot A) \rtimes P.
\]
Then $G_{Q,A_0}$ is a normal subgroup of $\tilde{G}$. 
The group $\tilde{G}$ contains the homothety $\vartheta_p \colon t \mapsto p \cdot t$ 
for every $p \in \Aut_P A_0$,
and so conjugation by such a homothety induces an automorphism  $\alpha_p$ of $G_{Q, A_0}$.
The calculation
\begin{align*}
(\tau_r \circ \alpha_p)(g) 
&= 
\tau_r(\vartheta_p \circ g \circ \vartheta_p^{-1})
=
(\vartheta_p \circ g \circ \vartheta_p^{-1})(t) - t\\
&=
\vartheta_p ( g ( p^{-1}t )) - t
=
p \cdot (p^{-1} t + \tau_r(g)) - t 
=
(p \cdot \tau_r)(g),
\end{align*}
valid for every sufficiently large real number $t$,
then shows  
that the formula 
\begin{equation}
\label{eq:Transformation-tau}
\tau_r \circ \alpha_p = p \cdot \tau_r
\end{equation}
holds  for each $p \in  \Aut_P A_0$. 
We conclude 
that $\tau_r$ can only be fixed by all of $\Aut_+ G_{Q,A_0}$ 
if $\Aut_P A_0$ is reduced to $1 \in \R^\times_{>0}$. 
This condition is fulfilled, for instance, if $A_0$ is infinite cyclic.

\begin{example}
\label{example:character-not-fixed}
Given a real number $p > 1$,
set $P = \gp(p)$ and $A = \Z[P] = \Z[p, p^{-1}]$.
Choose $A_0 = A$ and set $G = G_{P,A_0}$.
Then $\Aut_P A_0 = P$.
Concrete examples are rational integers $p \in \N \smallsetminus \{0, 1\}$,
with $A_0= \Z[1/p]$, or quadratic integers like $\sqrt{2} + 1$ with
$A_0 = A = \Z[\sqrt{2}\,]$. 
We shall come back to the second of these examples 
in section \ref{sssec:Group-of-units-elementary-examples}.
\end{example}
%
%========
%
\section{Characters fixed by $\Aut G(\R;A,P)$}
\label{sec:Homomorphisms-fixed-by-Aut-I-line}
%
%========
%
Let $I$ denote one of the intervals $[0,b]$, $[0, \infty[$ or $\R$,
 and let $G$ be a subgroup of $G(I;A,P)$ containing $B(I;A,P)$.
In Sections  
\ref{sec:Homomorphisms-fixed-by-Aut-I-compact}
and
\ref{sec:Homomorphisms-fixed-by-Aut-I-half-line}
groups with $I$ a compact interval or a half line have been studied.
In this section we now turn to the line $I = \R$.
Finding non-zero homomorphisms $\psi \colon G \to \R^\times_{>0}$ fixed by $\Aut G$,
is then harder than in the previously investigated cases,
and this for two reasons.
Firstly, 
subgroups of $G(\R;A,P)$ often admit decreasing automorphisms $\alpha$, 
in contrast to what happens if $I$ is a half line;
in the case of a decreasing automorphism,
$\lambda$ (or $\rho$) is only fixed by $\alpha$ if $\lambda$ coincides with $\rho$.
Secondly,
if the image of $\lambda$ or that of $\rho$ consists of translations,
neither $\lambda$ nor $\rho$ need be fixed by $\Aut_+ G$.

The plan of our investigation will be similar to that adopted 
in Section \ref{sec:Homomorphisms-fixed-by-Aut-I-half-line}.
We begin by discussing the existence of decreasing automorphisms
(in section \ref{ssec:I-line-Existence-decreasing-autos}),
move on to the main results about the existence of homomorphisms fixed by $\Aut_+ G$ or $\Aut G$
(in section \ref{ssec:I-line-Construction-homomorphisms-I})
and complement these results with more special findings in section 
\ref{ssec:I-line-Construction-homomorphisms-II}.
The layout of the middle section \ref{ssec:I-line-Construction-homomorphisms-I}
will resemble that of section \ref{ssec:Differentiability-criterion}.
%
%-----------------
\subsection{Existence of decreasing automorphisms}
\label{ssec:I-line-Existence-decreasing-autos}
%-----------------
%
As in the cases of a compact interval or a half line,
the existence of a decreasing automorphism has an easily stated consequence,
namely
\begin{lem}
\label{lem:Consequence-decreasing-auto-I-line}
Assume $G$  is a subgroup of $G(\R;A,P)$ that contains $B(\R;A,P)$.
Then every decreasing automorphism $\alpha$ induces an isomorphism 
$\alpha_* \colon \im \lambda \iso \im \rho $ 
that renders commutative the following square.
\begin{equation}
\label{eq:square-lambda-rho}
\xymatrix{G \ar@{->}[r]^-{\alpha} \ar@{->>}[d]^-{\lambda} & G \ar@{->>}[d]^-{\rho}\\
\im \lambda\ar@{->}[r]^-{\alpha_{*}} &\im \rho}
\end{equation}
\end{lem}
\begin{proof}
The claim can be established as in the proofs of Lemmata
\ref{lem:Consequence-decreasing-auto-I-compact}
and
\ref{lem:Consequence-decreasing-auto-I-half-line}.
\end{proof}
The images of $\lambda$ and $\rho$ 
are both subgroups of the affine group $Q= \Aff_o(IP \cdot A, P)$.
It is easy to describe some pairs of subgroups $(Q_1, Q_2)$ 
that are \emph{not} isomorphic for obvious reasons,
for instance if one is abelian, and the other is non-abelian.
We are, however, not aware of a classification of the isomorphism types of subgroups of
$\Aff_o(IP \cdot A, P)$ for parameters $A \neq \{ 0\} $ and $P \neq \{1\}$.
%
%-----------------
\subsection{Construction of homomorphisms: part I}
\label{ssec:I-line-Construction-homomorphisms-I}
%-----------------
%
We turn now to the construction of homomorphisms 
that are fixed by $\Aut_+ G$ or by $\Aut G$.
The next result is an analogue of Corollary \ref{crl:I-compact-summary-differentiability}.
The main ingredient in its proof is Proposition \ref{prp:Affine-near-infty-I-half-line}.
\begin{prp}
\label{prp:I-line-affine-near-infty}
Let $G$ be a subgroup of $G(\R;A,P)$ containing $B(\R;A,P)$
and let $\alpha$ be an automorphism of $G$ 
that is induced by conjugation by the auto-homeo\-mor\-phism 
$\varphi_\alpha \colon \R \iso \R$.
Then the following statements hold: 
\begin{enumerate}[(i)]
\item if $\alpha$ is increasing
\footnote{See Definition \ref{definition:Increasing-isomorphism}.}  
and $\im \rho$ is not  abelian,
$\varphi_\alpha$ is affine near $\infty$; 
\item if $\alpha$ is increasing and $\im \lambda$ is not abelian,
$\varphi_\alpha$ is affine near $-\infty$; 
\item 
if $\alpha$ is decreasing and $\im \rho$ is not abelian,
$\tilde{\varphi}_\alpha$ is affine, both near $-\infty$ and near $\infty$.
\end{enumerate}
\end{prp}

\begin{proof}
(i) is a restatement of the claim of Proposition \ref{prp:Affine-near-infty-I-half-line}.
To establish (ii), we show that (ii) can be reduced to (i).
Let $\vartheta \colon \R \iso \R$ be the reflection in the origin 0, 
set $G_1 = \vartheta \circ G \circ  \vartheta^{-1}$ 
and $\varphi_1 =  \vartheta \circ \tilde{\varphi}_\alpha \circ \vartheta^{-1}$.
We claim that Proposition \ref{prp:Affine-near-infty-I-half-line}
applies to the couple $(G_1, \varphi_1)$.
Indeed,
the groups $G(\R;A,P)$ and  $B(\R;A,P)$ are invariant under conjugation by $\vartheta$
and so $G_1$ is a subgroup of $G(\R;A,P)$ containing $B(\R;A,P)$.
Next, Lemma \ref{lem:Formula-involving-lambda-rho-theta} below
shows that
\[
\rho(G_1) = \rho\left( \vartheta \circ G \circ \vartheta^{-1}\right) 
= \vartheta \circ \lambda(G) \circ \vartheta^{-1}.
\]
The group $\vartheta \circ \lambda(G) \circ \vartheta^{-1}$ is isomorphic to $\im \lambda$, 
which is non-abelian by hypothesis,
and so $\rho(G_1)$ is non-abelian.
Proposition \ref{prp:Affine-near-infty-I-half-line} thus applies to  $G_1$ and to $\varphi_1$
and implies 
that $\varphi_1 = \vartheta \circ \varphi_\alpha \circ \vartheta^{-1}$ is affine near $+\infty$,
whence $\varphi$ itself is affine near $-\infty$.

(iii) Since $\alpha$ is \emph{decreasing},
the groups $\im \lambda$ and $ \im \rho$ are isomorphic
(see Lemma \ref{lem:Consequence-decreasing-auto-I-line});
the hypothesis on $\im \rho$ implies therefore
that the image of $\lambda$ is not abelian.
The idea now is to reduce (iii) to the previously treated cases (i) and (ii).
As before,
let $\vartheta \colon \R \iso \R$ denote the reflection in the origin 0,
and set $\varphi_2 = \vartheta \circ \varphi_\alpha$.
Then $\varphi_2$ is increasing 
and conjugation by $\varphi_2$ maps $G$ onto $\bar{G} = \vartheta \circ G \circ \vartheta^{-1}$. 
Proposition \ref{prp:Affine-near-infty-I-half-line} thus applies 
and guarantees that $\varphi_2$ is affine near $\infty$.
But $\varphi_2 = \vartheta \circ \varphi_\alpha$ 
and so $\varphi_\alpha$ itself is affine near $\infty$. 
Consider, secondly, $\varphi_3 = \varphi_\alpha \circ \vartheta$.
This map is again increasing, and conjugation by it maps 
$\bar{G} = \vartheta \circ G \circ \vartheta^{-1}$ onto $G$.
Invoking Proposition \ref{prp:Affine-near-infty-I-half-line} once more,
we learn that $\varphi_3$ is affine near $+\infty$,
and so $\varphi_\alpha$ itself is affine near $-\infty$.
All taken together, we have shown that $\varphi_\alpha$ is affine, 
both near $-\infty$ and $+\infty$, as asserted by claim (iii).
\end{proof}

We are left with proving
\begin{lem}
\label{lem:Formula-involving-lambda-rho-theta}
Let $\vartheta \colon \R \iso \R $ denote the reflection in 0.
Then the formula
\begin{equation}
\label{eq:Formula-involving-lambda-rho-theta}
\rho\left( \vartheta \circ g \circ \vartheta^{-1}\right) 
= \vartheta \circ \lambda(g) \circ \vartheta^{-1}
\end{equation}
holds for every $g \in \PL_o(\R)$.
\end{lem}
\begin{proof}
Let $\mu$ and $\nu$ denote the functions of $\PL_o(\R)$ into itself
given by the left hand and the right hand side of equation 
\eqref{eq:Formula-involving-lambda-rho-theta};
thus $\mu(g) = \rho(\vartheta \circ g \circ \vartheta^{-1})$ for $g \in PL_o(\R)$,
and similarly for $\nu$.
Both functions are homomorphisms of $\PL_o(\R)$ into $\Aff_o(\R)$
that vanish on $\ker \lambda$.
It suffices therefore to check equation \eqref{eq:Formula-involving-lambda-rho-theta}
on a complement of $\ker(\lambda \colon \PL_o(\R) \to \Aff(\R))$.
Such a complement is $\Aff_0(\R)$
and for affine maps $h$ the following calculation holds:
\[
\rho(\vartheta \circ h \circ \vartheta^{-1})
=
\vartheta \circ h \circ \vartheta^{-1}
=
\vartheta \circ  \lambda( h )\circ \vartheta^{-1}. \qedhere
\]
\end{proof}
%----------
\subsubsection{Some corollaries}
\label{sssec:Applications-Proposition-I-line-affine-near-infty}
%--------
%
The first corollary of Proposition \ref{prp:I-line-affine-near-infty}
deals with homomorphisms fixed by $\Aut_+$;
the corollary is an analogue of Theorem \ref{thm:Existence-psi-I-compact-increasing-autos}.
\begin{theorem}
\label{thm:Existence-psi-I-line-image-lambda-and-sigma-non-abelian-increasing}
Assume $G$ is a subgroup of $G(\R;A,P)$ that contains $B(\R;A,P)$.
If $\im \rho$ is not abelian 
then $\sigma_r$ is a non-zero homomorphism fixed by $\Aut_+ G$.
Similarly, 
$\sigma_\ell$ is a non-zero homomorphism fixed by $\Aut_+ G$ 
in case $\im \lambda$ is not abelian.
\end{theorem}

\begin{proof}
Let $\alpha$ be an increasing automorphism of $G$
and let $\varphi_\alpha$ be the increasing auto-homeomorphism of $\R$ 
inducing $\alpha$ by conjugation.
(The map exists thanks to Theorem \ref{thm:TheoremE16.4}.)
Assume first that $\im \rho$ is not abelian.
By part (i) of Proposition \ref{prp:I-line-affine-near-infty}
the map $\varphi_\alpha$ is then affine near $\infty$.
On the other hand,
the image of  $\rho$, being non-abelian,  
cannot consist merely of translations; 
so the homomorphism $\sigma_r \colon G \to P$ is non-zero.
The following calculation then reveals that $\sigma_r$ is fixed by $\alpha$:
\begin{align*}
\left(\sigma_r \circ \alpha\right)(g) 
&=
\sigma_r\left(\varphi_\alpha \circ g \circ \varphi_\alpha^{-1}\right)\\
&=
\lim\nolimits_{t \to \infty} 
\left(\varphi_\alpha \circ g \circ \varphi_\alpha^{-1}\right)'(t)\\
&=
\lim\nolimits_{t \to \infty}
\left( 
\varphi_\alpha'\left(g(\varphi_\alpha^{-1}(t))\right) \cdot g'(\varphi_\alpha^{-1}(t)) \cdot (\varphi_\alpha^{-1})'(t)
\right)\\
&=
\lim\nolimits_{t \to \infty}g'(\varphi_\alpha^{-1}(t))
= 
\sigma_r(g).
\end{align*}
In this calculation the facts
that the derivatives of $ \varphi_\alpha$ and of $g$ 
are constant on a half line of the form $[t_*, \infty[$
and that $\varphi_\alpha$ is an increasing homeomorphism,
have been used.

Assume next that $\im \lambda$ is not abelian.
By part (ii) of Proposition \ref{prp:I-line-affine-near-infty}
the map $\varphi_\alpha$ is then affine near $-\infty$.
and the homomorphism $\sigma_\ell \colon G \to P$ is non-zero.
Since the derivatives of every element $g \in G$ and of $\varphi_\alpha$ 
are constant near $-\infty$, 
a calculation similar to the preceding one will show 
that $\lambda$ is fixed by $\alpha$.
\end{proof}

As a second application of Proposition \ref{prp:I-line-affine-near-infty},
we present a result that furnishes a homomorphism $\psi$
that is fixed by every automorphism.
Note, however, that the hypotheses of the result do not imply
that $\psi$ is non-trivial.
\begin{theorem}
\label{thm:Existence-psi-I-line-image-rho-non-abelian}
Assume $G$ is a subgroup of $G(\R;A,P)$ containing $B(\R;A,P)$
and let $\psi \colon G \to P$ be the homomorphism $g \mapsto \sigma_\ell(g) \cdot \sigma_r(g)$.
If the images of $\lambda$ and of $\rho$ are both non-abelian,
the homomorphism $\psi\colon G \to P$ is fixed by $\Aut G$.
\end{theorem}

\begin{proof}
Let $\alpha$ be an automorphism of $G$ 
and let $\varphi_\alpha$ be the auto-homeomorphism of $\R$ 
that induces $\alpha$ by conjugation.
If $\varphi_\alpha$ is \emph{increasing} both $\sigma_\ell$ and $\sigma_r$ are fixed by $\alpha$
(see Theorem \ref{thm:Existence-psi-I-line-image-lambda-and-sigma-non-abelian-increasing}) 
and hence so is $\psi$.

Assume now that $\alpha$ is \emph{decreasing}.
Part (iii) of Corollary \ref{crl:I-compact-summary-differentiability} then guarantees 
that $\varphi_\alpha$ is affine near $-\infty$ and also near $\infty$.
These facts imply the relations
\begin{equation}
\label{eq:Transformations}
\sigma_\ell \circ \alpha = \sigma_r
\quad\text{and}\quad 
\sigma_r \circ \alpha = \sigma_\ell
\end{equation} 
(see below) and so $\psi = \sigma_\ell \cdot \sigma_r$
is fixed by $\alpha$.

We are left with verifying relations \eqref{eq:Transformations}.
The following calculation uses the fact 
that both $\varphi_\alpha$ and $g$ have constant derivatives near $-\infty$ and $+\infty$:
\begin{align*}
\left(\sigma_\ell \circ \alpha\right)(g) 
&=
\sigma_\ell\left(\varphi_\alpha \circ g \circ \varphi_\alpha^{-1}\right)\\
&=
\lim\nolimits_{t \to -\infty} 
\left(\varphi_\alpha \circ g \circ \varphi_\alpha^{-1}\right)'(t)\\
&=
\lim\nolimits_{t \to -\infty}
\left( 
\varphi_\alpha'\left(g(\varphi_\alpha^{-1}(t))\right) \cdot g'(\varphi_\alpha^{-1}(t)) \cdot (\varphi_\alpha^{-1})'(t)
\right)\\
&=
\lim\nolimits_{t \to -\infty}g'(\varphi_\alpha^{-1}(t))
= 
\sigma_r(g).
\end{align*}
A similar calculation establishes the second relation in \eqref{eq:Transformations}.
\end{proof}

We continue with an easy consequence of Theorem 
\ref{thm:Existence-psi-I-line-image-rho-non-abelian}.
If the group $G$ is all of $G(I;A,P)$ the homomorphism 
$\psi \colon g \mapsto \sigma_\ell(g) \cdot \sigma_r(g)$ is surjective; 
in addition,
$\im \lambda$ and $\im \rho$ both coincide with $\Aff(A,P)$ 
and thus are non-abelian.
Theorem \ref{thm:Existence-psi-I-line-image-rho-non-abelian} 
implies therefore
\begin{crl}
\label{crl:G=G(R;A,P)-I-line}
If $G =G(\R;A,P)$ the homomorphism $\psi \colon G \to P$, 
taking $g \in G$ to $\sigma_\ell(g) \cdot \sigma_r(g)$, 
is non-zero and fixed by $\Aut G$.  
\end{crl}

Corollary \ref{crl:G=G(R;A,P)-I-line}
 is an analogue of Corollaries \ref{crl:G=G(I;A,P)} and  \ref{crl:Existence-psi-I-halfline-fixed-by-AutG}.
Groups of the form $G(\R;A, P)$ have been investigated, so far, less often
than groups with $I$ a compact interval;
they have, however, their own merits if it comes to finite generation.
There exists, first of all, a characterization of the finitely generated groups of the form $G(\R;A, P)$,
namely
\begin{prp}[Theorem B7.1 in \cite{BiSt14}]
\label{prp:Characterization-G-fg-I-line}
The group $G(\R;A,P)$ is finitely generated if, and only if, 
$P$ is finitely generated and $A$ is a finitely generated $\Z[P]$-module.
\end{prp}

\begin{remark}
\label{remark:Continuously-many-non-isomorphic-groups}
Proposition \ref{prp:Characterization-G-fg-I-line} implies
that \emph{there are continuously many, pairwise non-isomorphic, 
finitely generated groups of the form $G(\R;A,P)$}.

To prove this assertion, we recall the following result:
\emph{if two groups of the form $G(\R; A, P)$ and $G(\R;\bar{A}, \bar{P})$ are isomorphic 
and if $P$ is not cyclic, then $P = \bar{P}$}.
\footnote{see Theorem E17.1 in \cite{BiSt14}.}

It suffices therefore 
to find a collection of finitely generated, pairwise distinct, subgroups 
$\{P_j \mid j \in J \}$ of $\R^\times_{>0}$ with $J$ an index set having the cardinality of $\R$,
and  to set $A_j = \Z[P_j]$ for each $j \in J$.
Such a collection of subgroups can be obtained as follows:
one constructs first  a family of irrational, real numbers $\{x_j \mid j \in J \}$
such that the extended family 
$\{1\} \cup \{x_j \mid j \in J \}$ 
is linearly independent (over $\Q$)
and then sets $P_j = \exp(\gp(\{1, x_j\})$.
Then each group $P_j$ is free abelian of rank two, hence not cyclic, 
and for indices $j_1 \neq j_2$ the groups $P_{j_1}$ and $P_{j_2}$ are distinct.
\end{remark}
%
%-----------------
\subsection{Construction of homomorphisms: part II}
\label{ssec:I-line-Construction-homomorphisms-II}
%-----------------
In this final part of Section \ref{sec:Homomorphisms-fixed-by-Aut-I-line},
we consider subgroups $G$ of $G(\R;A,P)$, containing $B(\R;A,P)$, 
with $\im \lambda$ and $\im \rho$ both abelian.
\footnote{If exactly one of $\im \lambda$ and $\im \rho$ is abelian, 
the group does not admit a decreasing automorphism
(by Lemma \ref{lem:Consequence-decreasing-auto-I-line})
and so Theorem \ref{thm:Existence-psi-I-line-image-lambda-and-sigma-non-abelian-increasing}
yields a non-zero homomorphism fixed by $\Aut G$.}
The most interesting subcase seems to be that 
where the images of $\lambda$ and $\rho$ consists only of translations.
Then two homomorphisms $\tau_\ell$ and $\tau_r$ of $G$ into $\R_{\add}$  can be defined:
they associate to $g \in G$ the amplitudes of the translations $\lambda(g)$ and $\rho(g)$, respectively.
One sees, as in section \ref{sssec:I-half-line-Groups-with-im-rho-translations},
that neither of these homomorphisms need be fixed by $\Aut_+ G$.

An exception occurs if the image of $\rho$ or of $\lambda$ is \emph{infinite cyclic}.
Suppose, for instance, that $\im \rho$ is infinite cyclic,
and let $f \in G$ be an element that maps onto the positive generator, say $x_f$, of $\im \tau_r$.
Consider an increasing automorphism $\alpha$ of $G$ 
and let $\varphi_\alpha$ be the homeomorphism of $\R$ that induces $\alpha$ by conjugation.
Then $\tau_r(\alpha(f))$ generates $\im \tau_r$, too,
and so $\tau_r(\alpha(f)) = \pm x_f$. 
Near $+\infty$, the map $f$ is a translation with positive amplitude,
hence so is $\alpha(f) = \varphi_\alpha \circ f \circ \varphi_\alpha^{-1}$,
and so $\tau_r(\alpha(f)) > 0$.
Thus $\tau_r(f) = (\alpha \circ \tau_r)(f)$.
We conclude  that $\tau_r$ is fixed by $\alpha$.
An analogous argument shows 
that $\tau_\ell$ is fixed by every increasing automorphism of $G$.

All taken together we have thus established
\begin{prp}
\label{prp:Images-tau-ell-and-tau-r-cylic}
Let $G$ be a subgroup of $G(\R;A,P)$ containing $B(\R;A,P)$.
Assume that the images of $\lambda$ and $\rho$ contain only translations 
and that these images are infinite cyclic.
Then $\tau_\ell$ and $\tau_r$ are both non-zero homomorphisms that are fixed by $\Aut_+ G$.
\end{prp}

\begin{example}
\label{example:images-tau-ell-and-tau-r-cyclic}
Suppose $P$ is an infinite cyclic group,
$A$ a (non-zero) $\Z[P]$-module and $b$ a positive element of $A$.
Set $\bar{G} = G([0,b];A,P)$.
Then there exists a homeomorphism $\vartheta \colon ] 0, b[\, \iso \R$ that induces, by conjugation,
an embedding 
\[
\mu \colon G([0,b];A,P) \mono G(\R; A, P)
\] 
whose image contains $B(\R;A, P)$.
\footnote{In special cases, for instance if $\bar{G}$ is Thompson's group $F$,
this fact is well-known (see, \eg{} \cite[Proposition 3.1.1]{BeBr05}); 
the general claim is established in \cite{BiSt14} (see Lemma E18.4).} 
Let $G$ denote the image of $\mu$.
The images of $\lambda\restriction{G}$ and $\rho \restriction{G}$
are both infinite cyclic and consist of translations.
The images of $\tau_\ell$ and $\tau_r$ are therefore infinite cyclic, too,
and so the previous lemma applies.

Let's now consider the special case where $P$ is generated by an integer  $n \geq 2$,
where $A = \Z[P] = \Z[1/n]$ and $b = 1$.
For a suitably chosen homeomorphism $\vartheta$ 
the image $G$ of $\mu$ consists then of all elements $g  \in G(\R;\Z[1/n], \gp(n))$
fulfilling the conditions
\begin{equation}
\label{eq:Describing-image-mu}
\sigma_\ell(g) = \sigma_r(g) = 1
\quad\text{and }\quad
\tau_\ell(g) \in \Z(n-1), \quad \tau_r(g) \in \Z(n-1);
\end{equation}
see \cite[Lemma E18.4]{BiSt14}.
This group $G$ is called $F_{n, \infty}$ in \cite[p.\;298]{BrGu98}.

By relaxing conditions \eqref{eq:Describing-image-mu} 
one obtains supergroups of $F_{n, \infty}$,
in particular the group called $F_n$ in \cite[p.\;298]{BrGu98} 
and defined by the requirements
\begin{equation}
\label{eq:Describing-image-mu-2}
\sigma_\ell(g) = \sigma_r(g) = 1
\quad\text{and}\quad
\tau_\ell(g) \in \Z, \;\; \tau_r(g) \in \Z, \;\; \tau_r(g) -  \tau_\ell(g) \in \Z(n-1);
\end{equation}
see \cite[Proposition 2.2.6]{BrGu98}.
Proposition \ref{prp:Images-tau-ell-and-tau-r-cylic} applies to the groups 
$F_{n, \infty}$, but also to the larger groups $F_n$.
Now, the groups $F_n$ and $F_{n, \infty}$ both admit decreasing automorphisms, 
in particular the automorphism induced by the reflection in the origin.
The homomorphisms $\tau_\ell$ and $\tau_r$ are therefore not fixed 
by the full automorphism group of the groups $F_{n, \infty}$ and $F_n$,
but the difference $\tau_r - \tau_\ell$ is a non-zero homomorphism,  with infinite cyclic image,
that enjoys this property.
\end{example}
%\newpage

%
%
%==========
\section{Characters fixed by $\Aut G$ with $G$ a subgroup of $\PL_o([0,b])$}
\label{sec:Generalization-GoKo10}
%==========
%
 In this section we prove Theorem  \ref{thm:Generalization-GoKo10}.
 For the convenience of the reader we restate this result here as
 \begin{theorem}
\label{thm:Generalization-GoKo10-bis}
Suppose $I = [0,b]$ is a compact interval of positive length
and $G$ is subgroup of $PL_o(I)$ 
that satisfies the following conditions: 
\begin{enumerate}[(i)]
\item no interior point of the interval $I = [0, b]$ is fixed by $G$;
\item the characters $\chi_\ell$ and $\chi_r$ are both non-zero;
\item the quotient group  $G/(\ker \chi_\ell \cdot \ker \chi_r)$ is a torsion group, and
\item at least one of the group of units $U(\im\chi_\ell)$ and  $U(\im \chi_r)$ is reduced to $\{1,-1\}$.  
\end{enumerate}
Then there exists a non-zero homomorphism $\psi \colon G \to \R^\times_{>0}$
that is fixed by every automorphism of $G$.
The group  $G$ has therefore property $R_\infty$.
\end{theorem}

We explain next the layout of Section \ref{sec:Generalization-GoKo10}.
We begin by recalling the definition of the invariant $\Sigma^1$
and stating some basic results concerning it.
In section \ref{ssec:Proof-Theorem-Generalization-GoKo10-bis}, 
we prove Theorem \ref{thm:Generalization-GoKo10-bis}.
The hypotheses of the theorem allow of variations 
that deserve some comments.
This topic is taken care of in sections 
\ref{ssec:Discussion-hypotheses-Generalization-GoKo10-bis}
through
\ref{ssec:Subgroups-of-finite-index}. 
%
%-----------------
\subsection{Review of $\Sigma^1$}
\label{ssec:Review-Sigma1}
%-----------------
%
Given an infinite group $G$, 
consider the real vector space $\Hom(G,\R)$ 
made up of all homomorphisms $\chi \colon G \to \R_{\add}$  into the additive group of $\R$.
These homomorphisms will be referred to as \emph{characters}.
Two non-zero characters $\chi_1$ and $\chi_2$ are called equivalent,
if one is a positive real multiple of the other.
Geometrically speaking,
the associated equivalence classes are (open) rays emanating from the origin.
The space of all rays is denoted by $S(G)$ 
and called the \emph{character sphere} of $G$.
In case the abelianization $G_{\ab} = G/[G,G]$ of $G$ is finitely generated,
the vector space $\Hom(G,\R)$ is finite dimensional 
and carries a unique topology, induced by its norms;
the sphere $S(G)$ equipped with the quotient topology 
is then homeomorphic to the spheres in a Euclidean vector space of dimension 
$\dim_\Q H_1(G, \Q) = \dim_\Q (G_{\ab} \otimes \Q)$.

The invariant $\Sigma^1(G) $ is a subset of $S(G)$.
It admits several equivalent definitions;
in the sequel, we use the definition in terms of Cayley graphs.
\footnote{See, \eg{}Chapter C in \cite{Str13} for alternate definitions.}
Fix a generating set $\XX$ of $G$ 
and define $\Gamma = \Gamma(G, \XX)$ to be the associated Cayley graph of $G$.
This graph can be equipped with $G$-actions;
as we want to work with  \emph{left} $G$-actions 
we define the set of positive edges of the Cayley graph like this: 
\[
E_+(\Gamma) =  \{(g, g \cdot x) \in  G \times G \mid (g, x) \in G \times \XX \}.
\]

We move on to the \emph{definition of} $\Sigma^1(G)$.
Given a non-zero character $\chi$,
consider the submonoid $G_\chi = \{ g \in G \mid \chi(g) \geq 0\}$ of $G$
and define $\Gamma_\chi = \Gamma(G, \XX)_\chi$ 
to be the full subgraph of $\Gamma(G; \XX)$ with vertex set $G_\chi$.
Both the submonoid $G_\chi$ and the subgraph $\Gamma_\chi$ remain the same 
if $\chi$ is replaced by a positive multiple;
so these objects depend only on the ray $[\chi] = \R_{>0} \cdot \chi$ represented by $\chi$.
The Cayley graph $\Gamma$ is connected,
but its subgraph $\Gamma_\chi$ may not be so;
the invariant $\Sigma^1(G)$ records the rays for which the subgraph 
$\Gamma_\chi = \Gamma(G, \XX)_\chi$ \emph{is connected}. 
In symbols,
\begin{equation}
\label{eq:Definition-Sigma1-XX}
\Sigma^1(G, \XX) = \{ [\chi] \in S(G) \mid \Gamma(G, \XX)_\chi \text{ is connected} \}.
\end{equation}
One now faces the problem,
familiar from Homological Algebra, 
that the definition of $\Sigma^1(G, \XX)$ involves an arbitrary choice
and that one wants to construct an object that does not depend on this choice.

Suppose, first, that $G$ is \emph{finitely generated}
and let $\XX_f$ be a \emph{finite} generating set.
Then the subgraph $\Gamma(G, \XX_f)_\chi$ is connected if, and only if,
all the subgraphs  $\Gamma(G, \XX)_\chi$, with $\XX$ a generating set, are connected
(see, \eg{}\cite[Lemma C2.1]{Str13}) 
and so the following definition is licit:
\begin{definition}
\label{definition:Sigma1}
Let $G$ be a finitely generated group and $\XX_f$ a \emph{finite} generating set of $G$.
Then $\Sigma^1(G)$ is defined to be the subset
\begin{equation}
\label{eq:Definition-Sigma1-fg}
\{ [\chi] \in S(G) \mid \Gamma(G, \XX_f)_\chi \text{ is connected} \}.
\end{equation}
\end{definition}
The fact 
that the set  \eqref{eq:Definition-Sigma1-fg}
does not depend on the choice of the finite set $\XX_f$, 
allows one to select $\XX_f$ in accordance with the problem at hand;
see \cite[Sections A2.3a and A2.3b]{Str13} for some consequences of this fact.

Suppose now that $G$ is an arbitrary group. 
A useful subset of $S(G)$
can then be obtained by defining
\begin{equation}
\label{eq:Definition-Sigma1}
\Sigma^1(G) = \{ [\chi] \in S(G) \mid \Gamma(G, \XX)_\chi \text{ is connected for every generating set } \XX\}
\end{equation}
(cf.\,\cite[Definition C2.2]{Str13}).
If $G$ happens to be finitely generated, 
the sets \eqref{eq:Definition-Sigma1-fg} and \eqref{eq:Definition-Sigma1} are equal;
for an arbitrary group,
the set  $\Sigma^1(G) $ coincides with the invariant  $\Sigma(G)$ defined 
by Ken Brown in \cite[p.\,489]{Bro87b} \emph{up to a sign};
in other words,
\begin{equation}
\label{eq:Relating-Brown-Cayley-graph}
\Sigma(G) = - \Sigma^1(G).
\end{equation}
The sign in this formula is caused by the fact 
that Brown uses right actions on $\R$-trees,
whereas \emph{left}  actions are employed in our definition of $\Sigma^1$.

The subset $\Sigma^1(G)$ of $S(G)$ is traditionally called  the $\Sigma^1$-\emph{invariant}.
The epithet ``invariant'' is justified by a fact that we explain next.
Suppose $\alpha \colon G \iso \bar{G}$ is an isomorphism of groups.
Then $\alpha$ induces, first of all,  a linear isomorphism of vector spaces  
$\Hom(\alpha, \R)  \colon \Hom(\bar{G}, \R) \iso \Hom(G, \R)$,
and so an isomorphism of spheres 
\begin{equation}
\label{eq:Invariance-Sigma1-isos}
\alpha^* \colon S(\bar{G}) \iso S(G), \qquad [\bar{\chi}] \mapsto [\bar{\chi} \circ \alpha]
\end{equation}
This second isomorphism maps the subset 
$\Sigma^1(\bar{G}) \subseteq S(\bar{G})$ onto $\Sigma^1(G) \subseteq S(G)$
(section B1.2a in \cite{Str13} has more details).

In the sequel,
the special case where $\alpha$ is an \emph{automorphism} will be crucial.
The assignment 
\begin{equation}
\label{eq:Representation-autos}
\alpha \longmapsto (\alpha^{-1})^* \colon \Sigma^1(G) \iso \Sigma^1(G)
\end{equation}
defines a homomorphism of the automorphism group of $G$  
into the group of bijections of $\Sigma^1(G)$,
and hence also one into that of its complement $\Sigma^1(G)^c$.

%-----------------
\subsection{$\Sigma^1$ of subgroups of $\PL_o([0,b])$}
\label{ssec:Sigma1-subgroups-PL(compact-interval)}
%-----------------
Given a subgroup $G$ of $\PL_o([0,b])$, 
let $\sigma_\ell$ be the homomorphism 
that assigns to a function $g \in G$ the value of its (right) derivative in the \emph{left} end point $0$;
similarly, 
define  $\sigma_r \colon G \to \R^\times_{>0}$ to be the homomorphism given by the formula
$\sigma_r(g) = \lim\nolimits_{t \to b} g'(t)$.
The homomorphisms $\sigma_\ell$ and $\sigma_r$ generalize the maps with the same names 
studied in Section \ref{sec:Homomorphisms-fixed-by-Aut-I-compact}.
By composing them with the natural logarithm function,
one obtains characters of $G$,
namely
\begin{equation}
\label{eq:Definition-chi-ell-chi-r}
\chi_\ell = \ln \circ \;\sigma_\ell 
\quad\text{and}\quad 
\chi_r = \ln \circ \;\sigma_r.
\end{equation}
The invariant $\Sigma^1(G)^c$ turns out to consist of precisely two points,
represented by the characters $\chi_\ell$ and $\chi_r$,
provided $G$ satisfies certain restrictions.
The first of them rules out
that $G$ is a direct product of subgroups $G_1$, $G_2$ 
with supports in two disjoint open subintervals $I_1$, $I_2$,
and more general decompositions;
the second requires that $\chi_\ell$ and $\chi_r$ be non-zero
and hence represent points of $S(G)$;
the third condition is natural in the sense
that it holds for all groups of the form $G([0;b];A,P)$ investigated in Section 
\ref{sec:Homomorphisms-fixed-by-Aut-I-compact}.
\begin{theorem} 
\label{thm:Generalization-BNS} 
Let $I$ be a compact interval of positive length and  $G$ a subgroup of $\PL_o(I)$.
Assume the following requirements are satisfied:
\begin{enumerate}[(i)]
\item no interior point of $I$ is fixed by $G$; 
\item the characters $\chi_\ell$ and $\chi_r$ are both non-zero, and
\item the quotient group $G/(\ker \chi_\ell \cdot \ker \chi)$ is a torsion-group.
\end{enumerate}
Then $\Sigma^1(G)^c=\{[\chi_\ell],[\chi_r]\}$.
\end{theorem}  

\begin{remarks}
\label{remarks:thm-Generalization-BNS}

a) 
Theorem \ref{thm:Generalization-BNS} 
generalizes Theorem 8.1 in \cite{BNS};
there $G$ is assumed to be finitely generated 
and condition (iii) is sharpened to  $G = \ker \chi_\ell \cdot \ker \chi_r$.
The theorem improves also on a result stated in \cite[Remark on p.\,502]{Bro87b}.
A proof of Theorem \ref{thm:Generalization-BNS}, 
based on the Cayley graph definition of $\Sigma^1(G)$,
can be found in \cite{Str15}; see Theorem 1.1.

b) 
We continue with a comment that seems overdue.
In \cite{BNS} an invariant $\Sigma_{G'}(G)$ is introduced for finitely generated groups $G$; 
in the sequel,
this invariant will be called $\Sigma^{BNS}(G)$.
It is defined in terms of a generation property 
that uses \emph{right} conjugation, 
while left action is employed in the definition of $\Sigma^1(G)$.
There is, however, a close connection between the two invariants:
if $G$ is finitely generated then 
\begin{equation}
\label{eq:Relating-BNS-Cayley-graph}
\Sigma^{BNS}(G) = - \Sigma^1(G),
\end{equation}
similar to the what happens for Brown's invariant $\Sigma(G)$;
see formula \eqref{eq:Relating-Brown-Cayley-graph}.

Now, PL-homeomorphism groups are examples of groups made up of permutations,
and for such a group $G$ the underlying set can be equipped with two familiar compositions.
Suppose the composition in the group $G$ is the one familiar to analysts 
(and used in this paper); 
to emphasize this fact call the group temporarily $G_{ana}$.
The assignment $g \mapsto g^{-1}$ defines then an anti-automorphism of $G_{ana}$
and hence an isomorphism $\iota \colon G_{ana} \iso G_{gt}$
onto the group obtained by equipping the set underlying $G_{ana}$ 
 with the composition defined by $f \circ g \colon t \mapsto f(t) \mapsto g(f(t))$
and preferred by many \emph{g}roup \emph{t}heorists.
The invariants of the groups $G_{ana}$ and $G_{gt}$ are then related by the formulae
\[
\Sigma^1(G_{ana}) = - \Sigma^1(G_{gt}) 
\quad \text{and} \quad
\Sigma^{BNS}(G_{ana}) = - \Sigma^{BNS}(G_{gt}).
\]
The analogous formula holds for the invariant $\Sigma$ studied in \cite{Bro87b}.

The two parts of the comment, taken together, 
lead to the following formulae for groups made up of bijections:
\begin{align}
\text{$G_{gt}$ arbitrary } &\Longrightarrow \Sigma(G_{gt})= \Sigma^1(G_{ana}),
\label{eq:Sigma-related-to Sigma1}\\
\text{$G_{gt}$ is finitely generated } &\Longrightarrow \Sigma^{BNS}(G_{gt}) = \Sigma^1(G_{ana}).
\label{eq:SigmaBNS-related-to Sigma1} 
\end{align}
\end{remarks}
%
%-----------------
\subsection{Proof of Theorem \ref{thm:Generalization-GoKo10-bis}}
\label{ssec:Proof-Theorem-Generalization-GoKo10-bis}
%-----------------
Let $I = [0, b]$ be an interval of positive length 
and $G$ a subgroup of $\PL_o(I)$
that satisfies hypotheses (i) through (iv) stated in Theorem 
\ref{thm:Generalization-GoKo10-bis}.
Hypotheses (i), (ii) and (iii) allow one to invoke Theorem \ref{thm:Generalization-BNS}
and so $\Sigma^1(G)^c = \{[\chi_\ell], [\chi_r] \}$.
In view of the remarks made at the end of section \ref{ssec:Review-Sigma1},
every automorphism $\alpha$ of $G$ will therefore permute the set $\{[\chi_\ell], [\chi_r] \}$.
Two cases now arise, depending on whether or not the automorphism group of $G$
acts by the identity on $\Sigma^1(G)^c$.

Suppose first that $\Aut G$ \emph{acts trivially on} $\Sigma^1(G)^c$.
By hypothesis (iv),  one of the characters $\chi_\ell$ and $\chi_r$ , 
say $\chi_\ell$, 
has an image $B$ with $U(B) = \{1,-1\}$.
We assert that $\chi_\ell$ is fixed by $\Aut G$.
Consider an automorphism $\alpha$ of $G$.
It fixes the ray $\R_{>0} \cdot \chi_\ell$ 
and so $\chi_\ell \circ \alpha = s \cdot \chi_\ell$ for some positive real $s$.
The relation $\chi_\ell \circ \alpha = s\cdot \chi_\ell$ implies next
that
\[
 \im \chi_\ell= \im (\chi_\ell \circ \alpha) = s \cdot \im \chi_\ell. 
 \]
 So $s$ is a positive element of $U(\im \chi_\ell) = \{1, -1\}$ 
 and thus $s= 1$.

 So far we have assumed that $U(\chi_\ell)$ equals $\{1, -1\}$;
 if  $U(\im \chi_r)$ is so, one proves in the same way 
 that $\chi_r$ is fixed by $\Aut G$.
 The homomorphism $\psi \colon G \to \R^\times_{>0}$ 
 can thus be chosen to be $\sigma_\ell$ if $U(\im \chi_\ell) = \{1, -1\}$
 and  to be $\sigma_r$ if $U(\im \chi_r)= \{1, -1\}$.
 \smallskip
 
 Assume now that $\Aut G$ \emph{interchanges the points} $[\chi_\ell]$ and $[\chi_r]$.
 Let $\Aut_+G$ denote the subgroup of $\Aut G$ that fixes these points
 and pick an automorphism, say $\alpha_-$, that interchanges them.
 Then $\chi_r \circ \alpha_-= s \cdot \chi_\ell$ for some positive real $s$
 and so $\im \chi_r = s \cdot \im \chi_\ell$. 
 This relation implies that $U(\im \chi_\ell) = U(\im \chi_r) = \{1,-1\}$. 
 
 We claim that the homomorphism 
 \[
 \psi = \sigma_\ell \cdot (\sigma_\ell \circ \alpha_-) = \sigma_\ell \cdot (s \cdot \sigma_r)
 \]
 is fixed by $\Aut G$.
 Two cases arise.
 If $\alpha \in \Aut_+ G$ 
 then $\sigma_\ell$  is fixed by $\alpha$ in view of the first part of the proof.
 Moreover, $\alpha' = \alpha_- \circ \alpha \circ (\alpha_-)^{-1} \in \Aut_+ G$ 
 and so the calculation 
 \[
 \psi \circ \alpha 
 = 
 (\sigma_\ell \circ \alpha) \cdot (\sigma_\ell \circ \alpha_-) \circ \alpha
 =
 \sigma_\ell \cdot (\sigma_\ell \circ \alpha') \circ \alpha_- 
 =
 \sigma_\ell \cdot (\sigma_\ell \circ \alpha_-)  
 =\psi
 \]
 holds.
 If $\alpha = \alpha_-$ then $\alpha_-^2 \in \Aut_+ G$ 
 and so 
 $
 \psi \circ \alpha_- = (\sigma_\ell \circ \alpha_-) \cdot (\sigma_\ell \circ \sigma_-^2) = \psi.
 $
It follows that $\psi$ is fixed by $\Aut_+ G \cup \{\alpha_-\}$ and hence by $\Aut G$.
% 
 %-----------------
\subsection{Discussion of the hypotheses of Theorem \ref{thm:Generalization-GoKo10-bis}}
\label{ssec:Discussion-hypotheses-Generalization-GoKo10-bis}
%-----------------
This section and the next two
contain various remarks on the hypotheses of Theorem 
 \ref{thm:Generalization-GoKo10-bis}.
%
%--------
\subsubsection{Irreducibility}
\label{ssec:Discussion-irreducibility}
%--------
Let $G$ be a subgroup of $\PL_o([0,b])$.
The union of the supports of the elements of $G$ is then an open subset of $I = [0,b]$,
and hence a union of disjoint intervals $J_k $ for $k$ running over some index set $K$.
For each $k \in K$ 
the assignment $g \mapsto g \restriction{J_k}$ defines an epimorphism $\pi_k$ 
 onto a quotient group $G_k$ and so $G$ itself is isomorphic to a subgroup of the cartesian product
 $\prod \{G_k \mid k \in K\}$; 
 more precisely,
 $G$ is a subdirect product of the quotient groups $G_k$.
 Hypothesis  (i) requires that $K$ be a singleton,
and so the group $G$ does not admit such obvious decompositions.
 This fact prompted the authors of \cite{BNS} to call a group $G$ \emph{irreducible}
 if  $\card(K) = 1$. 
 
 If the group $G$ is not irreducible it may be a direct product $G_1 \times G_2$ 
 with each factor $G_k$ an irreducible subgroup of $\PL_o(I_k)$ where $I_k$ is the closure of $J_k$.
 Then $\Sigma^1(G)^c$ can contain more than 2 points
 (for more details, see \cite[Section 4.1]{Str15}).
 %
 
 %--------
\subsubsection{Non-triviality of the characters $\chi_\ell$ and $\chi_r$}
\label{ssec:Discussion-non-triviality}
%--------
%
In Theorem \ref{thm:Generalization-GoKo10-bis} 
the characters  $\chi_\ell$ and $\chi_r$ are assumed to be non-zero.
They represent therefore points of $S(G)$;
the remaining hypotheses and Theorem \ref{thm:Generalization-BNS}
then guarantee that $\Sigma^1(G)^c = \{[\chi_\ell], [\chi_r]\}$ 
and so every automorphism of $G$ must permute the points $[\chi_\ell]$ and $[\chi_r]$.

There exists a variant of Theorem \ref{thm:Generalization-GoKo10-bis}
in which only one of the characters, say $\chi_\ell$, is non-zero, 
the remaining hypotheses being as before.
Then $\Sigma^1(G)^c = \{[\chi_\ell]\}$ (see \cite[Theorem 1.1]{Str15})
and so the argument in the first part of the proof of Theorem \ref{thm:Generalization-GoKo10-bis}
applies and shows that  $\psi = \chi_\ell$ is fixed by every automorphism of $G$.

Note that hypothesis (iii) holds automatically if $\chi_\ell$ or $\chi_r$ vanishes.
%--------
\subsubsection{Almost independence of $\chi_\ell$ and $\chi_r$}
\label{ssec:Discussion-independence}
%--------
Among the assumptions of Theorem 8.1 in \cite{BNS}, 
a sharper form of hypothesis (iii) is listed,
namely $G = \ker \chi_\ell \cdot \ker \chi_r$;
in addition, $G$ is assumed to be finitely generated.
The authors of \cite{BNS} refer to his stronger condition 
by saying that ``$\chi_\ell$ and $\chi_r$ are independent''.
In what follows,
we exhibit various versions of this stronger requirement 
and explain then the reason 
that led the authors of \cite{BNS} to adopt the mentioned locution.

We start out with a general result.
\begin{lem}
\label{lem:Equivalence-conditions}
Let $\psi_1 \colon G  \epi H_1$ and $\psi_2 \colon G \epi H_2$ be epimorphisms of groups.
Then the following statements imply each other:
\begin{align*}
&\text{(i) }H_1 = \psi_1(\ker \psi_2), &\quad
&(\text{ii) }H_2 = \psi_2(\ker \psi_1),\\
&\text{(iii) }G = \ker \psi_1 \cdot \ker \psi_2, &\quad
&\text{(iv) } (\psi_1, \psi_2) \colon G \to H_1 \times H_2 \text{ is surjective}.
\end{align*}
\end{lem}

\begin{proof}
Note first that the product $\ker \psi_1 \cdot \ker \psi_2$ is a normal subgroup of $G$. 
Next, $\psi_1$ maps $G$ onto $H_1$ and $\ker \psi_1 \cdot \ker \psi_2$ onto $\psi_1(\ker \psi_2)$
and induces thus an isomorphism
\begin{equation}
\label{eq:Equivalence-i-and-iii}
(\psi_1)_*\colon G/(\ker \psi_1 \cdot \ker \psi_2) \iso H_1/\psi_1(\ker \psi_2).
\end{equation}
It follows, in particular,
that statements (i) and (iii) are equivalent.
By exchanging the rôles of the indices 1 and 2, 
one sees that statements (ii) and (iii) are equivalent.

Assume now that statements (i) and (ii) hold 
and consider $(h_1, h_2) \in H_1 \times H_2$.
Since $\psi_1$ is surjective, $h_1$ has a preimage $g_1 \in G$;
as statement (i) holds, this preimage can actually be chosen in $\ker \psi_2$.
If this is done, one sees that $(\psi_1,\psi_2)(g_1) = (h_1, 1)$.
One finds similarly 
that there exists $g_2 \in \ker \psi_1$ with $(\psi_1, \psi_2)(g_2) = (1, h_2)$.
The product $g_1 \cdot g_2$ is therefore a preimage of $(h_1, h_2)$ under $(\psi_1,\psi_2)$.

The preceding argument  proves that the conjunction of (i) and (ii) implies statement (iv).
Assume, finally, that (iv) holds.
Given $h_1 \in H_1$, there exists then $g_1 \in G$ with $(\psi_1,\psi_2)(g_1) = (h_1, 1)$;
so $g_1$ is a preimage of $h_1$ lying in $\ker \psi_2$.
Implication (iv) $\Rightarrow$ (i) is thus valid,
and so the proof is complete.
\end{proof}

%Here is a useful consequence of the isomorphism \eqref{eq:Equivalence-i-and-iii}:
%\begin{crl}
%\label{crl:Characterization-of-almost-independence}
%%
%The notation being as in Lemma \ref{lem:Equivalence-conditions},
%the group $G/(\ker \psi_1 \cdot \ker\psi_2)$ is finite if, and only if, 
%$H_1/\psi_1 (\ker\psi_2)$ is so.
%Similarly, $G/(\ker \psi_1 \cdot \ker\psi_2)$ is a torsion group precisely if 
%$H_1/\psi_1 (\ker\psi_2)$ is so.
%\end{crl}

\begin{remark}
\label{remark:Origin-independent}
Lemma  \ref{lem:Equivalence-conditions} allows one to understand 
why the locution ``$\chi_\ell$ and $\chi_r$ are independent'' is used in \cite{BNS}
to express the requirement that  $G = \ker \chi_\ell \cdot \ker \chi_r$,
the group $G$ being a finitely generated, irreducible subgroup of $\PL_o([0,b])$. 
Let $\psi_1$ denote the epimorphism  $ G \epi \im \chi_\ell$ 
obtained by restricting the domain of $\chi_\ell \colon G \to \R$ to $\im \chi_\ell$,
and let $\psi_2$ be defined analogously.
If  statement (iii) holds 
implication $(iii) \Rightarrow (iv)$ of Lemma \ref{lem:Equivalence-conditions} 
shows
that the image of  $(\chi_\ell, \chi_r) \colon G \to \R_{\add} \times \R_{\add}$ 
is $\im \chi_\ell \times \im \chi_r$.
This fact amounts to say 
that the values of the characters $\chi_\ell$ and $\chi_r$ can be prescribed \emph{independently}
(within $\im \chi_\ell \times \im \chi_r$),
in contrast to what happens, for instance,  
if the characters satisfy a relation like $\chi_2 = -\chi_1$.
\footnote{Example  \ref{Example2-for-main-result-I-compact} considers more general relations.}

By analyzing the proof of Theorem 8.1 in \cite{BNS} one finds 
that it suffices to require that the normal subgroup $\ker \chi_\ell \cdot \ker \chi_r$ 
has finite index in the finitely generated group $G$;
a condition that we shall paraphrase by saying 
that $\chi_\ell$  and $\chi_r$ are \emph{almost independent}.
Theorem \ref{thm:Generalization-BNS}  extends this result to possibly infinitely generated groups $G$;
the new form of hypothesis (iii) will likewise be referred to by saying 
that $\chi_\ell$  and $\chi_r$ are \emph{almost independent}.
This form of almost independence is used in the proof Theorem \ref{thm:Generalization-BNS}
to find commuting elements of a certain type (see, \eg{}\cite[Section 3.3]{Str15}).
It remains unclear what $\Sigma^1(G)^c$ looks like
if $\chi_\ell$ and $\chi_r$ are not almost independent.
\footnote{Sections 4.2 and 4.3 in \cite{Str15} have some preliminary results.}
\end{remark}
%
%--------
\subsection{Group of units}
\label{ssec:Group of units}
%--------
In section \ref{ssec:Generalization-approach-GoKo10},
the \emph{group of units} $U(B)$ of a subgroup $B$ of $\R_{\add}$ is introduced.
This notion allows one to state a very simple condition 
that implies,
in conjunction with the hypotheses of Theorem \ref{thm:Generalization-BNS},
that $\Aut G$ fixes the character $\chi_\ell$
if it fixes the ray $[\chi_\ell] = \R \cdot \chi_\ell$.

In this section,
we discuss the group of units of some concrete examples of subgroups $B$ of $\R_{\add}$
and study then two types of subgroups $B$ of $\R_{\add}$
where methods taken from the Theory of Transcendental Numbers allow one to establish 
that $B$ has only trivial units.
%
%-----------------
\subsubsection{Elementary examples}
\label{sssec:Group-of-units-elementary-examples}
%-----------------
%
We begin with an observation:
 \emph{a subgroup $B$ and a non-zero real multiple $s \cdot B$ of $B$ have the same group of units}.
If $B$ is not reduced to 0, we may therefore assume that  $1 \in B$.

a)
If $B$ is infinite cyclic,  it is a positive multiple of $\Z$. 
Clearly $U(\Z) = \{1,-1\}$.

b) If $B$ is free abelian of rank 2, we may assume that it is generated by 1 
and an irrational number $\vartheta$;
so $B = \Z \cdot 1 \oplus \Z \cdot \vartheta$.
If $u$ is a unit of $B$ then $u = u \cdot 1 \in B$, 
say $u = a  + b \cdot \vartheta$ with $(a,b) \in \Z^2$.
The condition $u \cdot B \subseteq B$ implies next 
that $u \cdot \vartheta = a \cdot \vartheta + b \cdot \vartheta^2$ lies in $B$.
If $b \neq 0$, the real $\vartheta $ is thus a quadratic algebraic number;
if $b = 0$, the condition that $u \cdot B = B$ 
forces $a$ to 1 or $-1$.
It follows that $U(B) = \{1, -1\}$ 
if $\vartheta$ is an irrational,
but not a quadratic algebraic number.

c) Let $B$ be the additive group of a subring $R$ of $\R$,
for instance the additive group of the ring $\Z[P]$ generated 
by a subgroup $P $ of $\R^\times_{>0}$ 
or of a ring of algebraic integers.
Then $U(B)$ is nothing but the group of units $U(R)$ of $R$;
if $R$ is a ring of the form $\Z[P]$ its group of units contains, of course, $P \cup -P$,
but it may be considerably larger;
moreover,
rings of algebraic integers have also often units of infinite order.
Note, however, that not every subring $R \neq \Z$ of $\R$ has non-trivial units,
an example being the polynomial ring $\Z[s]$ generated by a transcendental number $s$.
%
%-----------------
\subsubsection{Transcendental subgroups}
\label{sssec:Transcendental-subgoups}
%-----------------
%
Many familiar examples of subgroups of $\PL_o([0,b])$
consists of PL-homeomorphisms with rational slopes;
this is true for Thompson's group $F$, 
but also for its generalizations  $G_m =G([0,1];\Z[1/m], \gp(m))$ with $m \geq 3$ an integer
and for many of the groups studied in Stein's paper \cite{Ste92}.

The values of the characters $\chi_\ell$ are then natural logarithms of rational numbers,
and so transcendental numbers or 0
(see, \eg\ \cite[Theorem 9.11c]{Niv56}).
We are thus led to study the unit groups $U(B)$ of subgroups $B \subset \R$ 
that contain transcendental numbers;
in view of the fact that $U(B) = U(s \cdot B)$ for every $s \neq 0$,
it is not so much the nature of the elements of $B$ that is important,
but the nature of the quotients $b_1/b_2$ of non-zero elements in $B$.
The following definition singles out a class of subgroups $B$
that turn out to be significant.
\begin{definition}
\label{definition:Transcendental}
\begin{enumerate}[a)]
\item Let $B \neq \{0\}$ be a subgroup of the additive group $\R_{\add}$ of the reals.  
We say $B$ is \emph{transcendental} if, 
for each ordered pair $(b_1, b_2)$ of non-zero elements  in $B$, 
the quotient $b_1/b_2$ is either rational or transcendental.

\item We call a non-zero character $\chi \colon G\to \R$ \emph{transcendental} 
if its image in $\R_{\add}$ is transcendental. 
\end{enumerate} 
\end{definition}

The next result explains 
why transcendental subgroups are welcome in our study.
\begin{prp}
\label{prp:U-of-transcendental-B-1}
If $B$ is a non-trivial, finitely generated, transcendental subgroup of $\R_{\add}$ 
then $U(B) = \{1,-1\}$.
\end{prp}
\begin{proof}
Suppose $u$ is a unit of $B$. 
Then $u \cdot B = B$.
Pick $b \in B \smallsetminus {0}$; 
this is possible since $B$ is not reduced to 0. 
The assignment $1 \mapsto b$ extends to a homomorphism  $\Z[u] \to  B$ of $\Z[u]$-modules;
it is injective since $\R$ has no zero-divisors.
The fact that $B$ is finitely generated implies  next
that the additive group of the integral domain $\Z[u]$ is finitely generated 
and so $u$ is an algebraic integer;
as $B$ is transcendental by assumption,
$u$ must therefore be an algebraic integer and also a rational number,
hence an integer.
Finally,
$u^{-1}$ satisfies also the relation $u^{-1} \cdot B = B$,
and so $u^{-1}$ is an integer, too.
\end{proof}

We continue with a combination of Theorem \ref{thm:Generalization-GoKo10-bis}
and Proposition \ref{prp:U-of-transcendental-B-1}.
\begin{crl}
\label{crl:Generalization-GoKo10-bis}
Suppose $I = [0,b]$ is a compact interval of positive length
and $G$ is subgroup of $PL_o(I)$ 
that satisfies the following conditions: 
\begin{enumerate}[(i)]
\item  no interior point of the interval $I = [0, b]$ is fixed by $G$;
\item the characters $\chi_\ell$ and $\chi_r$ are both non-zero;
\item the quotient group  $G/(\ker \chi_\ell \cdot \ker \chi_r)$ is a torsion group $G$, and
\item the image of $\sigma_\ell$ or that of $\sigma_r$ is finitely generated and transcendental. 
\end{enumerate}
Then there exists a non-zero homomorphism $\psi \colon G \to \R^\times_{>0}$
that is fixed by every automorphism of $G$.
\end{crl}

%-----------------
\subsubsection{Examples of transcendental subgroups of $\R_{\add}$}
\label{sssec:Examples-transcendental-subgoups}
%-----------------
%
In order to make use of Proposition \ref{prp:U-of-transcendental-B-1},
one needs a supply of transcendental subgroups of $\R$.
The simplest ones are the cyclic subgroups;
non-cyclic subgroups are harder to come by.

Example \ref{example:logarithm-1} below describes a first collection of transcendental subgroups.
It is based on the following theorem, established independently
by A. O. Gelfond in 1934 and by T. Schneider in 1935:
\begin{theorem}[Gelfond-Schneider Theorem]
\label{thm:Gelfond-Schneider}
If $p_1$ and $p_2$ are non-zero (real or complex) algebraic numbers and if $p_2 \neq1$ 
then $\ln p_1/ \ln p_2$ is either a rational or a transcendental number.
\end{theorem}

\begin{proof}
See, \eg{}Theorem 10.2 in \cite{Niv56}.
\end{proof}

\begin{example} 
\label{example:logarithm-1}
Let $P$ denote a subgroup of $\R^\times_{>0}$
generated by a set $\PP$ of algebraic numbers
and define $B = \ln P $ to be its image in $\R_{\add} $ under the natural logarithm.
Then every element in $P$ is a positive algebraic number
and so the Gelfond-Schneider Theorem implies 
that every quotient $\ln p_1 /\ln p_2$ of elements in 
$P \smallsetminus \{1\}$ is either rational or transcendental.  
\end{example}

In Example \ref{example:logarithm-1} the set $\PP$ is allowed to be infinite;
for such a choice,
the group $B = \ln(\gp(\PP))$ is not finitely generated 
and so neither Proposition \ref{prp:U-of-transcendental-B-1}
nor its Corollary \ref{crl:Generalization-GoKo10-bis} applies.
Now, in Proposition \ref{prp:U-of-transcendental-B-1} 
the finite generation of $B$ is only used to infer that a unit $u$ of $B$,
which, by the transcendence of $B$, is either rational or transcendental,
is also an algebraic integer, and hence a rational integer.

Proposition \ref{prp:U-of-transcendental-B-2} below
furnishes examples of infinitely generated, transcendental groups
that have only 1 and $-1$ as units.
Its proof makes use of the following result,
due to C. L. Siegel and rediscovered by S. Lang 
(see \cite[Theorem II.1]{Lan66} or \cite[Theorem (1.6)]{Lan71}):
\begin{theorem}[Siegel-Lang Theorem]
\label{thm:Siegel-Lang}
Suppose $\beta_1$, $\beta_2$ and $z_1$, $z_2$, $z_3$ are non-zero complex numbers.
If the subsets $\{\beta_1, \beta_2\}$ and $\{z_1, z_2, z_3\}$ are both $\Q$-linearly independent
then at least one of the six numbers 
\[
\exp(\beta_i \cdot z_j) \text{ with } (i,j) \in \{1,2\} \times \{1,2,3\}
\]
is transcendental.
\end{theorem}

Here then is the announced result:
\begin{prp}
\label{prp:U-of-transcendental-B-2}
Let $\PP$ be a set of positive algebraic numbers
and set $B = \ln \gp(\PP)$.
If $B$ is free abelian of positive rank then $U(B) = \{1,-1\}$.
\end{prp}
\begin{proof}
Note first that every element of $P = \gp(\PP)$ is a positive algebraic number.
Consider now a unit $u$ of $B$.
Since $B$ has positive rank, it contains a non-zero element $b_1 = \ln q_1$.
Then $u \cdot b \in B\smallsetminus \{0\}$;  
so  $b_2 = u \cdot b_1$ has the form $\ln q_2$
and thus $u$ is either rational or transcendental (by the Gelfond-Schneider Theorem).

Assume first that $u$ is rational, say $u = m/n$ 
where $m$ and $n$ are relatively prime integers.
The hypothesis $(m/n) \cdot B = B$ implies then that $m B = n B$.
As $B$ is free abelian of positive rank
this equality can only hold if $|m| = |n| = 1$. 
So $u \in \{1, -1\}$.

Assume now
that $u$ is transcendental.
Fix $p \in P \smallsetminus \{1\}$.
Then $u \cdot \ln p \in B = \ln P$; so there exists $q \in P$ with $\ln q = u \cdot \ln p$;
put differently, $\exp (u \cdot \ln p)$ lies in $P$ and is thus an algebraic number.
As the powers of $u$ are again units of $B$ 
it follows that $\exp(u^\ell \cdot \ln p) \in P$ for every $\ell \in \N$.
Set
\[
\beta_1 = \ln p, \quad \beta_2  = u \cdot \ln p
\quad\text{and }\quad
z_j =u^j \text{ for } j = 1, 2, 3.
\]
Then the sets $\{\beta_1, \beta_2\}$ and $\{z_1, z_2, z_3\}$ 
fulfill the hypotheses of Theorem \ref{thm:Siegel-Lang};
its conclusion, however, is contradicted by the previous calculation.
This state of affairs shows that the unit $u$ cannot be transcendental.
\end{proof}

\begin{example} 
\label{example:logarithm-2}
Let $\PP$ be a non-empty set of (rational) prime numbers 
and let $P$ denote the subgroup of $\Q^\times_{>0}$ generated by $P$.
Then $P$ is free abelian with basis $\PP$ 
(by the unique factorization in $\N^\times_{>0}$)
and so $U(\ln P) = \{1,-1\}$.

More generally, every non-trivial subgroup $P$ of  $\Q^\times_{>0}$ is a free abelian group 
and hence $B = \ln P$ has only the units $1$ and $-1$.
\end{example}
%
%A third collection of transcendental subgroups $B$ with $U(B) = \{1, -1\}$
%will be constructed in Section \ref{sec:????}.
%
%------------
\subsubsection{Some properties of transcendental subgroups and transcendental characters}
\label{ssec:Properties-transcendental-subgroups}
%----------
%
The transcendence of a character is a property 
that has not yet been discussed in the literature on the invariant $\Sigma^1$.  
In this section, we assemble therefore a few useful properties of this notion.

Assume $B \subset \R_{\add}$ is a transcendental subgroup.
Then 
\begin{enumerate}[a)]
\item every non-trivial subgroup $B' \subseteq B$ is transcendental 
(immediate from the definition);
\item if $\chi \colon G \to \R$ is a character whose image is a non-trivial subgroup of $B$ 
then $\chi$ is transcendental (by (a)),
and so are all the compositions $\chi \circ \pi$ 
with $\pi \colon \tilde{G} \epi G$ an epimorphism of groups
(immediate from the first part);
\item if $\chi$, $\chi'$ are characters of $G$ with images equal to $B$,
 the image of $\chi + \chi'$ is contained in $B$, 
 and so the character $\chi+ \chi'$ is transcendental, unless it is 0;
 \item if $\chi$ is transcendental and $\alpha_1$, \ldots, $\alpha_m$ are automorphisms of $G$
 the character
 \[
 \eta = \chi \circ \alpha_1 + \cdots + \chi \circ \alpha_m
 \]
 is transcendental, unless it is zero.
\end{enumerate}
A further property is discussed in part (iv) of Proposition \ref{prp:Passage-to-finite-index} below.
%
%------------
\subsection{Passage to subgroups of finite index}
\label{ssec:Subgroups-of-finite-index}
%----------
%
The next proposition shows that the hypotheses stated in Corollary
\ref{crl:Generalization-GoKo10-bis}
are inherited by subgroups of finite index.
\begin{prp}
\label{prp:Passage-to-finite-index}
Let $G$ be a subgroup of $\PL_o([0,b])$ and $H \subseteq G$ a subgroup of finite index.
Denote the restrictions of $\chi_\ell$ and $\chi_r$ to $H$ by $\chi'_\ell$ and $\chi'_r$.
Then the following statements are valid:
\begin{enumerate}[(i)]
\item  $G$ is irreducible if, and only if, $H$ is so;
\item $\chi_\ell$ is non-zero precisely if $\chi'_\ell$ has this property,
and similarly for $\chi_r$ and $\chi'_r$;
\item the characters $\chi_\ell$ and $\chi_r$ are almost independent 
if, and only if,  $\chi'_\ell$ and $\chi'_r$ have this property;
\item $\chi_\ell$ is transcendental exactly if $\chi'_\ell$ is so,
and similarly for $\chi_r$ and $\chi'_r$.
\end{enumerate}
\end{prp}

\begin{proof}
Claim (i) holds since the support of a PL-homeomorphism $f$ 
coincides with that of its positive powers $f^m$.
Assertion (ii) is valid 
since the image of a character is a subgroup of $\R_{\add}$ and hence torsion-free.
The fact that the quotient $b_1 /b_2$ of non-zero  real numbers coincides, 
for every positive integer $m$,
with the quotient $(mb_1)/(mb_2)$ allows one to see 
that  a non-zero character $\chi$ of $G$ is transcendental 
if its restriction to $H$ is so;
the converse is covered by property a) stated in section \ref{ssec:Properties-transcendental-subgroups}.
We are left with establishing statement (iii).

To achieve this goal,
we compare the quotient groups 
$G/(\ker \chi_\ell \cdot \ker \chi_r)$ and $H/(\ker \chi'_\ell \cdot \ker \chi'_r)$.
By formula \eqref{eq:Equivalence-i-and-iii},
the first of them is isomorphic to the quotient group $A_1 =\im \chi_\ell/\chi_\ell( \ker \chi_r)$,
the second one is isomorphic to $A_2=\im \chi'_\ell/\chi'_\ell(\ker \chi'_r)$.
Clearly $A_2 = \im \chi'_\ell/\chi_\ell( \ker \chi'_r)$.
The groups $A_1$ and $A_2$ fit into the short exact sequences
\begin{align}
A_2 =  \im \chi'_\ell/\chi_\ell( \ker \chi'_r)  
&\incl A = \im \chi_\ell/ \chi_\ell (\ker \chi'_r)
\epi 
\im \chi_\ell/ \im \chi'_\ell,
\label{eq:Extending-A2}\\
\chi_\ell(\ker \chi_r)/ \chi_\ell(\ker \chi'_r)  
&\incl A = \im \chi_\ell/ \chi_\ell (\ker \chi'_r)
\epi 
A_1  =\im \chi_\ell/\chi_\ell( \ker \chi_r).
\label{eq:Mapping-onto-A1}
\end{align}
The claim now follows from the fact 
that $\im \chi_\ell/ \im \chi'_\ell$ and $\chi_\ell(\ker \chi_r)/ \chi_\ell(\ker \chi'_r)$ are finite groups 
with orders that divide the index of $H$ in $G$. 
\end{proof}

A first application of Proposition \ref{prp:Passage-to-finite-index} is
\begin{crl} 
\label{crl:Commensurable-groups}
Let $G$ be a finitely generated, irreducible subgroup of $PL_o(I)$. 
If the characters $\chi_\ell$ and $\chi_r$ are almost independent 
and one of them is transcendental,
then any group $\Gamma$ commensurable
\footnote{Two groups $G_1$ and $G_2$ are called \emph{commensurable} 
if they contain subgroups $H_1$, $H_2$ 
that are isomorphic and of finite indices in $G_1$ and in $G_2$, respectively.}
 with $G$ has property $R_\infty$.  
\end{crl}

\begin{proof}  
Let $H_0\subset G$ be a finite index subgroup of $G$
that is isomorphic to a finite index subgroup $\Gamma_0$ of $\Gamma$.  
There exists then a \emph{finite index subgroup $\Gamma_1$ of $\Gamma_0$
that is characteristic in $\Gamma$}
(see, \eg{}\cite[Theorem IV.4.7]{LS77}).
Let $H_1$ be the subgroup of $H_0$ that corresponds to $\Gamma_1$ 
under an isomorphism  $H_0 \iso \Gamma_0$.
Then $H_1$ has finite index in $G$
and thus Proposition \ref{prp:Passage-to-finite-index} allows us to infer 
that $H_1$ inherits the properties enunciated for $G$ in the statement of 
Corollary \ref{crl:Generalization-GoKo10-bis}.
This corollary applies therefore to $H_1$ 
and shows 
that $H_1$ admits a non-zero homomorphism $\psi_1 \colon H_1 \to \R^\times_{>0}$
that is fixed by $\Aut H_1$.
So $H_1$, and hence $\Gamma_1$, satisfy property $R_\infty$.
Use now that $\Gamma_1$ is a characteristic subgroup of $\Gamma$
and apply \cite[Lemma 2.2(ii)]{MuSa14a} to infer 
that $\Gamma$ satisfies property $R_\infty$. 
\end{proof}  

\begin{remark}
\label{remark:Commensurability-and-Reidemeidster-number}
If the group $G_1$  has property $R_\infty$, 
then a group $G_2$ commensurable to $G_1$ need not have this property,
as is shown by the group $G_1$ of the Klein bottle and the group $G_2$ of a torus:
the group $G_1$ has property $R_\infty$ by \cite[Theorem 2.2]{GoWo09}),
while the automorphism $-\id$ of $G_2 = \Z^2$ has Reidemeister number 4.
\end{remark}
%\newpage

%
%\newpage
%
%==========
\section{Miscellaneous examples}
\label{sec:Miscellaneous-examples}
%==========
%
In this section
we illustrate the notions of irreducible subgroup,
of almost independence of $\chi_\ell$ and $\chi_r$
and of the group of units by various examples.
%
%------------
\subsection{Irreducible subgroups}
\label{ssec:Irreducible-subgroups}
%----------
Let $b$ be a positive real number and $G$ a subgroup of $\PL_o([0,b])$. 
Recall that $G$ is called \emph{irreducible} if no interior point of $I = [0,b]$ is fixed by all of $G$ 
(see section \ref{ssec:Discussion-irreducibility}
for the motive that led to this name).

The group is irreducible if, and only if, the supports of the elements of $G$ 
cover the interior $\Int(I)$ of $I$,
or, equivalently, 
if the supports of the elements in a generating set $\XX$ of $G$ cover $\Int(I)$;
these claims are easily verified.
If $G$ is cyclic, generated by $f$, say,
it is therefore irreducible if $f$ fixes no point in $\Int(I)$
or, equivalently if, $f^\varepsilon(t) < t$ for $t \in \Int(I)$ and some sign $\varepsilon$.
Such a function is often called a \emph{bump}.

\begin{example}
\label{eq:One-bump}
Here is a very simple kind of PL-homeomorphism bump.
Given a positive slope $s \neq 1$,
set
\begin{equation}
\label{eq:Bump-function}
f_s(t) = 
\begin{cases}
(1/s)t,                                         & \text{ if } 0\le t\le s/(s+1)\cdot b \\ 
s \left(t- \frac{s\cdot b}{s+1}\right) + \tfrac{b}{s+1} , & \text{ if } s/(s+1) \cdot b < t \leq b.
\end{cases}
\end{equation}
Then $f_s$ is continuous at $s/(s + 1) \cdot b$;
since $f_s(0)= 0$ and $f_s(b) = b$,
the function $f_s$ lies in $\PL_o([0,b])$.
Let $G_s$ denote the group generated by $f_s$
and let $\alpha$ be the automorphism that sends $f_s$ to its inverse $ f_s^{-1}$.
Then $(\chi_\ell\circ \alpha)(f_s) = \chi_\ell(f^{-1}_s) =- \chi_\ell(f_s)$; 
similarly $(\chi_r \circ \alpha) (f_s) = -\chi_r(f_s)$,
whence
\begin{equation}
\label{eq:Relation-chi-ell-chi-r}
\chi_\ell \circ \alpha = - \chi_\ell \text{ and } \chi_r \circ \alpha = - \chi_r. 
\end{equation}
So neither  $\chi_\ell$ nor $\chi_r$ is fixed by $\Aut(G_s)$.
Theorem \ref{thm:Generalization-BNS}  cannot be applied,
as requirement (iii) is violated; indeed, $\ker \chi_\ell = \ker \chi_r = \{\id\}$
and so $G_s/(\ker \chi_\ell \cdot \ker \chi_r)$ is infinite cyclic.
The conclusion of Theorem \ref{thm:Generalization-BNS} is likewise false,
for $\Sigma^1(G)^c = \emptyset$ (this follows, \eg\ from Example A2.5a in \cite{Str13}).
 Property $R_\infty$, finally, does not hold, either;
for the Reidemeister number of the automorphism $\alpha$ is 2,
as a simple calculation shows.
\end{example}

The groups in  the previous example are cyclic;
more challenging groups are considered in
 \begin{example}
\label{eq:Several-bumps}
Let $d > 1$ be an integer and $s_1$, \ldots, $s_d$ pairwise distinct, positive real numbers $\neq 1$.
For each index $i \in \{1, \ldots, d\}$, 
define $f_i$ by formula \eqref{eq:Bump-function} with $s = s_i$,
and set 
\[
G = G_{\{s_1, \ldots, s_d\}} = \gp(f_1, \ldots, f_d).
\]
The group $G$ inherits two properties from the group $G_s$ in the previous example:
it is irreducible (obvious),
and the assignment $f_i \mapsto f_i^{-1}$  extends to an automorphism $\alpha$;
indeed, the special form of the elements $f_i$ implies 
that conjugation by the reflection in the mid-point of $I = [0, b]$ sends $f_i$ to its inverse.
It follows, as before, that the relations \eqref{eq:Relation-chi-ell-chi-r} are valid;
so neither $\chi_\ell$ nor $\chi_r$ is fixed by $\Aut G $.

Now to another property of the automorphism $\alpha$.
The calculation
\begin{equation}
\label{eq:Relation-chi-ell-chi-r-bis}
(\chi_\ell \circ \alpha)(f_i) = \chi_\ell (f_i^{-1})  =- \chi_\ell(f_i)= \chi_r(f_i)
\end{equation}
is valid for every index $i$. 
It shows that $\alpha$ exchanges $\chi_\ell$ and $\chi_r$.
It follows, in particular,
that $\ker \chi_\ell = \ker \chi_r$ and so the quotient
\[
G/(\ker \chi_\ell \cdot \ker \chi_r) = G/\ker \chi_\ell \iso \im \chi_\ell = \gp(\ln s_1, \ldots, \ln s_d)
\]
is a non-trivial free abelian group of rank at most $d$.
Requirement (iii) in Theorem \ref{thm:Generalization-BNS}
is thus violated 
and so we cannot use that result to determine $\Sigma^1(G)^c$.
Actually, only the following meager facts are known about $\Sigma^1(G)^c$:
both $\chi_\ell$  and $\chi_r = - \chi_\ell$ represent points of $\Sigma^1(G)^c$
(by \cite[Proposition 2.5]{Str15});
moreover,
the existence and form of the automorphism $\alpha$ 
and formula \eqref{eq:Representation-autos} imply
that $\Sigma^1(G)^c$ is invariant under the antipodal map $[\chi] \mapsto [-\chi]$.

The computation \eqref{eq:Relation-chi-ell-chi-r-bis} shows
that $\chi_\ell \circ \alpha = - \chi_\ell$.
This conclusion holds, actually,  for every character $\chi \colon G \to \R$
and proves that no non-zero character of $G$ is fixed by $\alpha$.
\end{example}
%
%------------
\subsection{Independence of $\chi_\ell$ and $\chi_r$}
\label{ssec:Independence}
%----------
As before,
let $G$ be a subgroup of $\PL_o([0,b])$ with $b$ a positive real number.
Recall that the characters $\chi_\ell$ and $\chi_r$ are called \emph{independent} 
if $G = \ker  \chi_\ell \cdot \ker \chi_r$ (see section \ref{ssec:Discussion-irreducibility}).
It follows that $\chi_\ell$ and $\chi_r$ are independent  if, and only if, 
$G$ admits a generating set $\XX = \XX_\ell \cup \XX_r$ 
in which the elements of $\XX_\ell$ have slope 1 near $b$ 
and those of $\XX_r$ have slope 1 near $0$.

It is thus very easy to manufacture groups for which $\chi_\ell$ and $\chi_r$ are independent.
In the next example,
some very particular specimens are constructed.
\begin{example}
\label{example:Recipe-for-independence}
Choose a real number  $b_1 \in \; ]b/2,b[$.
Given positive real numbers $s_1$, \ldots, $s_{d_\ell}$ 
that are pairwise distinct and not equal to 1,
let $f_i$ be the bump defined by formula \eqref{eq:Bump-function}
but with $s =s_i$ and $b = b_1$.
Next let $s'_1$, \ldots, $s'_{d_r}$  be another sequence of positive reals
that are pairwise distinct and different from 1.
Use them to define bump functions $g_j$ with supports in $]b-b_1, b[$ like this:
let $h_j$ be the function given by formula  \eqref{eq:Bump-function}
but with $s = s'_j$ and $b = b_1$,
and define then $g_j $ to be $h_j$ conjugated by the translation with amplitude $b- b_1$.
Finally set
\begin{equation}
\label{eq:Construction-independent}
G 
= 
G_{\{s_1, \ldots, s_{d_\ell}, s'_1, \ldots s'_{d_r}; b_1 \}}
=
\gp(f_1, \ldots, f_{d_\ell}, g_1, \ldots, g_{d_r} ).
\end{equation}
In the sequel we assume that $d_\ell$ and $d_r$ are positive.
Then $G$ is irreducible (since $b_1 > b-b_1$),
the characters $\chi_\ell$, $\chi_r$ are non-zero and independent, 
and thus Theorem  \ref{thm:Generalization-BNS} allows us to conclude 
that $\Sigma^1(G)^c = \{[\chi_\ell], [\chi_r] \}$.

The character $\chi_\ell$ is transcendental 
if all the positive reals $s_1$, \ldots, $s_{d_\ell}$ are algebraic
(cf.\ Example \ref{example:logarithm-1}).
Then $G$ admits a non-zero homomorphism $\psi \colon G \to \R^\times_{>0}$
that is fixed by $\Aut G$ (see Theorem \ref{thm:Generalization-GoKo10-bis}).
If $G$ does not admit an automorphism $\alpha$ with $(\chi_\ell \circ \alpha) \in [\chi_r]$ 
the homomorphism $\psi$ can be chosen to be $\sigma_\ell$ 
(see the second paragraph of section \ref{ssec:Proof-Theorem-Generalization-GoKo10-bis}).
The stated condition holds, in particular, if there does not exists a number $s$ with 
$\im \chi_r  = s\cdot \im \chi_\ell$.
Similar remarks apply to $\chi_r$.
\end{example}

%--------
\subsubsection{Independence versus almost independence}
\label{sssec:Independence-versus-almost-indepndence}
%-------
%
The characters $\chi_\ell$ and $\chi_r$ are called almost independent
if  $G/(\ker \chi_\ell \cdot \ker \chi_r)$ is a torsion group
(see Remark \ref{remark:Origin-independent}).
Statement (iii) of Proposition \ref{prp:Passage-to-finite-index} shows 
that almost independence of $\chi_\ell$ and $\chi_r$ is inherited
by the restricted characters $\chi'_\ell  = \chi_\ell \restriction{H}$ and $\chi'_r  = \chi_r \restriction{H}$
whenever $H \subseteq G$ is a subgroup of finite index.
The next result characterizes those ordered pairs  $(G, H)$,
with $\chi_\ell$, $\chi_r$ independent
whose restrictions  $\chi'_\ell$ and $\chi'_r$ are again independent.
\begin{lem}
\label{lem:Independence-and-finite-index}
Let $G$ be a subgroup of $\PL_o([0,b])$ for which $\chi_\ell$ and $\chi_r$ are independent
and let $H\subset G$ be a subgroup of finite index.
Then the restrictions $\chi'_\ell$ and $\chi'_r$ of these characters are independent if,
and only if the homomorphism
\begin{equation}
\label{eq:Characterizing-homomorphism}
\zeta \colon \chi_\ell(\ker \chi_r)/ \chi_\ell(\ker \chi'_r)  
\longrightarrow 
\im \chi_\ell/ \im \chi'_\ell,
\end{equation}
induced by the inclusions, is injective.
\end{lem}
\begin{proof}
The justification will be an assemblage of facts extracted 
from the proof of Lemma \ref{lem:Equivalence-conditions}
and from that of Proposition \ref{prp:Passage-to-finite-index}.
Firstly, $\chi_\ell$ and $\chi_r$ are independent if, and only if, 
the abelian group
$A_1 =\im \chi_\ell/\chi_\ell( \ker \chi_r)$ is 0.
Similarly,  
$\chi'_\ell$ and $\chi'_r$ are independent  precisely 
if $A_2 =  \im \chi'_\ell/\chi_\ell( \ker \chi'_r) $ is the zero group.
%Next, our claim is true precisely if the analogous group 
%$A_2 =  \im \chi'_\ell/\chi_\ell( \ker \chi'_r) $ is trivial.
%(This follows from the isomorphism \eqref{eq:Equivalence-i-and-iii}
%and the analogous isomorphism for the restricted maps.)
The groups $A_1$ and $A_2$ occur among the groups in the short exact sequences
\eqref{eq:Extending-A2} and \eqref{eq:Mapping-onto-A1}.
Since $A_1 = 0$, these exact sequences lead to the short exact sequence
\[
\im \chi'_\ell/\chi_\ell( \ker \chi'_r)  
\incl \chi_\ell(\ker \chi_r)/ \chi_\ell(\ker \chi'_r) 
= \im \chi_\ell/ \chi_\ell (\ker \chi'_r)
\epi 
\im \chi_\ell/ \im \chi'_\ell,
\]
It shows that $A_2 = \im \chi'_\ell/\chi_\ell( \ker \chi'_r)$ is the kernel of the homomorphism $\zeta$.
\end{proof}

It is now easy to construct independent characters $\chi_\ell$ and $\chi_r$ of $G$
whose restrictions to a subgroup of finite index are no longer independent.
\begin{example} 
\label{example:Restrictions-need-not-be-independent}
Let $G$ be a subgroup of $\PL_o([0,b])$
and $H$ a subgroup of finite index.
Assume the characters $\chi_\ell$ and $\chi_r$ are independent.
According to Lemma \ref{lem:Independence-and-finite-index}
the restricted characters $\chi'_\ell$ and $\chi'_r$ of $H$ are independent,
if, and only if, the obvious homomorphism
\[
\zeta \colon \chi_\ell(\ker \chi_r)/ \chi_\ell(\ker \chi'_r)  
\longrightarrow 
\im \chi_\ell/ \im \chi'_\ell
\]
is injective.
The characters $\chi'_\ell$ and $\chi'_r$ of $H$ will therefore \emph{not} be independent
whenever
\begin{equation}
\label{eq:Sufficient-condition-not-independent}
\im \chi'_\ell = \im \chi_\ell
\quad \text{but} \quad
\chi_\ell(\ker \chi_r \cap H) \subsetneq \chi_\ell(\ker \chi_r).
\end{equation}

Now to some explicit examples.
We begin with quotients of the groups we shall ultimately be interested in.
Set $\bar{G} = \Z^2$, let $p \geq 2$ be an integer 
and set 
\[
\bar{H} = \Z(p,0) + \Z (1,1).
\]
Then $\bar{H}$ has index $p$ in $\bar{G}$.

Next, let $\chi_1$, $\chi_2$ denote the canonical projections of $\Z^2$ onto its factors.
Then
\[
\chi_1(\bar{G}) = \Z = \chi_1 (\bar{H}),\quad \ker \chi_2  = \Z (1,0)
\text{ and } \ker \chi_2 \cap \bar{H} = \Z(1,0) \cap \bar{H} = \Z(p,0),
\] 
and thus
\[
 \chi_1(\ker \chi_2 \cap \bar{H} ) =  \Z \cdot p 
 \subsetneq \
 \chi_1(\ker \chi_2) = \Z.
\]
The auxiliary groups $\bar{G}$ and $\bar{H}$ satisfy therefore the relations 
\eqref{eq:Sufficient-condition-not-independent}.

We are now ready to define the group $G$;
it will be of the kind considered in Example \ref{example:Recipe-for-independence}
with $d_\ell = d_r = 1$.
Fix $b > 0$ and $b_1 \in \;]b/2, b[$
and choose positive numbers $s_1$, $s'_1$, both different from 1.
Define $f_1$ and $g_1$ as in Example \ref{example:Recipe-for-independence}
and set
\[
G =\gp(f_1, g_1).
\]
Then $G$ is an irreducible subgroup of $\PL_0([0,b])$
and the characters $\chi_\ell$ and $\chi_r$ of $G$ are independent.
Moreover, $G_{\ab}$ is free abelian of rank 2, 
freely generated by the canonical images of $f_1$ and $g_1$.
Set $H = \gp( f_1^p, f_1 \circ g_1, [G,G] )$.
The above calculations then imply that
\[
\chi_\ell(G) = \Z \cdot \ln s_1 = \chi_\ell (H),
\text{ and }
\chi_\ell(\ker \chi_r \cap H ) =  \Z \cdot p \cdot \ln s_1
\subsetneq
\chi_\ell(\ker \chi_r) = \Z \ln s_1.
\] 
\end{example} 
%
%------------
\subsection{Eigenlines}
\label{ssec:Eigenlines}
%----------
Let $G$ be an irreducible subgroup of $\PL_o([0,b])$.
If the characters $\chi_\ell$ and $\chi_r$ are non-zero and almost independent,
then $\Sigma^1(G)^c$ consist of the two points $[\chi_\ell]$ and $[\chi_r]$
(by Theorem  \ref{thm:Generalization-BNS}).
Every automorphism $\alpha$ of $G$ either fixes or exchanges them.
Suppose we are in the first case.
Then $\chi_r \circ \alpha = s \cdot \chi_r$ for some positive real $s$,
and so $\R \cdot \chi_r$ is an eigenline, with eigenvalue $s$, in the vector space
$\Hom(G,\R)$ acted on by $\alpha^*$.
No example with $s \neq1$ has been found so far.

If the compact interval $[0, b]$ is replaced by the half line $[0, \infty[$, 
such examples exist,
provided $\chi_r$ is replaced by a suitable analogue $\tau_r$. 
In order to construct examples,
we return to the set-up of Section \ref{sec:Homomorphisms-fixed-by-Aut-I-half-line}.
So $P$ is a non-trivial subgroup of $\R^\times_{>0}$ 
and $A$ is a non-trivial $\Z[P]$ submodule of $\R_{\add}$.
Define $G$ to be the kernel of the homomorphism 
$\sigma_r \colon G([0,\infty[\,; A, P) \to \R^\times_{>0}$;
thus $G$ consists of all the elements of $G([0,\infty[\,; A, P)$ that are translations near $+\infty$.
The analysis in section \ref{sssec:I-half-line-Groups-with-im-rho-translations} shows
that conjugation by the PL-homeomorphism $f_p \colon \R \iso \R$,
given by $f_p(t) = p \cdot t$ for $t \geq 0$,
induces, for every $p \in P$, 
an automorphism $\alpha_p$ of $G$ that satisfies the relation
\begin{equation}
\label{eq:Transformation-tau-bis}
\tau_r \circ \alpha_p = p \cdot \tau_r;
\end{equation}
here $\tau_r \colon G \to \R_{\add}$ is the character 
that sends $g \in G$ to the amplitude of the translation that coincides with $g$ near $+\infty$.
This character $\tau_r$ shares an important property with the character $\chi_r$: 
the invariant $\Sigma^1(G)^c$ consists of two points,
one represented by $\chi_\ell$, the other by $\tau_r$ (see \cite[Theorem 1.2]{Str15}).

The image of $\tau_r$ in $\R_{\add}$ is a subgroup $B$ of $A$, 
namely 
\[
B = IP \cdot A  = \left\{\sum \, (p-1) \cdot  a \;\big | \; p \in P \text{ and } a \in A \right\}
\]
(cf.\ assertion (iii) of \cite[Corollary A5.3]{BiSt14}).
The group of units $U(B)$ of $B$ contains the group $P$
and so it is not reduced to $\{1, -1\}$.

The subgroup $B$ is typically infinitely generated; 
if so, $G$ is likewise infinitely generated.
Examples of finitely generated groups $G = \ker \sigma_r$ 
are harder to find, and they are so far rare.
Suppose the group $G([0,b]; A, P)$ is finitely generated for some $b \in A_{>0}$.
Then $G([b,2b];A,P)$ is a finitely generated subgroup of  
the group of bounded elements $B([0,\infty\,[;A,P)$. 
Pick now an element $g_0 \in G$ 
that moves every point of the open interval $]0, \infty[$ to the right
and satisfies the inequality $g_0(b) < 2b$.
Then translates of the interval $]b, 2b[$ under the powers of $g_0$ will then cover $]0, \infty[$.
It follows that the subgroup 
\[
N = \gp(\{g_0^j \circ G([b, 2 b];A,P) \circ g_0^{-j} \mid j \in \Z \})
\]
coincides with the bounded group $B([0, \infty[\,;A,P)$ (use \cite[Lemma E18.9]{BiSt14}).
So the group $B([0, \infty[\,;A,P) \rtimes \gp(g_0)$ is finitely generated.
The group $G$, finally,  is finitely generated 
if $G/N \iso \im \tau_r = IP \cdot A$ is finitely generated.

To show that finitely generated groups of the form $G = \ker \sigma_r$ exist
we need thus an example of a group $G([0,b];A, P)$ 
where both $G([0,b]; A, P )$ and the abelian group underlying $B = IP \cdot A$ are finitely generated.
The parameters
\[
P = \gp \left(\sqrt{2\,} +1\right), \quad A = \Z\left[\sqrt{2\,}\right] = \Z[P], \quad b = 1
\]
lead to such a group; see \cite{Cle95}.
%-----------------
\subsection{Variation on Theorem \ref{thm:Generalization-GoKo10-bis}}
\label{ssec:Complement-to-main-theorem}
%-----------------
%
Among the hypotheses of Theorems 
\ref{thm:Generalization-GoKo10-bis}
and \ref{thm:Generalization-BNS} 
figures the requirement that $G$ acts irreducibly on the interval $[0,b]$.
This requirement rules out, in particular, that $G$ is a product $G_1 \times G_2$
with $G_1$ acting irreducibly on some interval $I_1 = [0,b_1]$ and $G_2$ acting irreducibly 
on an interval $I_2= [b_2, b]$ that is disjoint from $I_1$.

Suppose now we are in this excluded case
and that the groups $G_1$, $G_2$ satisfy the assumptions of Theorem \ref{thm:Generalization-BNS}, 
suitably interpreted; 
more explicitly,
suppose the characters $\chi_{1, \ell}$ and $\chi_{1, r}$ of $G_1$ 
are non-zero and almost independent,
and similarly for the characters $\chi_{2, \ell}$ and $\chi_{2, r}$ of $G_2$.
The question then arises whether $G = G_1 \times G_2$ 
admits a non-zero homomorphism $\psi \colon G \to \R^\times_{>0}$ that is fixed by $\Aut G$.
We shall see that this is the case if at least one of the four groups 
$\im \chi_{1, \ell}$, $\im \chi_{1, r}$ and $\im \chi_{2, \ell}$, $\im \chi_{2, r}$ 
has a unit group that is reduced to $\{1,-1\}$.

 The following result is a variation on Theorem 3.2 in \cite{GoKo10}.
\begin{prp} 
\label{prp:Linearly-independent-characters} 
Let $G$ be a group for which $\Sigma^1(G)^c$ is a non-empty finite set with $m$ elements.
Assume the rays $[\chi] \in \Sigma^1(G)^c$ 
span a subspace of $\Hom(G, \R)$ having dimension $m$ over $\R$
and that $U(\im \chi_1) = \{1,-1\}$ for some point $[\chi_1] \in \Sigma^1(G)^c$.
Then $G$ admits a non-trivial homomorphism $\psi \colon G \to \R^\times_{>0}$ 
that is fixed by $\Aut G$.
\end{prp} 

\begin{proof}  
The automorphism group $\Aut G$ acts on $\Sigma^1(G)^c$ 
via the assignment 
\[
(\alpha, [\chi]) \mapsto [\chi \circ \alpha^{-1}];
\]
let $\{ [\chi_1],  \ldots, [\chi_n] \}$ be the orbit in $\Sigma^1(G)^c$ containing $[\chi_1]$.
If $n = 1$, 
the point $[\chi_1]$ is fixed by $\Aut G$; 
hence $\chi_1$ itself is fixed by $\Aut G$ in view of the assumption 
that $U(\im \chi_1) = \{1, -1\}$, 
and so we can take $\psi = \exp \circ \chi_1$.

Suppose now that $n > 1$ and choose, for every $i \in \{1, \ldots, n\}$,
an automorphism $\alpha_i$ with $[\chi_i] = [\chi_1 \circ \alpha_i]$.
Let $\alpha$ be an automorphism of $G$.
For every index $i \in \{1, \ldots, n\}$ there exists then an index $j$ 
so that $[\chi_i \circ \alpha^{-1}] =  [(\chi_1 \circ \alpha_i) \circ \alpha^{-1}]$ 
is equal to $[\chi_j] = |\chi_1 \circ \alpha_j]$.
It follows that there exists a positive real number $s_{i,j}$ so that
\[
\chi_1 \circ \alpha_i \circ \alpha^{-1} = s_{i,j} \cdot \chi_1 \circ \alpha_j.
\]

But if so, $\beta =  \alpha_i \circ \alpha^{-1} \circ \alpha_j^{-1}$ 
is an automorphism with $\chi_1 \circ \beta = s_{i,j} \cdot \chi_1$.
The assumption that $U(\im \chi_1) = \{1,-1\}$ permits one then to deduce that $s_{i,j} = 1$.
So $\Aut G$ permutes the set of characters
\begin{equation}
\label{eq:Representative-characters}
\chi_1 \circ \alpha_1, \quad \chi_1 \circ\alpha_2, \ldots,  \chi_1 \circ\alpha_n.
\end{equation}
Their sum $\eta$ is therefore fixed by $\Aut G$.
It is non-zero since the characters displayed in \eqref{eq:Representative-characters}
are linearly independent over $\R$.
Set $\psi = \exp \circ \eta$.
\end{proof}  

\begin{crl}
\label{crl:Direct-product-of-irreducible-groups}
Let $G_1$ be a subgroup of $\PL_0([0,b_1])$ and let $G_2$ be a subgroup of $\PL_0([b_2, b])$
with $0 < b_1 < b_2 < b$.
Assume $G_1$ and $G_2$ are irreducible, 
the characters $\chi_{1,\ell}$ and $\chi_{1, r}$ of $G_1$ are non-zero and almost independent,
and that the characters $\chi_{2,\ell}$ and $\chi_{2, r}$ of $G_2$ have the same properties.
If the image of at least one of the four characters 
$\chi_{1,\ell}$, $\chi_{1, r}$ and $\chi_{2,\ell}$, $\chi_{2, r}$ has a unit group 
that is reduced to $\{1, -1\}$ 
then $G = G_1 \times G_2$ 
admits a non-trivial homomorphism $\psi \colon G \to \R^\times_{>0}$
that is fixed by $\Aut G$.
\end{crl}

\begin{proof}
The hypothesis on $G_1$ and $G_2$ allow us to apply Theorem \ref{thm:Generalization-BNS} 
and so 
\[
\Sigma^1(G_1)^c = \{[\chi_{1,\ell}], [\chi_{1,r}] \}
\quad \text{and} \quad
\Sigma^1(G_2)^c = \{[\chi_{2,\ell}], [\chi_{2,r}] \}.
\]
The product formula for $\Sigma^1$ then implies that $\Sigma^1(G)^c$ consists of the four points represented by
\begin{equation}
\label{eq:Points-of-product}
\chi_{1,\ell} \circ \pi_1, \quad \chi_{1,r} \circ \pi_1, 
\qquad
\chi_{2,\ell} \circ \pi_2, \quad \chi_{1,\ell} \circ \pi_2; 
\end{equation}
here $\pi_i \colon G \epi G_i$ denotes the canonical projection onto the $i$-th factor $G_i$
(see, \eg\ \cite[Proposition C2.55]{Str13}).
These four characters are $\R$-linearly independent
since all are non-zero, 
as $\ker \chi_{1, \ell} \neq \ker \chi_{1,r}$  
by the almost independence of $\chi_{1,\ell}$  and $\chi_{1,r}$,
as $\ker \chi_{2, \ell} \neq \ker \chi_{2,r}$  
by the almost independence of $\chi_{2,\ell}$  and $\chi_{2,r}$,
and since $\pi_1^*(\Hom(G_1, \R))$ and $\pi_2^*(\Hom(G_2, \R))$ 
are complementary subspaces of $\Hom(G, \R)$.
Finally, at least one of the images of the four characters 
displayed in \eqref{eq:Points-of-product} has an image $B$ with $U(B) = \{1,-1\}$.
All the assumptions of Proposition \ref{prp:Linearly-independent-characters} are thus satisfied
and so the contention of the corollary follows from that proposition.
\end{proof}

\begin{remark}
\label{Direct-product-R-infty}
It is not known 
whether the direct product of groups $G_1$, $G_2$
each of which has property $R_\infty$ has again property $R_\infty$.
The previous corollary implies that this will be so
if the groups $G_1$ and $G_2$ satisfy the assumptions of the corollary.
\end{remark}
%\newpage

%
%======
\section{Complements} 
\label{sec:Complements}
%======
%
By Remark \ref{remark:Continuously-many-non-isomorphic-groups}
there exists continuously many pairwise non-isomorphic, 
finitely generated groups of the form $G(\R;A, P)$,
and by Corollary \ref{crl:G=G(R;A,P)-I-line}
each of these groups admits a non-zero homomorphism $\psi$ into $P$. 
These facts prompt the question
 whether there exist similarly large collections of finitely generated subgroups of $\PL_o(I)$
 with $I$ a compact interval, say $I = [0,1]$.
 Since only countably many \emph{finitely generated} 
 groups of the form $G([0,1]; A, P)$ have been found so far,
 we look for finitely generated groups 
 that satisfy the assumptions of Theorem \ref{thm:Generalization-GoKo10-bis}.
 
 In section \ref{ssec:Uncountable-class} 
 we exhibit a collection $\GG$ of 3-generator groups with the desired properties.
 Checking 
 that each group in $\GG$ satisfies the assumptions  of Theorem \ref{thm:Generalization-GoKo10-bis}
 is fairly easy;
 the verification that distinct groups in $\GG$ are not isomorphic,
 however,
 is more demanding.
 We shall succeed by exploiting properties of the $\Sigma^1$-invariant 
 of the groups in $\GG$ in a roundabout manner.
 In section \ref{ssec:Unexpected-isomorphisms}
 we describe then a collection of 2-generator groups
 which, despite appearances, turn out to be pairwise isomorphic.
 This indicates once more that criteria which allow one to prove 
 that two given, similarly looking, groups are not isomorphic, are very useful.
 In the final section, we give such a criterion.
% 
%-----------
\subsection{A large collection of groups $G$ with characters fixed by $\Aut G$} 
\label{ssec:Uncountable-class}
%-----------
%
In this section we construct a collection $\GG$ of pairwise non-isomorphic groups $G_s$
with the following properties:
\begin{enumerate}[(i)]
\item each $G_s \in \GG$ is an irreducible subgroup of $\PL_o([0,1])$
generated by 3 elements;
\item the characters $\chi_\ell$, $\chi_r$ of $G_s$ 
are independent and have ranks 1, respectively 2;
\item for each $G_s \in \GG$ the character $\chi_\ell$ is fixed by $\Aut G_s$, and
\item the cardinality of $\GG$ is that of the continuum.
\end{enumerate}
%
%------------
\subsubsection{Construction of the groups $G_s$}
\label{sssec:Construction-groups}
%----------
The groups $G_s$ are obtained by the recipe described in 
Example \ref{example:Recipe-for-independence}.
Fix a triple $s = (s_1, s_2= s'_1, s_3= s'_2)$ of real numbers in $]1, \infty[$.
Let $f_s$ be the PL-homeomorphism defined by formula \eqref{eq:Bump-function} 
with $s = s_1$ and $b = \tfrac{3}{4}$. 
Next, let $g$ be the function obtained by putting $s = s_2$, $b =  \tfrac{3}{4}$ 
and by conjugating then the function so obtained by translation with amplitude  $\tfrac{1}{4}$.
Similarly, 
let $h_s$ be the function obtained by setting $s = s_3$, $b =  \tfrac{3}{4}$ 
and by conjugating the function so obtained by the translation $t \mapsto t + \tfrac{1}{4}$.
Finally, set
\begin{equation}
\label{eq:Dinition-G-sub-s}
G_s = G_{\{s_1, s_2, s_3\}} = \gp \left(f_s, g_s, h_s\right).
\end{equation}
The definition of $G_s$ shows
that it is an irreducible subgroup of $\PL_o([0,1])$ 
with non-zero and independent characters $\chi_\ell$ and $\chi_r$. 
By Theorem \ref{thm:Generalization-BNS},
the complement of $\Sigma^1(G_s)$ consists therefore of the two rays $[\chi_\ell]$ and $[\chi_r]$.

Consider now an automorphism $\alpha$ of $G_s$. 
It induces an auto-homeomorphism $\alpha^*$ 
of the sphere $S(G_s)$ 
that maps the subset $\Sigma^1(G_s)^c$ onto itself.
Suppose  $\alpha^*$ is the identity on $\Sigma^1(G_s)^c$.
Since $\chi_\ell$ has rank 1 and is thus transcendental,
the first two paragraphs of section \ref{ssec:Proof-Theorem-Generalization-GoKo10-bis} apply 
and show 
that $\alpha$ fixes the character $\chi_\ell$ and hence also 
the homomorphism $\sigma_\ell \colon G \to \R^\times_{>0}$.
This homomorphism $\sigma_\ell$ will therefore be fixed by all of $\Aut G_s$
whenever the images of $\chi_\ell$ and $\chi_r$ are not isomorphic.
%
%------------
\subsubsection{Additional assumptions}
\label{sssec:More-assumptions-groups}
%----------
Assume therefore that $s_2$ and $s_3$ are \emph{multiplicatively independent}.
Then the free abelian group
\[
\im \chi_r = \ln \gp(\{s_2, s_3\}) = \Z \ln s_2 + \Z \ln s_3.
\] 
has rank 2.

Consider now two triples $s$ and $s'$ where $ s'_2 = s_2$
and where  both pairs $\{s_2, s_3 \}$ and $\{s_2, s'_3\}$ are multiplicatively independent.
Suppose there exists an isomorphism $\beta \colon G_s \iso G_{s'}$.
Then $\beta$ induces a homeomorphism $\beta^* \colon S(G_{s'}) \iso S(G_s)$
that maps $\{[\chi'_\ell], [\chi'_r]\}$ onto $\{[\chi'_\ell], [\chi'_r]\}$.
The ranks of the involved characters imply
that $\beta^*[\chi'_r] = [\chi_r]$;
so there exists a positive real number $u$ with $\chi'_r \circ \beta = u \cdot  \chi_r$.
It follows that $\im \chi'_r = u \cdot \im \chi_r$ or, equivalently, that 
\[
\Z (\ln s'_3) + \Z (\ln s_2) = u \cdot  \left(\Z (\ln s_3) + \Z (\ln s_2) \right).
\]
This equality amounts to say that there exists a matrix 
$T = \left(\begin{smallmatrix} a&b\\c&d \end{smallmatrix}\right) \in \GL(2, \Z)$ 
such that
\[
\begin{pmatrix} \ln s'_3 \\ \ln s_2 \end{pmatrix} 
= 
u \cdot T \cdot \begin{pmatrix} \ln s_3 \\ \ln s_2 \end{pmatrix}
=
u \cdot \begin{pmatrix} a \cdot  \ln s_3 + b \cdot \ln s_2 \\c \cdot \ln s_3  + d \cdot \ln s_2\end{pmatrix}.
\]
It follows that
\[
\frac{\ln s'_3}{\ln s_2} =  \frac{a(\ln s_3/\ln s_2) + b}{c(\ln s_3/\ln s_2)  + d}\;;
\]
alternatively put, 
the numbers $\ln s_3$, $\ln s'_3$ lie in the same orbit of the group 
\begin{equation}
\label{eq:Definition-group-H-sub-s2}
H_{s_2} = \begin{pmatrix}\ln s_2&0\\0&1\end{pmatrix}
\cdot \GL(2,\Z) \cdot 
\begin{pmatrix}\ln s_2&0\\0&1\end{pmatrix}^{-1}
\end{equation}
acting on the extended real line $\R \cup \{\infty\}$ by fractional linear transformations.
%
%------------
\subsubsection{Consequences}
\label{sssec:Consequences}
%----------
%
It is now easy to exhibit a collection of groups $\GG$ 
that enjoy the properties stated at the beginning of section \ref{ssec:Uncountable-class}.
Choose first a number $s_1 > 1$; for instance $s_1 = 2$,
and select $s_2$ so that $\ln s_2$ is rational, for instance $s_2 =  \exp 1$.
The group $H_{s_2}$ is then a subgroup of $\GL(2, \Q)$;
it acts on $\R \cup \{\infty\}$ by fractional linear transformations.
The set $\Q \cup \{\infty\}$ is an orbit;
all other orbits are made up of irrational numbers.
Use the axiom of choice to find a set of representative $\TT$ of the orbits of $H_{s_2}$ 
contained in $\R \smallsetminus \Q$. 
For every $t \in \TT$ 
the numbers $\ln s_2$ and $t$ are then  $\Q$-linearly independent,
and hence $s_2$ and $\exp t$ are multiplicatively independent.
Since $\R \smallsetminus \Q$ has the cardinality of the continuum 
and $H_{s_2}$ is countable,
the set $\TT$ has likewise the cardinality of the continuum.
The collection
\begin{equation}
\label{eq:Definition-GG}
\GG = \{ G_{(s_1, s_2, \exp t)} \mid t \in \TT \}
\end{equation}
enjoys therefore properties (i) through (iv) 
stated at the beginning of section \ref{ssec:Uncountable-class}.
%
%------------
\subsection{Some unexpected isomorphisms}
\label{ssec:Unexpected-isomorphisms}
%----------
%
Let $t_1$, $t_2$ be distinct irrational numbers 
and consider the groups $G_1 = G_{(2, \exp 1, \exp t_1)}$ and $G_2= G_{(2, \exp 1, \exp t_2)}$.
We don't know under which conditions on $t_1$ and $t_2$ 
the groups $G_1$ and $G_2$ are isomorphic.
In the construction of the collection $\GG$, 
carried out in section \ref{ssec:Uncountable-class},
we proceeded therefore in a very cautious manner
and required that distinct elements in the parameter space $\TT$ fail to satisfy a certain condition.
The question now arises whether this approach is overly pessimistic.
The next example indicates that caution may have been appropriate.
We begin with a simple, but surprising, lemma.
\footnote{The third author got word of this result in discussions with Matt Brin and Matt Zaremsky.}
\begin{lem}
\label{lem:Ubiquity-of-F}
Suppose $G$ is a subgroup of $\PL_o([a, d])$ generated by two PL-homeo\-mor\-phisms $f$ and $g$
with the following properties:
\begin{enumerate}[(i)]
\item $\supp f = \,]a, c[$ and $f(t) < t$ for $t \in \supp f$,
\item $\supp g= \,]b, d[$ and $g(t) < t$ for $t \in \supp g$,
\item $a < b < c < d$ and $f(g(c)) \leq b$.
\end{enumerate}
Then $G$ is isomorphic to Thompson's group $F$.
\end{lem}

\begin{proof}
Set $h = f \circ g$ and note that  $h(t) < t$ for every $t \in\; ]a,d[\,$.
Property (iii) then implies that  $h(c) \leq b$ 
and so the supports of $g$ and that of $\act{h}{-2}{ f }= h \circ f \circ  h ^{-1}$ 
are disjoint,
as are the supports of $g$ and that of $\act{h^2}{-2}{f}$.
The first fact implies that $g$ commutes with $\act{h}{-1}{ f}$  
and leads to the chain of equations
\begin{equation}
\label{eq:First-relator}
\act{h \circ h}{-2}{f} 
= 
\act{f}{-2}{\left(\act{g \circ h}{-1}{f} \right)} 
= 
\act{f}{-1}{\left( \act{h}{-1}{f}\right)} 
=
\act{f \circ h}{-1}{ f }.
\end{equation}
The second fact leads to the equations
\begin{equation}
\label{eq:Second-relator}
\act{h \circ h^2}{-2}{f} 
= 
\act{f}{-2}{\left(\act{g \circ h^2}{-1}{f} \right)} 
= 
\act{f}{-1}{\left( \act{h^2}{-1}{f}\right)} 
=
\act{f \circ h^2}{-1}{ f }.
\end{equation}
Thompson's group $F$, on the other hand,
 has the presentation 
 \[
 \langle x, x_1 \mid  \act{x^2}{-1}{x_1}=  \act{x_1 x}{-1}{x_1},
\act{x^3}{-1}{x_1} =  \act{x_1 x^2}{-1}{x_1} \rangle;
 \]
see, \eg\ \cite[Examples D15.11]{BiSt14}. 
 The assignments $x \mapsto h$, $x_1 \mapsto f$ extend 
 therefore to an epimorphism $\rho \colon F \epi G$.
 As the derived group of $F$ is simple (see, \eg\ \cite[Theorem 4.5]{CFP96})
 and as $G$ is non-abelian, $\rho$ must be injective, hence an isomorphism,
 and so the proof is complete.
 \end{proof}
 
 Our next result shows 
 that the assumptions of the previous lemma can be satisfied 
 by PL-homeomorphisms with pre-assigned values for the slopes in the end points.
 \begin{lem}
 \label{lem:Ubiquity-of-F-bis}
 Let $s_f$, $s_g$ be positive reals with $s_f< 1 < s_g$
 and let $a$, $b$, $c$, $d$, be real numbers with $a < b < c < d$.
 Then there exist PL-homeomorphisms $f$ and $g$
 that satisfy properties (i) through (iii) listed in Lemma \ref{lem:Ubiquity-of-F}
 and, in addition,
 \begin{enumerate}
 \item[(iv)] $f'(a) = s_f$ \quad \text{and} \quad $g'(d) = s_g$.
 \end{enumerate}
 \end{lem}
 
 \begin{proof}
 The generators $f$ and $g$ will both be affine interpolations of 5 interpolation points.
 To define them fix numbers $t_1$, $t_2$, $t_3$, $t_4$ so that
 \[
 a < t_1 < b < t_2 \leq t_3 < c < t_4 < d.
 \]
 Next choose $t_0 \in \; ]a, t_1[$ so that $(t_0-a)/(t_1-a) = s_f$.
 Then $a < t_1 < t_3 < c  < d$ and $a < t_0 < b < c <d$ 
 and so the affine interpolation, given by the 5 points
 \[
 (a,a), \quad (t_1, t_0),\quad  (t_3, b), \quad (c,c),\quad  (d,d),
 \]
exists and is an increasing PL-homeomorphism, say $f$, with $f'(a) = s_f$.
Next,  there exists a number $t_5 \in \; ]t_4, d[$ 
so that $(d-t_4)/(d-t_5) = s_g$.
Then $a < b <  c < t_5 < d$ and $a < b < t_2 < t_4 < d$
and so the affine interpolation, 
given by the 5 points
\[
(a,a), \quad (b, b), \quad (c, t_2), \quad (t_5, t_4) ,\quad  (d, d)
\]
exists and is an increasing PL-homeomorphism, say $g$;
the definition of $t_5$ implies, in addition, that $g'(d) = s_g$.
Finally, $f(g(c)) = f(t_2)  \leq f(t_3) = b$.
 \end{proof}
 
 \begin{remarks}
 \label{remarks:Ubiquity-of-F-bis}
 a) In the statement of Lemma  \ref{lem:Ubiquity-of-F-bis}
 the slopes $s_f$ and $s_g$ have been chosen 
 so that $s_f < 1 < s_g$.
 This requirement can be weakened to $s_f \neq 1$ and $s_g \neq1$;
 indeed  the four pairs $\{f,g \}$, $\{f, g^{-1}\}$ and $\{f^{-1},g\}$, $\{f^{-1}, g^{-1}\}$
 generate the same group.
 
 b) The generators $f_s$, $g_s$ and $h_s$ of the groups $G_s$, 
 constructed in section \ref{ssec:Uncountable-class}, 
 are simpler then those used in Lemma \ref{lem:Ubiquity-of-F-bis} 
 in that they are defined by affine interpolations of 3 rather than  of 5 points.
 But a variant of the Lemma \ref{lem:Ubiquity-of-F-bis}
 holds even in this more restricted set-up.
 
 Suppose $s_1 = s_2 = 2$ and $s_3 \geq 2$.
 The function $f_s$ is then given by the formula
\begin{equation}
\label{eq:Function-f}
f_s(t) = 
\begin{cases}
\tfrac{1}{2}t, & \text{ if }0\le t\le \tfrac{1}{2},\\ 
2 \left(t- \tfrac{1}{2}\right)+ \tfrac{1}{4}, &  \text{ if }\tfrac{1}{2} \leq t\leq \tfrac{3}{4},
\end{cases}
\end{equation}
and it is the identity outside of $]0, \tfrac{3}{4}[$,
while $g_s$ is defined by
\begin{equation}
\label{eq:Function-g}
g_s(t) = 
\begin{cases}
\tfrac{1}{2}\left(t- \tfrac{1}{4} \right)+\tfrac{1}{4}, 
& \text{ if } \tfrac{1}{4}\le t\le \tfrac{3}{4},\\ 
2 \left(t- \tfrac{3}{4}\right)+ \tfrac{1}{2}, 
&  \text{ if }\tfrac{3}{4} \leq t\leq 1,
\end{cases}
\end{equation}
and it is the identity outside of $]\tfrac{1}{4}, 1[$. 
The function $h_s$, finally, is defined by
\begin{equation}
\label{eq:Definition-h-sub-s}
h_s(t) = 
\begin{cases}
(1/s_3) \left(t- \tfrac{1}{4} \right)+\tfrac{1}{4}, 
& \text{ if } \tfrac{1}{4}\le t\leq \tfrac{3s_3}{4(s_3+1)} + \tfrac{1}{4},\\ 
s_3 \left(t- \tfrac{3s_3}{4(s_3+1)}- \tfrac{1}{4}\right)+ \tfrac{3}{4(s_3+1)} + \tfrac{1}{4}, 
&  \text{ if }\tfrac{3s_3}{4(s_3+1)} + \tfrac{1}{4} \leq t\leq 1,
\end{cases}
\end{equation}
and is the identity outside of \;$]\tfrac{1}{4}, 1[$. 
The function $g_s$ is not differentiable in $\tfrac{1}{4}$, $\tfrac{3}{4}$ and 1,
while the function $h_s$ has singularities in 
$\tfrac{1}{4}$, $t_* = \tfrac{3s_3}{4(s_3+1)} + \tfrac{1}{4}$ and in 1.
Since $s_3 \geq s_2 = 2 $, the inequality $t_* \geq \tfrac{3}{4}$ holds,
as one verifies easily.
The calculation
\[
f_s \left(h_s(\tfrac{3}{4}) \right)
= 
f_s\left( (1/s_3) \left(\tfrac{3}{4}- \tfrac{1}{4} \right)+\tfrac{1}{4}\right)
=
f_s\left( \tfrac{1}{2s_3} + \tfrac{1}{4} \right)
\leq 
f_s\left( \tfrac{1}{4} + \tfrac{1}{4} \right) \leq \tfrac{1}{4}.
\]
then shows that the functions $f_s$ and $h_s$ 
fulfill the assumptions imposed on the functions $f$ and $g$ 
in Lemma \ref{lem:Ubiquity-of-F}. 
It follows that the groups $\gp(f_s, h_s)$  
are isomorphic to each other for every $s_3 \geq s_2 = 2$.
\end{remarks}

%------------
\subsection{A criterion}
\label{ssec:Precluding-isomorphisms}
%----------
%
The groups $G_s$ studied in section \ref{ssec:Uncountable-class} are generated by 3 elements;
in addition the image of $\chi_\ell$ is infinite cyclic 
and that of $\chi_r$ is free abelian of rank 2.
Any isomorphism $\beta \colon G_s \iso G_{s'}$ between two such groups must therefore
induce an homeomorphism $\beta^* \colon S(G_{s'}) \iso S(G_s)$ 
with $\beta^*([\chi'_r]) = [\chi_r]$.
This consequence amounts to say 
that there exists a positive real number $u$ so that $\chi'_r \circ \beta = u \cdot \chi_r$
and this new condition implies the equality
\begin{equation}
\label{eq:Consequence-iso}
\im \chi'_r = u \cdot \im \chi_r.
\end{equation}
In section \ref{ssec:Uncountable-class}  
we did not study this condition in general;
we dealt only with the special case 
where 
\[
\im \chi_r = \Z (\ln s_3) + \Z (\ln s_2) 
\quad\text{and}\quad 
\im \chi'_r = \Z (\ln s'_3) + \Z (\ln s_2)
\]
and exploited then the fact
that, in this particular case,
condition \eqref{eq:Consequence-iso} involves basically only the two numbers $\ln s_3$ and $\ln s'_3$.
In this final section we shall investigate another special case.
It is reminiscent of a situation considered in \ref{ssec:Group of units}.

Let $B_1$ and $B_2$ be finitely generated subgroups of  $\R_{\add}$
and suppose there exists a positive real number $u$ with $B_2 = u \cdot B_1$.
If $B_2$ coincides with $B_1$, then $u$ is a unit of $B_1$ 
and the results of section \ref{ssec:Group of units} apply.
They show, in particular, that $u = 1$ 
whenever $B$ is the image under $\ln$ of a subgroup $P$ of  $\R^\times_{>0}$
that is generated by finitely many algebraic numbers.
The proof of this consequence relies on Theorem \ref{thm:Gelfond-Schneider},
the Gelfond-Schneider Theorem. 
Below we give an analogue of this criterion, 
but dealing with the equation $B_2 = u \cdot B_1$. 
In the proof, 
both the Gelfond-Schneider Theorem and the Siegel-Lang Theorem will be used.
\begin{lem}
\label{lem:Non-existence-u}
Let $P_1$ and $P_2$ be subgroups $\Q^\times_{>0}$
and set $B_1 = \ln P_1$ and $B_2 = \ln P_2$.
Suppose there exists a prime number $\pi$ 
that occurs with non-zero power in the factorization of an  element in $P_1$,
but not in that of an element of $P_2$.
If the rank of $B_1$ is at least 3,
then $B_2$ is distinct from $u \cdot B_1$ for every positive real number $u$.
\end{lem}

\begin{proof}
Let $p_1 \in P_1$ be an element with a prime factorization that involves the prime $\pi$, 
and let $u$ be positive real number.
Assume first that $u \in \Q$.
Then $\pi$ occurs in the prime factorization of $p_1^u$, 
so $q_1^u \notin P_2$,
and thus $u \cdot B_1 =u  \ln P_1 \neq \ln P_2 = B_2$.
Suppose now that $u$ is irrational and that $u \cdot \ln p_1 \in B_2$.
There exists then a rational number $p_2 \in P_2$ with $u = \ln p_2/\ln p_1$ 
and so $u$ is transcendental by the Gelfond-Schneider Theorem.
Choose, finally, three  $\Q$-linearly independent elements $z_1$, $z_2$ and $z_3$ in $B_1$
(this is possible as the rank of $B_1$ is at least 3)
and consider the six numbers
\[
\exp(1 \cdot z_j) \text{ with } j = 1,2, 3 
\quad \text{and} \quad
\exp(u \cdot z_j) \text{ with } j = 1,2, 3.
\]
The first three of them are in $P_1$, and hence rational.
As the subsets $\{1, u\}$ and $\{z_1, z_2, z_3 \}$ are both linearly independent over $\Q$,
Theorem \ref{thm:Siegel-Lang} implies therefore 
that at least one the remaining three numbers, say $\exp(u \cdot z_{j_*})$, is transcendental.
This number is therefore outside of $P_2$ 
and so $u \cdot z_{j_*} \in u B_1 \smallsetminus B_2$.
\end{proof}

We terminate with an application of the preceding lemma.
\begin{example}
\label{example:Non-isomorphic-groups}
Given a non-empty set of prime numbers $\PP$,
let $G_\PP$ be a subgroup of $PL_o([0,1])$ generated by a set  $\{f_p, g_p \mid p \in \PP \}$
of elements that satisfy the following two conditions
\begin{enumerate}[(i)]
\item $\sigma_\ell (f_p) = p$, $\sigma_\ell (g_p) = 1$ 
and  $\sigma_r (f_p) = 1$, $\sigma_\ell (g_p) = 1$  for every $p \in \PP$;
\item the union of the supports of the generators $f_p$ and $g_p$ is $]0,1[$.
\end{enumerate}
The group $G_\PP$ admits then an epimorphism $\psi \colon G_\PP  \epi \gp(\PP)$ 
that is fixed by every automorphism of $G_\PP$
(use Corollary \ref{crl:Generalization-GoKo10-bis}).
Moreover,
if $\PP_1$ and $\PP_2$ are distinct sets of primes of cardinality at least 3, 
the groups $G_{\PP_1}$ and $G_{\PP_2}$ are not isomorphic in view of Lemma
\ref{lem:Non-existence-u} and the considerations at the beginning of section \
\ref{ssec:Precluding-isomorphisms}.
\end{example}

\def\cprime{$'$} \def\cprime{$'$}

%\newpage
%
%\bibliographystyle{amsalpha}%

%\bibliography{References15}

\begin{thebibliography}{Bro87b}

\bibitem[BB05]{BeBr05}
James~M. Belk and Kenneth~S. Brown.
\newblock Forest diagrams for elements of {T}hompson's group {$F$}.
\newblock {\em Internat. J. Algebra Comput.}, 15(5-6):815--850, 2005.

\bibitem[BFG08]{BFG08}
Collin Bleak, Alexander Fel{\cprime}shtyn, and Daciberg~L. Gon{\c{c}}alves.
\newblock Twisted conjugacy classes in {R}. {T}hompson's group {$F$}.
\newblock {\em Pacific J. Math.}, 238(1):1--6, 2008.



\bibitem[BNS87]{BNS}
Robert Bieri, Walter~D. Neumann, and Ralph Strebel.
\newblock A geometric invariant of discrete groups.
\newblock {\em Invent. Math.}, 90(3):451--477, 1987.

\bibitem[BS85]{BiSt85}
Robert Bieri and Ralph Strebel.
\newblock On groups of {P}{L}-homeomorphisms of the real line.
\newblock preprint, 1985.

\bibitem[BS14]{BiSt14}
Robert Bieri and Ralph Strebel.
\newblock On groups of {P}{L}-homeomorphisms of the real line.
\newblock preprint, \texttt{arXiv:1411.2868v1}, 2014.

\bibitem[Br96]{Bri96}
Matthew~G. Brin.
\newblock The chameleon groups of {R}ichard {J}. {T}hompson: automorphisms and
  dynamics.
\newblock {\em Inst. Hautes \'Etudes Sci. Publ. Math.}, (84):5--33 (1997),
  1996.

\bibitem[BrG98]{BrGu98}
Matthew~G. Brin and Fernando Guzm{\'a}n.
\newblock Automorphisms of generalized {T}hompson groups.
\newblock {\em J. Algebra}, 203(1):285--348, 1998.

\bibitem[BrS85b]{BrSq85}
Matthew~G. Brin and Craig~C. Squier.
\newblock Groups of piecewise linear homeomorphisms of the real line.
\newblock {\em Invent. Math.}, 79(3):485--498, 1985.



\bibitem[Bro87a]{Bro87a}
Kenneth~S. Brown.
\newblock Finiteness properties of groups.
\newblock In {\em Proceedings of the {N}orthwestern conference on cohomology of
  groups ({E}vanston, {I}ll., 1985)}, volume~44, pages 45--75, 1987.

\bibitem[Bro87b]{Bro87b}
Kenneth~S. Brown.
\newblock Trees, valuations, and the {B}ieri-{N}eumann-{S}trebel invariant.
\newblock {\em Invent. Math.}, 90(3):479--504, 1987.


\bibitem[CFP96]{CFP96}
J.~W. Cannon, W.~J. Floyd, and W.~R. Parry.
\newblock Introductory notes on {R}ichard {T}hompson's groups.
\newblock {\em Enseign. Math. (2)}, 42(3-4):215--256, 1996.

\bibitem[Cle95]{Cle95}
Sean Cleary.
\newblock Groups of piecewise-linear homeomorphisms with irrational slopes.
\newblock {\em Rocky Mountain J. Math.}, 25(3):935--955, 1995.

\bibitem[Cle00]{Cle00}
Sean Cleary.
\newblock Regular subdivision in {$\mathbf{Z}[\frac{1+\sqrt 5}{2}]$}.
\newblock {\em Illinois J. Math.}, 44(3):453--464, 2000.

\bibitem[FT]{FeTr12}
Alexander Fel{\cprime}shtin and Evgenij Troitsky.
\newblock Twisted conjugacy classes in residually finite groups.
\newblock preprint, \texttt{arXiv:} \texttt{1204.3175v2[math.GR], 30 April}, 2012.

\bibitem[GK10]{GoKo10}
D.~Gon{\c{c}}alves and D.~H. Kochloukova.
\newblock Sigma theory and twisted conjugacy classes.
\newblock {\em Pacific J. Math.}, 247(2):335--352, 2010.

\bibitem[GS14]{GoSa14}
Daciberg~L. Gonçalves and Parameswaran Sankaran.
\newblock Sigma theory and twisted conjugacy--{I}{I}: Houghton groups and pure
  symmetric automorphism groups.
\newblock to appear in Pacific J. Math., \texttt{arXiv:}
  \texttt{1412.80468v1[math.GR], 27 Dec}, 2014.

\bibitem[GW09]{GoWo09}
D.~Gon{\c{c}}alves and P.~Wong.
\newblock Twisted conjugacy classes in nilpotent groups.
\newblock {\em J. Reine Angew. Math.}, 633:11--27, 2009.

\bibitem[Lan66]{Lan66}
Serge Lang.
\newblock {\em Introduction to transcendental numbers}.
\newblock Addison-Wesley Publishing Co., Reading, Mass.-London-Don Mills, Ont.,
  1966.

\bibitem[Lan71]{Lan71}
Serge Lang.
\newblock Transcendental numbers and diophantine approximations.
\newblock {\em Bull. Amer. Math. Soc.}, 77:635--677, 1971.

\bibitem[LL00]{LeLu00}
Gilbert Levitt and Martin Lustig.
\newblock Most automorphisms of a hyperbolic group have very simple dynamics.
\newblock {\em Ann. Sci. \'Ecole Norm. Sup. (4)}, 33(4):507--517, 2000.

\bibitem[LS01]{LS77}
Roger~C. Lyndon and Paul~E. Schupp.
\newblock {\em Combinatorial group theory}.
\newblock Classics in Mathematics. Springer-Verlag, Berlin, 2001.
\newblock Reprint of the 1977 edition.

\bibitem[MS14a]{MuSa14a}
T.~Mubeena and P.~Sankaran.
\newblock Twisted conjugacy classes in abelian extensions of certain linear
  groups.
\newblock {\em Canad. Math. Bull.}, 57(1):132--140, 2014.

\bibitem[MS14b]{MuSa14b}
T.~Mubeena and P.~Sankaran.
\newblock Twisted conjugacy classes in lattices in semisimple {L}ie groups.
\newblock {\em Transform. Groups}, 19(1):159--169, 2014.

\bibitem[Niv56]{Niv56}
Ivan Niven.
\newblock {\em Irrational numbers}.
\newblock The Carus Mathematical Monographs, No. 11. The Mathematical
  Association of America. Distributed by John Wiley and Sons, Inc., New York,
  N.Y., 1956.

\bibitem[Ste92]{Ste92}
Melanie Stein.
\newblock Groups of piecewise linear homeomorphisms.
\newblock {\em Trans. Amer. Math. Soc.}, 332(2):477--514, 1992.

\bibitem[Str13]{Str13}
Ralph Strebel.
\newblock Notes on the {S}igma-invariants, {V}ersion 2.
\newblock preprint, \texttt{arXiv:} \texttt{1204.0214v2, 1 Mar}, 2013.

\bibitem[Str15]{Str15}
Ralph Strebel.
\newblock Sigma 1 of {P}{L}-homeomorphism groups.
\newblock preprint, \texttt{arXiv:} \texttt{1507.00135, 1 Jul}, 2015.

\end{thebibliography}
%
%-----------------------------
%
\end{document}